\documentclass[fontsize=12pt,a4paper,headings=normal,
twoside=false,leqno,parskip=half-,abstract=true]{scrartcl}
\usepackage[english]{babel}
\usepackage[utf8]{inputenc}
\usepackage{changepage}
\usepackage[version=4]{mhchem}
\usepackage{graphicx}
\usepackage{mdframed}
\usepackage{float}
\usepackage{hyperref}
\usepackage{enumitem}
\hypersetup{
 pdftitle={},
 pdfauthor={},
 colorlinks=true,
 linkcolor=black,
 citecolor=black,
 filecolor=black,
 urlcolor=black}
 \usepackage{soul}
 \usepackage{comment}

\usepackage[format=plain,labelfont=bf,font=small]{caption}
\usepackage{subfigure}%replaced by subcaption
\usepackage{pdflscape}
\usepackage{xcolor}
\usepackage[arrow, matrix, curve]{xy}

\usepackage{tikz-cd}
\usepackage{caption}
\usepackage{amsmath,amsthm}
\usepackage{amssymb} 
\usepackage[normalem]{ulem}
\usepackage{blkarray}
\newtheorem{theorem}{Theorem}[section]
\newtheorem{lemma}[theorem]{Lemma}
\newtheorem{cor}[theorem]{Corollary}
\newtheorem{prop}[theorem]{Proposition}
\newtheorem{remark}[theorem]{Remark}

\theoremstyle{definition}
\newtheorem{definition}[theorem]{Definition}

\numberwithin{equation}{section}

\DeclareMathAlphabet{\mathpzc}{OT1}{pzc}{m}{it}

\definecolor{refkey}{rgb}{1,0,0}
\definecolor{labelkey}{rgb}{1,0,0}

\title{Mathematical modeling and analysis of the Notch--Delta pathway}
\author{Angela Stevens, Nicola Vassena}

\date{\today}

\begin{document}

\vspace{-0.2cm}

\maketitle

\begin{center}
{\large \textit{Dedicated to Vera Natalia (*2026) and to Heinrich (*1928)}}
\end{center}
\begin{abstract}
    In this paper mathematical models for the evolutionary conserved Notch--Delta pathway are developed and analyzed in order to better understand how two neighboring biological cells can become different. We pursue a structure-based stoichiometric type of approach, such that no specific reaction kinetics have to be defined. Only their dependencies on the relevant species participating in the model network are taken into account. Reaction networks and their related systems of ODEs are presented and analyzed with respect to their capacity for symmetry-induced bifurcations. The possibility to obtain a singular Jacobian is analyzed symbolically. This approach is valid for parameter-rich kinetics, where the parametrization of the steady-state fluxes and of the first derivatives of the reaction rates evaluated at the steady state are independent. In this context, also with the help of abstract minimal models, we could mathematically identify some of the Notch pathway's features being more relevant than others.\\
    
    {\footnotesize \noindent \textbf{Keywords:} Notch pathway, reaction networks, symmetry-induced bifurcation, symbolic analysis,  systems of ODEs}
\end{abstract}

\vspace{-0.28cm}

{\small 
\tableofcontents}

\section{Introduction}

The so-called Notch or Notch--Delta pathway is one of the 
evolutionary highly conserved and central
cellular signaling pathways among
 metazoan organisms, which include 
 worms, fish, flies, amphibians, reptiles, mammals, birds,  
 and humans. The Notch pathway controls several fundamental differentiation
processes and regulates binary cell fate decisions. Thus its proper
functioning is of utmost importance. 
Nevertheless, the precise mechanisms on how 
neighboring and seemingly similar 
cells can finally become 
different w.r.t. the expression of Notch receptors (and/or Delta/Serrate/LAG-2 ligands) are not fully understood. 
The outcome of such a dynamical process may not be robust in the selection of 
specific cells among a few, but it has at least to be robust in some sense, 
e.g. w.r.t. the percentage and spatial patterning of differently expressing 
cells in a given tissue. 

In order to contribute to a better understanding of such dynamics 
from a theoretical 
point of view, we set up and analyze mathematical models  for two cells 
only. 
Our models are first formulated as reaction 
networks and take into account  
some of the known and conjectured structural features of the Notch pathway. 
We then identify the most relevant 
of those features 
which can give rise to 
symmetry breaking in the systems of ordinary differential
equations related to these networks.  
Based on biological descriptions, our models focus only on those processes that 
`mathematically' can create qualitative differences between two initially identical cells. 
%and is not intended as a quantitative model of expression levels; rather, 
We target symmetry breaking, without imposing specific relative 
orderings between the two cells w.r.t. quantitative expression
levels of Notch/Delta. 
We do not explain layers, stripes etc. in tissues here, which can also be seen
in Notch--Delta patterning.

Spontaneous symmetry breaking is well understood mathematically, with applications in biology too. It may arise, for instance, through the loss of stability of a homogeneous steady state via a zero-eigenvalue bifurcation, which generically leads to a pitchfork bifurcation and the emergence of two inhomogeneous, i.e. asymmetric, steady states \cite{GuHoBook}. This perspective has been developed in detail within \emph{equivariant dynamical systems theory} \cite{golubitsky2003symmetrybook}. It provides a framework to analyze symmetry-induced bifurcations, including networks. Here we draw on this perspective but focus on reaction networks where symmetry is intrinsic since the interacting cells are initially identical. A systematic approach that combines symmetry with stoichiometric considerations, as customary in reaction network theory, appears to be largely underexplored. To the best of our knowledge, symmetry has only occasionally been used to establish bistability in specific reaction-network models, see e.g. \cite{HellRendall15proof}.

Many mathematical models of the Notch--Delta pathway trace back to \cite{Collier96},
where the pathway's dynamics are described by explicitly given nonlinear regulatory functions.
Notch in cell 1 is activated by Delta of cell $2$, while this Notch activity
inhibits the production of Delta in the same cell $1$, and vice versa. The mechanism leading to differentiation,
i.e. different expression levels of Notch/Delta in the two cells, is thus `directly' encoded in the
specific choice of the nonlinear interaction functions. In our present paper, the focus is different.
We are, instead, searching within the network structure itself for the possibility of differentiation,
and do not fix nonlinearities.

The model in \cite{Collier96} has subsequently been extended into several directions,
including the coupling with other signaling pathways and also with diffusion \cite{BaniYaghoub2010}, long-range signaling  \cite{BerkemeierPage23}, or stochastic cell rearrangements \cite{Oguma2023}.
A different body of work on the Notch pathway relates to \cite{Sprinzak2010}. This model was mathematically rigorously analyzed in \cite{Moore19}.
It adopts a structural description of the system, closer in spirit to our perspective here,
and in particular also focuses on the additional possible cis-interaction between Notch and Delta,
i.e., their interaction within the same cell, %{which blocks other reactions ofthe two entities for a while,}
see \ref{eq:modelii} below, and in particular Remark \ref{rmk:moore}. 

Finally, the Notch--Delta pathway has also been analyzed for oscillatory gene-expression dynamics in the context of the segmentation clock in vertebrate embryos, often in interaction with other pathways \cite{RodriguezGonzalez2007,Goldbeter2008}.

%\cite{RodriguezGonzalez2007,Goldbeter2008,Sun2020,Zhang2024}. 

\subsection{Biochemically informed mathematical models}
Our mathematical models are based on descriptions %and discussions 
in \cite{HenSchwei19, Andersson11, yamamoto2010}, see Fig.~\ref{fig:figure1} for an overview. In order to focus on the internal structure and interactions of the pathway, we deliberately omit ubiquitous production and degradation reactions. These are not specific to the model and bear the risk of obscuring, rather than clarifying, the analysis. Thus their omission allows us to better understand the essential features of this highly conserved pathway.

\begin{figure}[H]
    \centering   \includegraphics[width=0.8\textwidth]{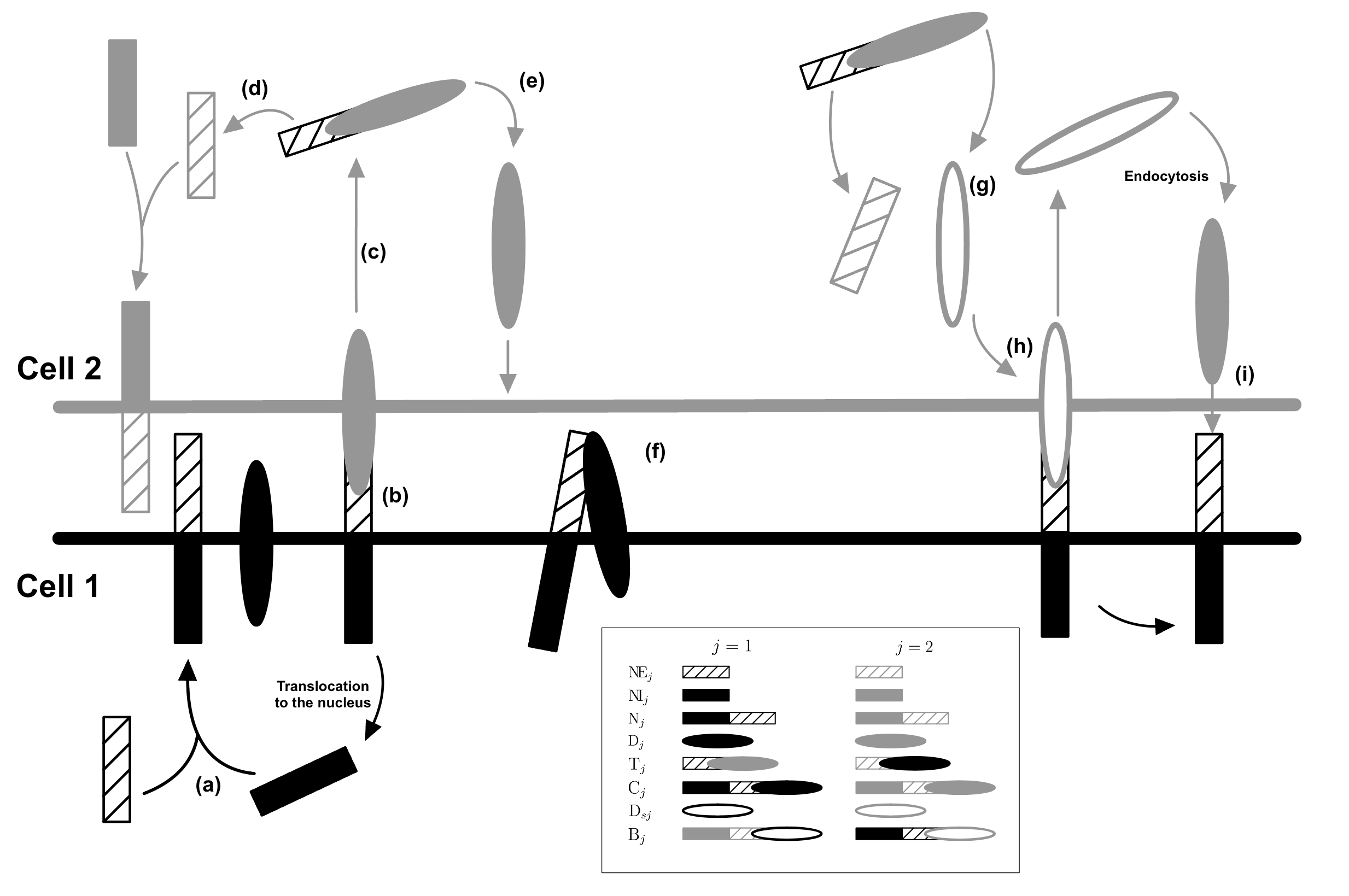}
\caption{Schematic representation of the biological mechanisms taken
into account in 
our mathematical models, on the left \ref{eq:modeli} and the cis-reactions \ref{eq:modelii}, on the right the ligand-activation variant \ref{eq:modeliii}. Not all interactions are {depicted 
symmetrically}, for visual clarity, 
but they are to be understood as occurring 
symmetrically in both cells. Letters refer to the biochemical explanation in the text. For the complete hypergraph 
representation of the reaction models, see Fig.~\ref{fig:figure2}.}
\label{fig:figure1}
\end{figure}

For the central part of our mathematical model, we consider the Notch receptors $\ce{\ce{N_1}}$
of e.g. cell 1, which arise from an extracellular domain 
$\ce{N\!E_1}$ and an intracellular  domain $\ce{\ce{N\!I_1}}$ \textbf{(a)}. The Notch 
receptors are activated by cell-surface ligands from a neighboring cell \textbf{(b)}, 
i.e. the receptors release their intracellular domain upon transactivation of their extracellular domain by Delta/Serrate/LAG-2 ligands $\ce{D_2}$ of cell 
2.
The released intracellular domain $\ce{N\!I_1}$ then, after some other reactions (see Remark \ref{rmk:nucleus}) translocates to the nucleus of 
cell 1. The otherwise newly formed \emph{transaction complex}  $\ce{T_2}$ is endocytozed by cell 2 \textbf{(c)}, and it can be 
recycled into an extracellular Notch
domain $\ce{N\!E_2}$ \textbf{(d)} and a ligand $\ce{D_2}$ of cell 2
\textbf{(e)}. These central processes take place in both cells and are, 
in essence, represented by the following six reactions, which are pairwise symmetric. 
This symmetry is also {reflected by}  
our choice of numbering the reactions.
\begin{equation}\label{eq:modeli}\tag{\textbf{BI}}
\begin{aligned}
\ce{{N\!I_1}} + \ce{N\!E_1} \quad&\overset{11}{\ce{->}} \quad \ce{N_1} \quad &\quad \quad \ce{N\!I_2} + \ce{N\!E_2} \quad&\overset{21}{\ce{->}} \quad \ce{N_2}\\
    \ce{N_1}+\ce{D_2} \quad &\overset{12}{\ce{->}} \quad \ce{N\!I_1} + \ce{T_2} \quad &\quad \quad  \ce{N_2}+\ce{D_1} \quad &\overset{22}{\ce{->}} \quad \ce{N\!I_2} + \ce{T_1}\\
\ce{T_1} \quad &\overset{13}{\ce{->}} \quad \ce{N\!E_1} + \ce{D_1}\quad &\quad \quad \ce{T_2} \quad &\overset{23}{\ce{->}} \quad \ce{N\!E_2} + \ce{D_2}.
\end{aligned}
\end{equation}

\begin{remark}\label{rmk:nucleus}
The translocation of $\ce{{N\!I}}$  to the nucleus 
can be explicitly modeled by introducing $\ce{N\!I^{Nucleus}_1}$ and 
$\ce{N\!I^{Nucleus}_2}$, and splitting the reactions 12 and 22 as
\begin{equation}\label{eq:modeli'}
\tag{\textbf{BI'}}
\begin{aligned}
\ce{N_1}+\ce{D_2} \quad &\overset{12'}{\ce{->}} \quad \ce{N\!I^{Nucleus}_1} + \ce{T_2} \quad &\quad \quad  \ce{N_2}+\ce{D_1} \quad &\overset{22'}{\ce{->}} \quad \ce{N\!I^{Nucleus}_2} + \ce{T_1}\\
\ce{N\!I^{Nucleus}_1} \quad &\overset{12''}{\ce{->}} \quad 
\ce{N\!I_1} \quad & \quad \quad \ce{N\!I^{Nucleus}_2} \quad &\overset{22''}{\ce{->}} \quad \ce{N\!I_2}.
\end{aligned}
\end{equation}
These additions do not affect the behavior of \emph{\ref{eq:modeli}} w.r.t. the capacity for differentiation (see Remark~\ref{rmk:nucleus2} and \hyperref[pt:translocation]
{p.~\pageref*{pt:translocation}}). They  only increase the system's dimension
and are therefore dismissed.  
\end{remark}
A cell's Notch receptors $\ce{N}_j$ can also interact with the very same 
cell's ligands $\ce{D}_j, \; j=1,2$, (\emph{cis reaction}), \textbf{(f)}. It is conjectured that this reaction does not lead to the type of Notch activation described above. So this reaction 
acts like a kind of
inhibition.
Despite its likely broad relevance, however, this interesting regulatory mechanisms
is most often overlooked \cite{HenSchwei19,Andersson11}. 
Therefore we extend our central model \ref{eq:modeli} %to a second one 
{by introducing} 
a product $\ce{C_j}$ for this cis reaction.
This product is resolved after a while again.
\begin{equation}\label{eq:modelii}\tag{\textbf{BII}}
   \ce{N_1}+\ce{D_1} \quad \underset{15}{\overset{14}{\ce{<=>}}}\quad \ce{C_1} \quad \quad \quad  \ce{N_2}+\ce{D_2} \quad \underset{25}{\overset{24}{\ce{<=>}}}\quad \ce{C_2}
\end{equation}

\begin{remark}\label{rmk:moore}
The model in \emph{\cite{Sprinzak2010}} considers 
the Notch receptor $\ce{N}$ with its intracellular domain $\ce{N\!I}$, and the Delta ligand $\ce{D}$, within one cell.
Further, constant production of $\ce{N}$ and $\ce{D}$ and linear degradation of $\ce{N}$, $\ce{N\!I}$,
and $\ce{D}$ are assumed. A fixed concentration $\ce{D}_{\operatorname{ext}}$ accounts for the Delta concentration in neighboring cells. In the mathematical follow up paper \emph{\cite{Moore19}} two interacting cells were considered in an analogous way
\begin{equation}\tag{\textbf{Cis-R}}\label{eq:moore}
\begin{aligned}
\ce{N}_1+\ce{D}_1\quad&\underset{\ce{S}1}{\ce{->}}\quad \dots  & \quad \ce{N}_2+\ce{D}_2\quad&\underset{}{\ce{->}}\quad \dots \\
\ce{N}_1+\ce{D}_{2}\quad \underset{\ce{S}2}{\ce{->}} \quad&\ce{N\!I_1} \quad \ce{->} \quad \dots & \quad \ce{N}_2+\ce{D}_{1}\quad \underset{}{\ce{->}} \quad&\ce{N\!I_2} \quad \ce{->} \quad \dots\\
\dots \quad \ce{->} \quad&\ce{N_1} \quad \ce{->} \quad \dots & \quad \dots \quad \ce{->} \quad&\ce{N_2} \quad \ce{->} \quad \dots\\
\dots \quad \ce{->} \quad&\ce{D_1} \quad \ce{->} \quad \dots & \quad \dots \quad \ce{->} \quad&\ce{D_2} \quad \ce{->} \quad \dots\\
\end{aligned}
\end{equation}

%\begin{equation}\tag{\textbf{Cis-S}}\label{eq:sprinzak}
%\begin{aligned}
%\ce{N}+\ce{D}\quad&\underset{\ce{S}1}{\ce{->}}\quad \dots\\
%\ce{N}+\ce{D}_{\operatorname{ext}}\quad \underset{\ce{S}2}{\ce{->}} \quad&\ce{N\!I} \quad \ce{->} \quad \dots\\
%\dots \quad \ce{->} \quad&\ce{N} \quad \ce{->} \quad \dots\\
%\dots \quad \ce{->} \quad&\ce{D} \quad \ce{->} \quad \dots\\
%\end{aligned},
%\end{equation}
where the dots encode production by and degradation to entities external to the system. In \emph{\cite{Sprinzak2010}} only the reactions on the left-hand side of \emph{\ref{eq:moore}} were considered with $\ce{D}_2=\ce{D}_{\operatorname{ext}}$.
It was proven in \emph{\cite{Moore19}} that the mass-action ODE system associated to \emph{\ref{eq:moore}} possesses a unique asymptotically stable steady-state if the respective parameters in the two cells are considered to be identical. In both cases, also in \emph{\cite{Sprinzak2010}}, differentiation in the sense of our present paper is excluded, unless parameters vary in the two cells from the beginning, which we want to exclude here. Although 
our model involves more entities and considers dynamics where no production or degradation terms are required,
we remark the consistency of reactions S1 and S2 in \emph{\ref{eq:moore}} with reactions 14/24 in \emph{\ref{eq:modelii}} and 12/22 in \emph{\ref{eq:modeli}}. 
\end{remark}
Further, in a third model we extend our central model \ref{eq:modeli} by
the so-called \emph{ligand activation hypothesis} \cite{yamamoto2010}. 
It is assumed that newly synthesized Delta, $\ce{D}_{sj}$, initially appears at the cell's surface in a silenced state, see \textbf{(g)} in Fig.~\ref{fig:figure1}, incapable of signaling. Thus an 
interaction between $\ce{N}_i$ and $\ce{D}_{sj}, \; i\neq j$ 
results in a \emph{blocked complex} $\ce{B_j}$, which eventually resolves \textbf{(h)}. To become activated, the silenced Delta must first be endocytozed. Only after being recycled back to the cell's surface it becomes activated ($\ce{D_{j}}$) and capable of signaling \textbf{(i)}.
The molecular mechanism of this `activation' remains
to be identified. In order to account for this process, we modify \ref{eq:modeli} to obtain
\begin{equation}\label{eq:modeliii}\tag{\textbf{BIII}}
\begin{aligned}
\ce{N\!I_1} + \ce{N\!E_1} \quad&\overset{11}{{\ce{->}}} \quad \ce{N_1} \quad &\quad \quad \ce{N\!I_2} + \ce{N\!E_2} \quad&\overset{21}{\ce{->}} \quad \ce{N_2}\\
\ce{N_1} + \ce{D_{2}} \quad &\overset{12}{{\ce{->}}} \quad \ce{N\!I_1} + \ce{T_2} \quad&\quad\quad  \ce{N_2} + \ce{D_{1}} \quad &\overset{22}{{\ce{->}}} \quad \ce{N\!I_2} + \ce{T_1}\\
\ce{N_1} + \ce{D_{s2}} \quad &\underset{17}{\overset{16}{\ce{<=>}}} \quad \ce{B_2} \quad&\quad\quad \ce{N_2} + \ce{D_{s1}} \quad &\underset{27}{\overset{26}{\ce{<=>}}}\quad \ce{B_1}\\
\ce{D_{s1}} \quad &\overset{18}{{\ce{->}}} \quad \ce{D_{1}}\quad&\quad\quad\ce{D_{s2}} \quad &\overset{28}{{\ce{->}}} \quad \ce{D_{2}}\\
\ce{T_1} \quad &\overset{19}{\ce{->}} \quad \ce{N\!E_1} + \ce{D_{s1}}\quad &\quad \quad \ce{T_2} \quad &\overset{29}{\ce{->}} \quad \ce{N\!E_2} + \ce{D_{s2}} 
\end{aligned}
\end{equation}

\begin{remark}
The blocked complex $\ce{B}_j$ is intercellular and therefore does not clearly belong to either cell 1 or 2. The choice of labeling is consequently entirely arbitrary, and nothing would differ by exchanging $\ce{B}_1$ with $\ce{B}_2$, and vice versa.
\end{remark}
\begin{remark}
The silenced state $\ce{D}_{sj}$ of Delta acts crucially different
w.r.t. the capacity for differentiation in comparison to 
the translocation of the intracellular Notch domain
$\ce{N\!I_j}$ to the nucleus as
 $\ce{N\!I_j}^{Nucleus}$ in \emph{\ref{eq:modeli'}}. This is due to the presence 
of the `blocked' reactions $j6$ for $j=1,2$.
\end{remark}

So we consider the Notch--Delta dynamics in three variants:
\ref{eq:modeli}, \ref{eq:modeli}+\ref{eq:modelii}, and \ref{eq:modeliii}, 
as networks of biochemical reactions and deal with systems of ordinary differential equations (ODEs), see Sec.~\ref{sec:mainresult}. 
Since we model the interaction between two initially 
identical cells the
network graph exhibits a \emph{structural symmetry}, specifically a 
$\mathbb{Z}_2$-symmetry, induced by the involutive graph automorphism. 
Our choice of labels for the reactions makes this symmetry explicit.
Reactions $11-19$ refer to cell 1 and reactions $21-29$ to cell 2. 
Due to the supposed identity of the two cells prior to differentiation, we also assume \emph{kinetic symmetry}, 
i.e., the nonlinearities of reactions mapped into each other by the symmetry 
automorphism are identical. This assumption renders the system 
$\mathbb{Z}_2$-equivariant \cite{golubitsky2003symmetrybook}: in other words, 
the dynamics commute with the swap of indices between the two cells.

We address the symmetry-breaking differentiation process, where two initially identical cells 
develop different levels of expression of Notch (and/or Delta/Serrate/LAG-2 ligands) 
by a local bifurcation of the ODE systems {steady-states}. 
%[in the following: usage of homogenous vs. identical should be perhaps clarified or avoided. `symmetric' is also a viable alternative.] 
The number of steady states remains unchanged as long as the Jacobian at the homogeneous steady state is invertible and the implicit function theorem holds. The emergence of multiple asymmetric steady states thus requires the Jacobian to be singular, with determinant zero. This bifurcation argument brings up 
our central structural question:

{\emph{\centerline{
Which of our network topologies do allow for a singular Jacobian}}}
{\emph{\centerline{at a homogeneous steady-state with kinetic symmetry?}}}
\vskip0.1cm

{
For a positive answer to this question, we say that the network \emph{has the capacity for differentiation}. We stress that this corresponds only to a necessary linear condition. The precise nonlinear unfolding depends of course on the specific choice of the nonlinear reaction rates. In this sense,} one main problem in the analysis of most biochemical reaction networks and beyond
is indeed the lack of quantitative knowledge of the reaction rates. 
Trying to quantify them in our given setting 
would not make sense.  
%Second, even if you assume you would know the form, there is
%a computational issue, since you end up with large systems
%which you can not really handle. 
Therefore we pursue a structure-based, i.e., stoichiometric type of 
analysis to answer the above question.  In order to do so, we consider the network being endowed with a broad class of kinetics, where only the dependencies of the reaction rates on the concentrations are specified. These dependencies are encoded by the stoichiometry itself. The possibility of 
a singular Jacobian is analyzed symbolically, i.e., no steady-state computation has to be done. 
Therefore the generality of the class of kinetics being considered is 
essential. 

Our approach is valid whenever the kinetics under consideration 
are \emph{parameter-rich} \cite{VasStad24}, that is, allowing for an 
independent parametrization of both the steady-state fluxes and 
the first derivatives of the reaction rates evaluated at the steady state. 
This concept works for a wide range of kinetic 
schemes relevant and used in concrete biochemical contexts, such as Michaelis--Menten kinetics \cite{MM13}, Hill kinetics \cite{Hill10}, 
and generalized mass-action kinetics \cite{Muller:12}, all of which are parameter-rich. In contrast to this, classic mass-action 
kinetics \cite{MA64} are not parameter-rich, and therefore the results presented here do not automatically apply in such a restricted framework. %[check that it is indeed the case in our models (perhaps mass action works)].
% Nevertheless, we emphasize 
%that the law of mass action does not appear to hold reliably for non-elementary reactions intercellular and intracellular environments.

%{\color{red} PERHAPS SHIFT TO THE END. BEFORE DISCUSSION}

\paragraph{Overview of the results obtained.} 
We will show the 
following features:
\begin{itemize}
\item The stoichiometry of our central Notch--Delta model, \ref{eq:modeli} 
does not structurally support the capacity for differentiation as defined above. 
Even without {any} 
kinetic symmetry constraints, zero-eigenvalue bifurcations are excluded under any monotone kinetics. 
The system admits only one steady state for any choice of parameters. 
%[NV: We could mention that differentiation might occur under non-monotone kinetics, 
%but I suggest postponing this until we assess how the paper stands without it.]

\item
{We prove} zero-eigenvalue bifurcation - also under kinetic symmetry constraints - for both variants 
of the central model: \ref{eq:modeli}+\ref{eq:modelii}, 
and \ref{eq:modeliii}.  
There is an analogy between variant \ref{eq:modeli}+\ref{eq:modelii} and 
 \ref{eq:modeliii}. Both models contain a `Notch--Delta' interaction without direct effects. 
 These seemingly ineffective interactions however subtly introduce non-autocatalytic positive feedback, in the sense introduced in Sec.~\ref{sec:preliminary}, which enables the capacity for differentiation. %For more details see Sec.~\ref{sec:preliminary}.

\item
Then we take the purely mathematical point of view and propose
and analyze 
 minimal mathematical models, 
symmetric and for two cells with 
an analogous type of differentiation as described above and with similar constraints. 
\end{itemize}

The paper is organized as follows. Sec.~\ref{sec:preliminary} introduces the formalism on reaction networks
needed for our analysis. Sec.~\ref{sec:mainresult} formally states the main 
results for \ref{eq:modeli}, and for the two extended models 
\ref{eq:modeli}+\ref{eq:modelii}, and \ref{eq:modeliii}. Further - for 
the latter two models - we present the sources of instability, 
in form of network motifs, which these modifications introduce when compared 
to \ref{eq:modeli}. In Sec.~\ref{sec:minimal}, we propose and analyze several minimal reaction network models of two initially identical cells  with the capacity for differentiation. 
Sec.~\ref{sec:discussion} summarizes and discusses our results. Proofs are presented in Sec.~\ref{sec:proofs}.

\section{Some Formalisms for Reaction Networks}\label{sec:preliminary}

A reaction network $\mathbf{\Gamma}$ is {a pair} of sets $(\mathbf{M},\mathbf{E})$, where $\mathbf{M}=(\ce{X}_1,...,\ce{X}_{|\mathbf{M}|})$ 
is the set of species of the network, and  $\mathbf{E}$ is the set of reactions. Each reaction $j$ is a {directed relation} between nonnegative linear combination of the species
\begin{equation}\label{eq:reactionj}
j:\quad \sum_{m=1}^{|\mathbf{M}|} s^j_m \ce{X}_m \quad \underset{j}{\rightarrow} \quad  \sum_{m=1}^{|\mathbf{M}|} \tilde{s}^j_m \ce{X}_m,
\end{equation}
where the $s^j_m,\tilde{s}^j_m \ge 0$ are {so-called} 
\emph{stoichiometric coefficients} of reaction $j$. 

Species $X_m$ on the left-hand side of \eqref{eq:reactionj} with coefficients 
$s^j_m>0$ are called \emph{reactants} of reaction $j$ 
and species $X_m$ on the right-hand side with 
{$\tilde{s}^j_m>0$} are called \emph{products} of the reaction $j$. 
The $|\mathbf{M}|\times|\mathbf{E}|$ \emph{stoichiometric matrix} $S$ with entries 
\begin{equation}\label{eq:stmat}
S_{mj}=\tilde{s}^j_m - s^j_m, \quad m= 1, \dots, |\mathbf{M}|, \quad j = 1, \dots, |\mathbf{E}|
\end{equation}
encodes the net consumption or production of each species in each reaction. {In practice,
stoichiometric coefficients are integers, most often 0, 1, 2.  For instance, the models \ref{eq:modeli}, \ref{eq:modeli}+\ref{eq:modelii}, and \ref{eq:modeliii} only involve the coefficients 0 and 1. Thus
 their stoichiometric matrices have entries $S_{mj}\in\{-1,0,1\}$. A natural explanation is that reactions requiring the simultaneous encounter of many molecules are unlikely to occur.
Nevertheless, from the mathematical perspective, there is no reason to impose
such a restriction.}

We assume 
each reaction to be irreversible by fixing an orientation. While thermodynamics dictates that all processes are fundamentally reversible, this irreversibility assumption 
is valid for approximations of biochemical networks and allows for the analysis of a smaller system.
Therefore also our model is idealized and, as usual, it is not intended to account for all reactions present in the respective networks. In this sense, we repeat that our models \ref{eq:modeli}, \ref{eq:modeli}+\ref{eq:modelii}, and \ref{eq:modeliii} omit production (inflow) and degradation (outflow) reactions, respectively
\begin{equation}
\ldots \quad \rightarrow \quad \ce{X}_m \quad\quad\quad \text{and} \quad\quad\quad \ce{X}_m \quad\rightarrow \quad \dots \quad,
\end{equation}
where the `$\ldots$' denote species external to the system.

The related ODE system describing the time-evolution $x(t)\in \mathbb{R}^{|\mathbf{M}|}_{>0}$ of the positive concentrations of $\ce{X}_1, \dots , \ce{X}_{|\mathbf{M}|}$ reads
\begin{equation}\label{eq:maineq}
\dot{x}=f(x):= S\mathbf{r}(x),
\end{equation}
where $S$ is {defined} in \eqref{eq:stmat} and $\mathbf{r}(x) \in \mathbb{R}^{|\mathbf{E}|}$ is the vector of nonlinear reaction rates (kinetics) that define the chemical laws governing each reaction. Left kernel vectors $w$ of the stoichiometric matrix $%\mathbb
{S}$ identify \emph{conservation laws}, since
\begin{equation}\label{eq:conquan}
\dfrac{d (w x(t))}{dt}=w\dot{x}=w%\mathbb
{S}\mathbf{r}({x})=0,
\end{equation}
\and thus $w{x}=K_w$ for a constant $K_w$, 
which is determined by the initial conditions 
$x_0\in \mathbb{R}^{|\mathbf{M}|}_{>0}$.  
Complementarily, for each $x_0$ the sets
\begin{equation}
{x_0 + \{Sv\;|\; v\in \mathbb{R}^{|\mathbf{E}|}\}}
\end{equation}
which are invariant under the flow of \eqref{eq:maineq}, are called \emph{stoichiometric compatibility classes}. 
In order to reduce the system to a given stoichiometric compatibility class, 
we take a basis $\{w_1,...,w_n\}$ of $\operatorname{ker}S^T$ and 
solve the linear constraints
\begin{equation}
w_{l} x = K_{w_l} = K_{w_l} (x_0)
%\begin{cases}
%    w_{1}x=C_{w_1}({x_0})\\
%    \vdots\\
%    w_{n}x=C_{w_n}({x_0})
%\end{cases}
\end{equation}
for $l=1, \dots, n$ arbitrary variables, {with} $n$ being the dimension of $\operatorname{ker}S^T$.

Since the exact quantitative form of the reaction rates $r_j$ is generally unknown, it is common in the literature to represent them as a vector of parametric functions $\mathbf{r}(x, \mathbf{p})$, depending on a set of parameters 
$\mathbf{p}$. 
However, even the very choice of the parametric nonlinearity is arbitrary 
for our models.
Accordingly, in this paper we only assume that all reaction rates $r_{j}(x,\mathbf{p})$ are \emph{monotone chemical functions} in the following sense.
\begin{definition}\label{Monotone chemical functions}
A function $r_{j}(x)$ is called \emph{monotone chemical} if 
\begin{enumerate}
\item $r_j(x) \ge 0$, for all $x\in\mathbb{R}^{\mathbf{|M|}}_{\ge0}$;
\item $r_j(x)>0$ if and only if $x_m >0$ for all $m$ with $s^j_m>0$;
\item for $s^j_m=0$ we have $\partial r_j / \partial x_m \equiv 0$;
\item for $x>0$ and $s^j_m>0$ we have $\partial r_j / \partial x_m > 0$.
\end{enumerate}
\end{definition}
{Thus} we consider monotone nonlinearities whose dependencies are fully encoded by the stoichiometry of the system. Widely-used kinetics such as mass action (classic and generalized), Michaelis--Menten, and Hill kinetics all are monotone chemical function.

\subsection{From stoichiometry to {singular} steady-states}

A positive steady-state  
$\bar{x}$ of system \eqref{eq:maineq},
\begin{equation}\label{eq:jacobianmatrix}
0=f(\bar{x})=S\mathbf{r}(\bar{x}),
\end{equation}
requires the existence of a right kernel vector $v>0$ of the stoichiometric matrix,
i.e. $Sv=0$. Networks satisfying this structural property are called \emph{consistent} in the literature \cite{Ang07}.
The Jacobian matrix 
reads
\begin{equation}\label{eq:Jacobiannv}
f_x=  (S \mathbf{r}(x))_x=S(\mathbf{r}(x))_x,
\end{equation}
where $(\mathbf{r}(x))_x$ indicates the $\mathbf{E}\times\mathbf{M}$ matrix of the partial derivatives:
\begin{equation}
    \big(\;(\mathbf{r}(x))_x\; \big)_{jm}=\dfrac{\partial r_j}{\partial x_m}(x).
\end{equation}
Since any $r_j(x)$ is a monotone chemical function, we have that $\partial r_j(x) / \partial x_m  \neq 0$ at $x>0$ 
if and only if $s^j_m >0$, i.e., the species $\ce{X}_m$ is a reactant of reaction $j$. Therefore we consider the Jacobian matrix of $f$ symbolically in the following sense.
\begin{definition}
The entries of the $|\mathbf{E}|\times |\mathbf{M}|$ \emph{symbolic reactivity matrix}  $R$ for system 
\eqref{eq:maineq} are defined by
\begin{equation}
R_{jm}:=\begin{cases}
r_{jm} \quad \text{if $s^j_m>0$}; \quad {\mbox{ here } r_{jm}> 0}\\
0 \quad \text{otherwise} \; .
\end{cases}
\end{equation}
{Their} relation to our system will be given below. 
The $|\mathbf{M}|\times |\mathbf{M}|$ \emph{symbolic Jacobian} $G$ is then defined as
\begin{equation}
G:=SR.
\end{equation}
\end{definition}
Now fix a concentration vector $\bar{x}>0$ and a choice $\bar{R}$ for the positive parameters in the symbolic reactivity matrix, i.e. $\bar{R}_{jm}=\{\bar{r}_{jm}\}$. Consider a consistent network. Clearly, in the wide class of monotone chemical functions, we can always find a suitable nonlinearity $\mathbf{r}$ such that 
\begin{equation}
\begin{cases}
S\mathbf{r}(\bar{x})=0;\\
\dfrac{\partial r_j(x)}{\partial x_m} \bigg|_{x=\bar{x}}=\bar{r}_{jm}\quad\text{for all $j=1, \dots ,  |\mathbf{E}|$ and $m=1, \dots , |\mathbf{M}|$}.
\end{cases}
\end{equation}
Here we want to be explicit on the possible choices of nonlinearities. 
%\emph{parameter-rich} kinetic models \cite{VasStad24} as follows. 
\begin{definition}\label{def:prich}
A parametric kinetic rate model $r_j(x,\mathbf{p})$ is \emph{parameter-rich} if for every choice of a steady-state value $\bar{x}>0$ and any choice of the symbolic reactivity matrix $\bar{R}$, there exist parameters $\bar{\mathbf{p}}$ such that
\begin{equation}
\bar{r}_{jm}=\dfrac{\partial r_j (x,\mathbf{p})}{\partial x_m}\bigg|_{(x,\mathbf{p})=(\bar{x},\bar{\mathbf{p}})} \; .
\end{equation}
\end{definition}

As mentioned already, most reaction schemes used in biochemical contexts, such as {Mi\-cha\-elis--Menten }kinetics, Hill kinetics, and generalized mass-action laws, are parameter-rich \cite{VasStad24}. A crucial advantage of this framework is that for a consistent reaction network with parameter-rich kinetics, it suffices to analyze how the stability of the symbolic Jacobian $SR$ changes as the parameters in $R$ vary in order to detect steady-state bifurcations.

In this paper we consider differentiation as a zero-eigenvalue bifurcation, i.e., when a real eigenvalue of the Jacobian crosses the imaginary axis. From the parameter-rich kinetics point of view, these dynamics can be detected by the characteristic polynomial {of the symbolic
Jacobian $G=SR$:}
\begin{equation}\label{eq:charpoly}
g(\lambda)=\sum_{k=0}^{|\mathbf{M}|}a_k (-\lambda)^{|\mathbf{M}|-k}  \; .
\end{equation}
Each of the coefficients $a_k$, sum of the principal minors of 
order $k$ of $G=SR$, is a function of the values of those symbolic entries $r_{jm}$ of $R$ which are larger than zero.\\ 
Conservation laws may force one or more of the coefficients $a_k$ to vanish identically, regardless of the choice of $R$. In reaction networks, this typically arises from linear conservation laws $w$ with $wS=0$. {We call} $a_k$ \emph{algebraically nonzero} if $a_k\not\equiv 0$ 
as a function of the entries of the symbolic reactivity matrix $R$. 
Let $a_{\tilde{k}}$ denote the algebraically nonzero coefficient with the highest index $\tilde{k}$. 
If the system is full rank and in particular has no conservation laws, 
{then} $a_{|\mathbf{M}|}$ corresponds to $\det G$.
In general, we say that the system is \emph{nondegenerate}, if the highest index $\tilde{k}$ for nonzero coefficient $a_{\tilde{k}}$ satisfies $\tilde{k}=|\mathbf{M}|-n$, where 
$n=\operatorname{dim}\operatorname{ker} S^T$ is the number of linearly 
independent conservation laws. In this case $a_{\tilde{k}}$ corresponds 
to the determinant of the Jacobian of \eqref{eq:Jacobiannv} 
%the right hand side of system \eqref{eq:maineq} 
reduced to a stoichiometric compatibility class. Thus differentiation can only occur if this algebraically nonzero coefficient is zero for a specific choice of parameters in $R$.

In order to tackle the question of whether there exists a choice $R^*$ such that the coefficient $a_{\tilde{k}}=0$
%, we leverage a formal structural expansion of the coefficients of the characteristic polynomial. The 
we use \emph{Child-Selections} as the main tool.
%, which definition we recall from \cite{VasStad24}.
%{\color{red}[All o.k. with $J(X_m), J(m)$ etc. now?]}
\begin{definition}
For a network $\pmb{\Gamma}=(\mathbf{M},\mathbf{E})$,
a \emph{$k$-Child-Selection triple} $\pmb{\kappa}=(\kappa,E_{\kappa},J)$, 
or $k$-CS, is 
%$\pmb{\kappa}=(\kappa,E_{\kappa},J)$ 
such that $|\kappa|=|E_{\kappa}|=k$,
$\kappa\subseteq \mathbf{M}$, $E_{\kappa}\subseteq \mathbf{E}$, and $J:\kappa\to
E_{\kappa}$ is the \emph{Child-Selection bijection} 
satisfying $s_m^{{J(\ce{X}_m)}}>0$ for all
${\ce{X}_m}\in\kappa$. 
\end{definition}
In particular, a $k$-CS $\pmb{\kappa}=(\kappa,E_\kappa, J)$ defines a
$k\times k$ stoichiometric matrix $S[\pmb{\kappa}]$ with entries
\begin{equation}
S[\pmb{\kappa}]_{ml} := S[\kappa,E_{\kappa}]_{m,{J(\ce{X}_l)}} =
\tilde s_{m}^{{J(\ce{X}_l)}} - s_{m}^{{J(\ce{X}_l)}}.
\end{equation}
Here, $S[\kappa,E_{\kappa}]$ denotes the submatrix of the stoichiometric matrix $S$ with row-indices in $\kappa$ and column-indices in $E_\kappa$. The matrix $S[\pmb{\kappa}]$ is then obtained from $S[\kappa,E_{\kappa}]$ by reshuffling the columns according to {$J$}. 
In particular, {if  $(\ce{X}_1,\dots,\ce{X}_k)$ are} the $k$ 
species relabeled 
in $\kappa$, then the stoichiometric column corresponding to the reaction $j={J(\ce{X}_m)}$ appears in $S[\pmb{\kappa}]$ as the $m^\text{th}$ column. For a given $k$-CS $\pmb{\kappa}$, we refer to the associated matrix $S[\pmb{\kappa}]$ as its \emph{Child-Selection matrix} (CS-matrix).  
If no species {$\ce{X}_m$} appears as both a reactant and product of any single reaction $j$, then CS-matrices identify all square submatrices of the stoichiometric matrix $S$ that, upon reshuffling of its columns, possess a strictly negative diagonal. This is the case in all our models   \ref{eq:modeli}, \ref{eq:modeli}+\ref{eq:modelii}, and \ref{eq:modeliii}. 

CS-matrices are instrumental in understanding the sign-changes of coefficients $a_k$ of the characteristic polynomial.
\begin{lemma}[Lemma 5.4{, \cite{VasHunt}}]\label{lem:CSexpansion}
Let $a_k$ be a coefficient of {the characteristic polynomial $g(\lambda)$ in \eqref{eq:charpoly}
of the symbolic Jacobian $G=SR$}  
. Then %the following expansion holds:
\begin{equation}\label{eq:CSexpansion}
a_k=\sum_\mathbf{\pmb{\kappa}}\operatorname{det}S[\mathbf{\pmb{\kappa}}] 
\prod_{{X_m}\in\kappa} R_{{J(X_m)}m},
\end{equation}
{where the summation is taken over all $k$-CS $\pmb{\kappa}$.}
\end{lemma}
In particular, since any $a_k$ is a multilinear homogeneous polynomial in the symbolic variables $R_{jm}$, 
it follows that
\begin{cor}[Corollary 5.7{, \cite{VasHunt}}]\label{cor:twosignsexpansion}
There exists a choice of $R$ such that $a_k=0$ if and only if there exist two $k\times k$ CS-matrices $S[\pmb{\kappa}_1]$ and $S[\pmb{\kappa}_2]$ with
\begin{equation}\label{eq:differentsign}
\operatorname{det}S[\pmb{\kappa}_1] \operatorname{det}S[\pmb{\kappa}_2]<0.
\end{equation}
\end{cor}
{Thus for consistent and nondegenerate networks, 
zero-eigenvalue bifurcations and the capacity for differentiation} require co-existence of two $\tilde{k}\times\tilde{k}$ CS-matrices with determinants of opposite sign, where again $\tilde{k}$ indicates the highest index of the nonzero coefficients $a_k$ in the characteristic polynomial of the symbolic Jacobian. \\ 

A second relevant corollary which follows from Lemma \ref{lem:CSexpansion} is
\begin{cor}[Corollary 5.1, in \cite{VasStad24}]\label{cor:csinst}
Assume there exists an unstable CS-matrix $S[\pmb{\kappa}]$, 
i.e. $S[\pmb{\kappa}]$ has at least one eigenvalue with positive real part,
then there exists a choice of $R$ such that the symbolic Jacobian $G=SR$ is unstable.
\end{cor}

\subsection{Positive feedback and instability motifs}
Let $\pmb{\kappa}=(\kappa,E_{\kappa},J)$ be a $k$-CS. Child-Selections are 
naturally nested: any restriction $\pmb{\kappa}'$ of $\pmb{\kappa}$, 
defined by a subset $\kappa'\subset \kappa$ and $E_{\kappa'}=
{J(\kappa')}\subset E_{\kappa}$, forms a $k'$-CS \; $\pmb{\kappa}'=(\kappa',E_{\kappa'},J)$.  The CS-matrix for $\pmb{\kappa}'$ is a principal submatrix of that for $\pmb{\kappa}$. This motivates to focus on minimal matrices satisfying a given property, in the following sense.

%A matrix $A$ is \emph{stable} if all eigenvalues have negative real part and \emph{unstable} if it has at least one %eigenvalue with positive real part. 
{We consider the property of a CS-matrix being unstable with a real positive eigenvalue. An \emph{unstable-positive feedback} is a $k\times k$ CS-matrix satisfying 
\begin{equation}\label{eq:posfeed}
\operatorname{sign}\operatorname{det}S[\pmb{\kappa}]=(-1)^{k-1},
\end{equation}
{such that no determinant of a $k'\times k'$ principal submatrix has 
the} sign $(-1)^{k'-1}$. {Descartes'} rule of signs applied to 
its characteristic polynomial shows that such {a} $S[\pmb{\kappa}]$ 
has exactly one real positive eigenvalue. 
{The presence of an unstable-positive feedback is sufficient 
for a consistent network to admit an unstable steady-state
due to Corollary~\ref{cor:csinst}.}
We refer to \cite{VasStad24} for further details and for an illustration 
of the general concept of \emph{unstable cores}, which encompass both positive 
and negative feedback. Corollary~\ref{cor:twosignsexpansion} implies that the occurrence of zero-eigenvalue bifurcations - and thus the capacity for differentiation - requires 
the presence of a CS matrix in the network, which is an unstable-positive 
feedback, since \eqref{eq:differentsign} cannot be satisfied if no 
CS-matrix in the network fulfills~\eqref{eq:posfeed}. Therefore, we will 
search in all three models, \ref{eq:modeli}, \ref{eq:modeli}+\ref{eq:modelii}, and \ref{eq:modeliii}, 
for CS-matrices that are unstable-positive feedbacks. We will focus on minimal topological structures via \emph{instability motifs}.
\begin{definition}
Consider a $k$-CS, $\pmb{\kappa}=(\kappa,E_\kappa,J)$ in a network $\mathbf{\Gamma}=(\mathbf{M},\mathbf{E})$ such that its associated CS matrix $S[\pmb{\kappa}]$ is an unstable-positive feedback. The associated \emph{instability motif} is defined as the subnetwork $\mathbf{\Gamma}[\pmb{\kappa}]=(\kappa, E_\kappa)\subseteq \mathbf{\Gamma}$ consisting of the species in $\kappa$ and the reactions in $E_\kappa$. If a reaction $j$ in $E_\kappa$ involves additional species $\ce{X}_m \not\in \kappa$, those species are omitted and not considered in the instability motif (see Frame~\ref{fig:toymodel} for 
an illustration).
\end{definition}
 This focus is particularly suitable for incomplete models, as most biochemical models inherently are. The presence of an unstable-positive feedback is preserved if additional reactions or species are included. Moreover, no additional unstable-positive feedback arises from the inclusion of degradation reactions. Any CS-matrix involving such reactions contains a column with a single $-1$ on the diagonal, so that its eigenvalues are given by $\{-1\}$ jointly with those of the corresponding principal submatrix obtained by omitting the degradation reaction and the corresponding species. This supports our choice to dismiss such reactions.

Finally, in experiments, \emph{autocatalysis} is a fundamental 
mechanism for self-amplification and positive feedback in chemical networks. Its formalization and its role in the origin of life and evolution in general is still much debated \cite{Andersen2020}. Here we follow the formalism in \cite{blokhuis20} which is based on the definition of the International Union of Pure and Applied Chemistry (IUPAC) \cite{IUPAC+C00876+2025}, and addresses autocatalysis from a purely structural and stoichiometric perspective. 
In this framework, autocatalysis is characterized by  certain stoichiometric submatrices named 
\emph{autocatalytic cores}. However, the presence of autocatalysis in this form does not necessarily translate into experimentally observable autocatalysis, as actual self-amplification depends not only on stoichiometry but also on more specific conditions on the reaction rates \cite{VasStad24, Nandan25}.

The formalism in \cite{VasStad24} implies that autocatalytic cores correspond exactly to CS-matrices that are unstable-positive feedbacks as defined above, with the further condition that off-diagonal entries are nonnegative
(Metzler matrices), so that we use the following:
\begin{definition}
In our context, a reaction network is called \emph{autocatalytic} if it possesses at least one unstable-positive feedback $S[\pmb{\kappa}]$ that is a Metzler matrix.
\end{definition}
Intuitively, such structures generalize simple autocatalytic pathways as e.g.
\begin{equation}
\ce{X}_1  \quad\ce{->}\quad \ce{X}_2 \quad\ce{->}\quad 2\ce{X}_1,
\end{equation}
which corresponds to the CS matrix $
\begin{pmatrix}
-1 & 2\\
1 & -1
\end{pmatrix}$, which is Metzler and unstable. 
In turn, unstable-positive 
feedbacks can also be \emph{non-autocatalytic}, in particular with some negative off-diagonal entry. Our results here show that the positive feedback, if present 
in our three models, is always non-autocatalytic.

\paragraph{Workflow - analysis of models \ref{eq:modeli}, \ref{eq:modeli}+\ref{eq:modelii}, \ref{eq:modeliii}.}
%Using the concepts introduced above, we analyze the biological models as schematically summarized below.
\begin{enumerate}
\item We check that the respective network is \emph{consistent} and \emph{nondegenerate}. This guarantees that the algebraically nonzero coefficient $a_{\tilde{k}}$ in the characteristic polynomial of the Jacobian matrix, with $a_k\equiv 0$ for $k>\tilde{k}$, corresponds to the determinant of the Jacobian restricted to a stoichiometric compatibility class, so that the stability of steady-states can be addressed by analyzing this coefficient.
\item We enforce \emph{kinetic symmetry} by assuming that the rates $r_j$ 
of symmetry-related reactions are identical. {The system 
is considered} at a homogeneous steady-state where the concentrations of symmetry-related species are also identical. Consequently, the symmetry-related entries $R_{jm}$ of $R$ coincide.
\item Under such symmetry constraint, we explicitly compute the characteristic polynomial $g(\lambda)$ of the symbolic Jacobian associated with the respective model.
\item We examine the coefficients $a_k$ of $g(\lambda)$ for terms with opposite signs in their expansion \eqref{eq:CSexpansion}. If no such terms are present, zero-eigenvalue bifurcation is precluded, and thus the capacity for differentiation. If such terms occur in the nonzero coefficient $a_{\tilde{k}}$, we say that the system has the capacity for a zero-eigenvalue bifurcation and differentiation.
\item When terms of opposite sign are present, we study the associated CS-matrices and identify the minimal ones, namely the $k\times k$ matrices $S[\pmb{\kappa}]$ satisfying \eqref{eq:posfeed} and such that no $k'\times k'$ principal minor has sign $(-1)^{k'-1}$. These matrices correspond to unstable-positive feedbacks, and the presence (resp. absence) of any such feedback that is also a Metzler matrix classifies the network as autocatalytic (resp. non-autocatalytic). Finally, we identify the corresponding \emph{instability motif}.
\end{enumerate}

\renewcommand{\tablename}{Frame}
\begin{table}[H]
\begin{mdframed}
{\scriptsize {{\textbf{Illustration of our {approach}:} Consider  
three species $\ce{X},\ce{Y}, \ce{Z}$, and three reactions $1,2,3$:
\begin{equation*}
\begin{split}
\ce{X}+\ce{Y}  \quad&\underset{1}{\ce{->}}\quad \ce{Z} 
\quad\underset{2}{\ce{->}}\quad 2\ce{X} \; \\
&\,\,\ce{X}\quad\underset{3}{\ce{->}}\quad \ce{Y}
\end{split}
\end{equation*}
The associated ODEs system reads:
\begin{equation}
\begin{pmatrix}
    \dot{x}\\
    \dot{y}\\
    \dot{z}\\
\end{pmatrix}=S\mathbf{r}(x)=\begin{pmatrix}
        -1 & 2 & -1\\
        -1 & 0 & 1\\
        1 & -1 & 0
    \end{pmatrix}\begin{pmatrix}
        r_1(x,y)\\
        r_2(z)\\
        r_3(x)
    \end{pmatrix}= \begin{pmatrix}
       -r_1(x,y)+2r_2(z)-r_3(x)\\
   -r_1(x,y)+r_3(x)\\
   r_1(x,y)-r_2(z)\\ 
    \end{pmatrix}
\end{equation}
The stoichiometric matrix $S$ admits a right positive kernel vector $v=(c,c,c)$, $c\in\mathbb{R}_{>0}$, and thus the network is consistent and admits a positive steady-state $(\bar{x},\bar{y},\bar{z})$ whenever $r_1(\bar{x},\bar{y})=r_2(\bar{z})=r_3(\bar{x})$. The symbolic reactivity matrix $R$ and 
the Jacobian $G=SR$ are 
\begin{equation*}
R=\begin{pmatrix}
        r_{1{\ce{X}}} & r_{1{\ce{Y}}} & 0\\
        0 & 0 & r_{2\ce{Z}}\\
        r_{3\ce{X}}& 0 & 0
    \end{pmatrix}; \quad G=\begin{pmatrix} 
    -r_{1{\ce{X}}}-r_{3\ce{X}} & -r_{1{\ce{Y}}} & 2r_{2\ce{Z}}\\
    -r_{1{\ce{X}}}+r_{3\ce{X}} & -r_{1{\ce{Y}}} & 0\\
    r_{1{\ce{X}}} & r_{1{\ce{Y}}} & -r_{2\ce{Z}}
    \end{pmatrix}.
\end{equation*}
The characteristic polynomial $g(\lambda)$ of $G$ expands via Lemma~\ref{lem:CSexpansion} as
\begin{equation*}
\begin{split}
g(\lambda)
&= -\lambda^3 + (- r_{1\ce{X}} -  r_{1\ce{Y}}-r_{2\ce{Z}}-r_{3\ce{X}})\lambda^2- \bigg(
\operatorname{det}\begin{pmatrix}
-1 & 2\\
1 & -1
\end{pmatrix}r_{1\ce{X}}r_{2\ce{Z}}\\ &+\operatorname{det}\begin{pmatrix}
-1 & -1\\
1 & -1
\end{pmatrix}r_{3\ce{X}}r_{1\ce{Y}}+\operatorname{det}\begin{pmatrix}
-1 & 2\\
0 & -1
\end{pmatrix}r_{3\ce{X}}r_{2\ce{Z}}+\operatorname{det}\begin{pmatrix}
-1 & 0\\
1 & -1
\end{pmatrix}r_{1\ce{Y}}r_{2\ce{Z}}\bigg)\lambda\\
=& -\lambda^3 + (- r_{1\ce{X}} -  r_{1\ce{Y}}-r_{2\ce{Z}}-r_{3\ce{X}})\lambda^2+ (r_{1\ce{X}}r_{2\ce{Z}} - 2r_{3\ce{X}}r_{1\ce{Y}} - r_{3\ce{X}}r_{2\ce{Z}} - r_{1\ce{Y}}r_{2\ce{Z}})\; . 
\end{split}
\end{equation*}
There is only one $k\times k$ CS-matrix with $\operatorname{sign}\det S[\pmb{\kappa}]=(-1)^{k-1}$ and thus only one unstable-positive feedback:
\begin{equation*}
S[\pmb{\kappa}]=
\begin{blockarray}{ccc}
& 1 & 2\\
\begin{block}{c(cc)}
\ce{X} & -1 & 2\\
\ce{Z} & 1 & -1\\
\end{block}
\end{blockarray}\;.
\end{equation*}
This corresponds to the species $\{\ce{X},\ce{Z}\}$ and reactions $\{1,2\}$. 
Since $S[\pmb{\kappa}]$ is a Metzler matrix, it is autocatalytic. 
The matrix is associated with the \emph{instability motif}
\begin{equation}
\ce{X} + \cdots 
\quad\underset{1}{\ce{->}}\quad 
\ce{Z} 
\quad\underset{2}{\ce{->}}\quad 
2\ce{X},
\end{equation}
which omits the species $\ce{Y}$, as it does not appear in the CS-matrix. }
%{Taking} the network verbatim without including any production  or degradation reactions, the stoichiometric matrix $S$ lacks a strictly  positive right-kernel vector $v>0$. Consequently, the network is \emph{not}  consistent and admits \emph{no} positive steady state, but nevertheless it is sufficient to serve our illustration purposes.
}}
\end{mdframed}
\caption{An autocatalytic toy model illustrating our concepts.}
\label{fig:toymodel}
\end{table}

\section{Mathematical analysis of the models}\label{sec:mainresult}
We analyze the ODE systems arising from the reaction networks of our 
central model \ref{eq:modeli}, the cis-model 
\ref{eq:modeli}+\ref{eq:modelii}, and the ligand-activation hypothesis 
\ref{eq:modeliii}, and draw conclusions on their capacity for 
differentiation. As explained in Sec.~\ref{sec:preliminary}, we associate 
the capacity for differentiation with the presence of unstable-positive 
feedback, which we characterize structurally in terms of square stoichiometric 
CS-matrices $S[\pmb{\kappa}]$. Each such matrix is associated with an \emph{instability motif} in the network. See Fig.~\ref{fig:figure2} for an overview of the instability motifs identified in the cis-model \ref{eq:modeli}+\ref{eq:modelii} and in the ligand-activation hypothesis \ref{eq:modeliii}.
\begin{figure}[htbp]
\centering
\setlength{\tabcolsep}{4pt} % no space between columns
\renewcommand{\arraystretch}{0} % no space between rows
\begin{tabular}{|c|c|}
  \hline
  \includegraphics[width=0.38\textwidth,height=0.38\textwidth,keepaspectratio]{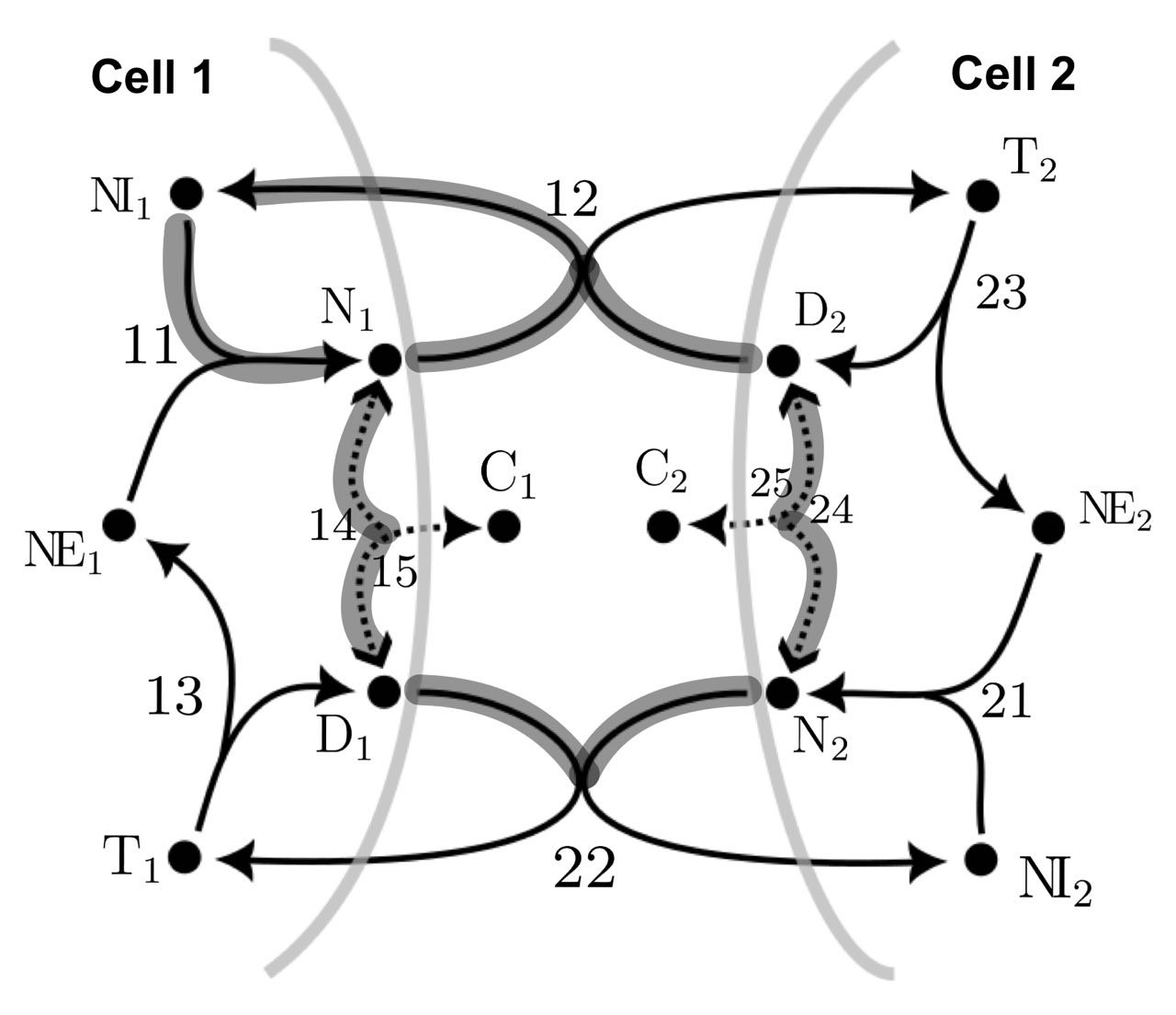} &
\includegraphics[width=0.38\textwidth,height=0.38\textwidth,keepaspectratio]{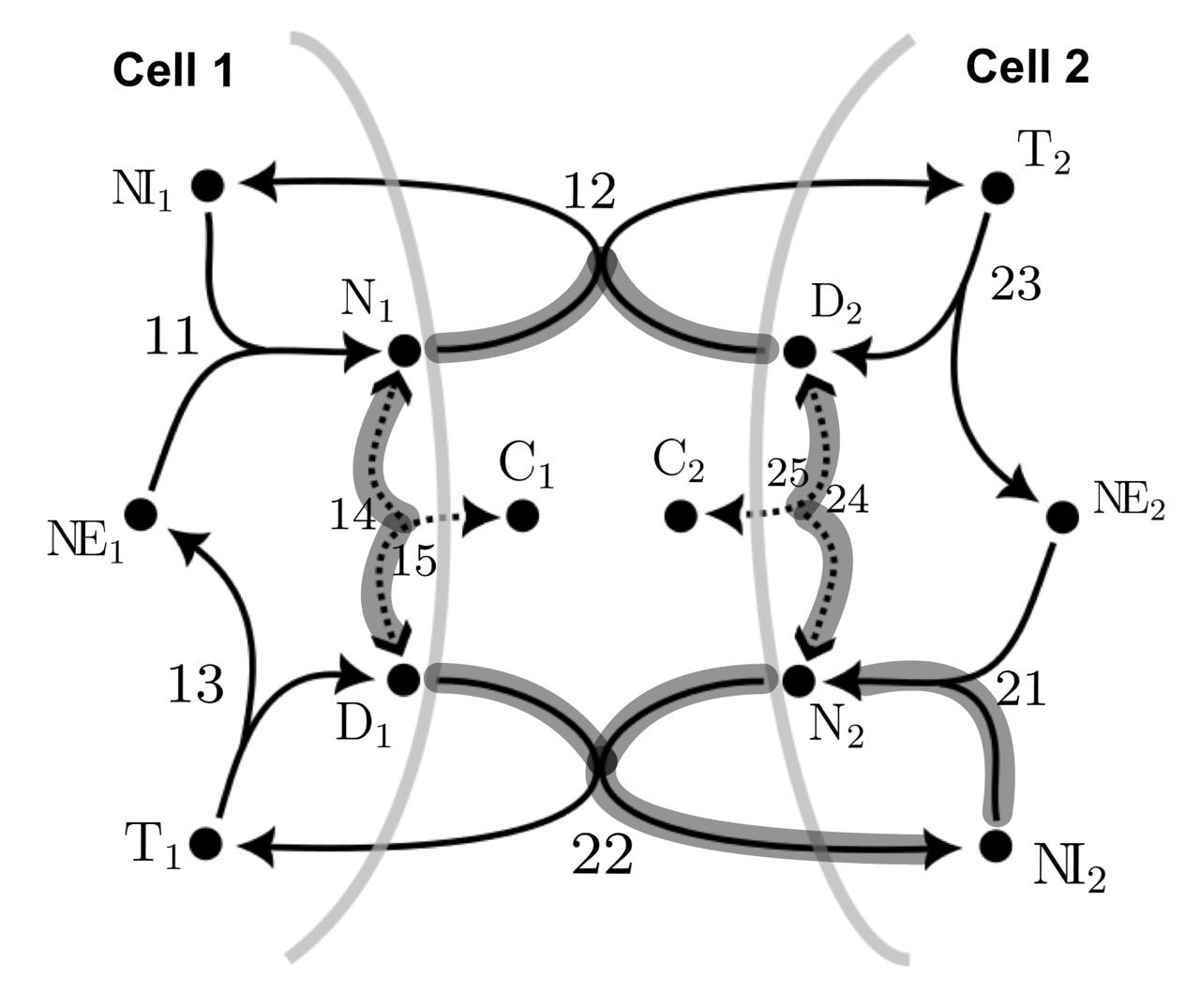}\\
  \hline
  \includegraphics[width=0.38\textwidth,height=0.38\textwidth,keepaspectratio]{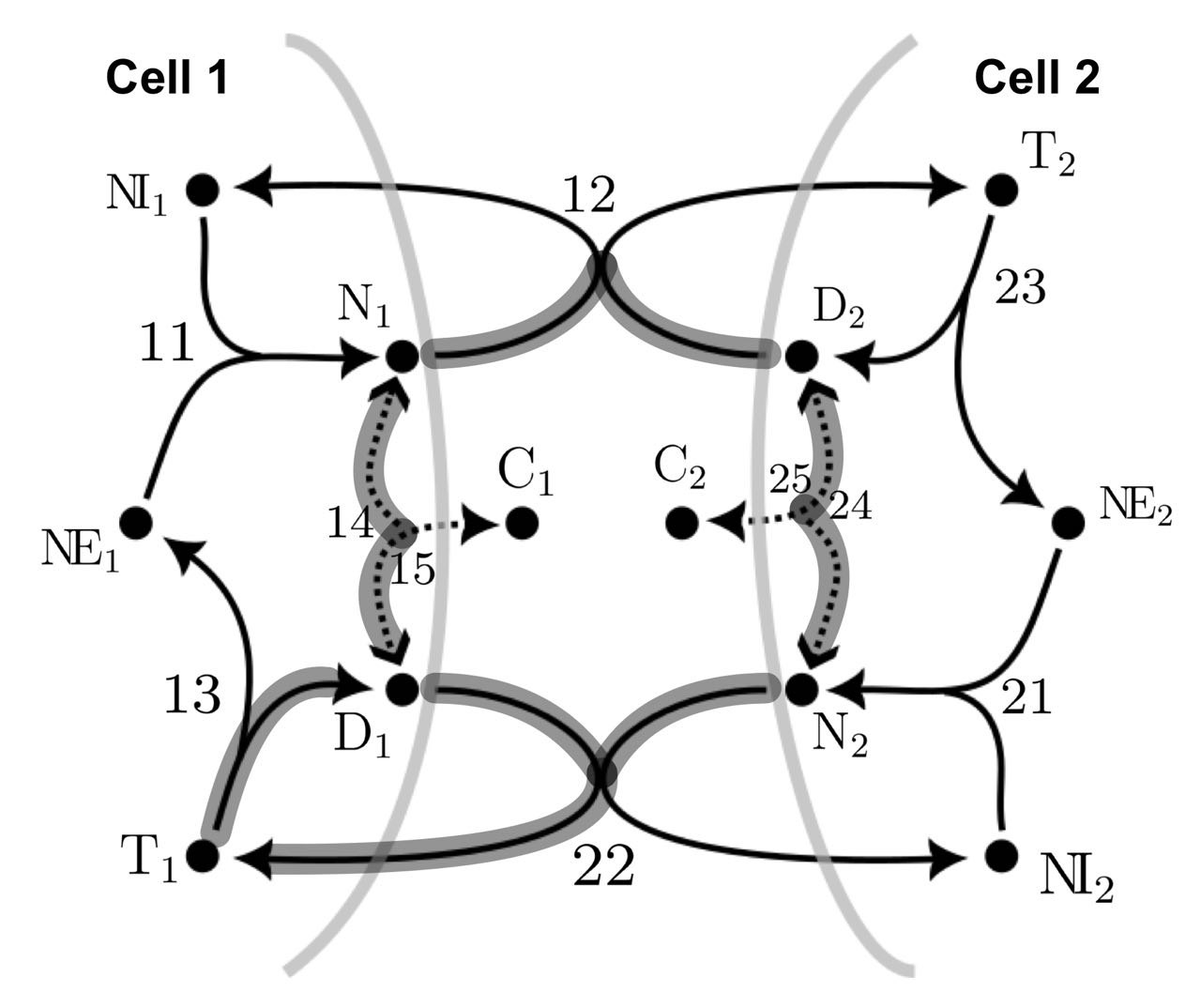} &
   \includegraphics[width=0.38\textwidth,height=0.38\textwidth,keepaspectratio]{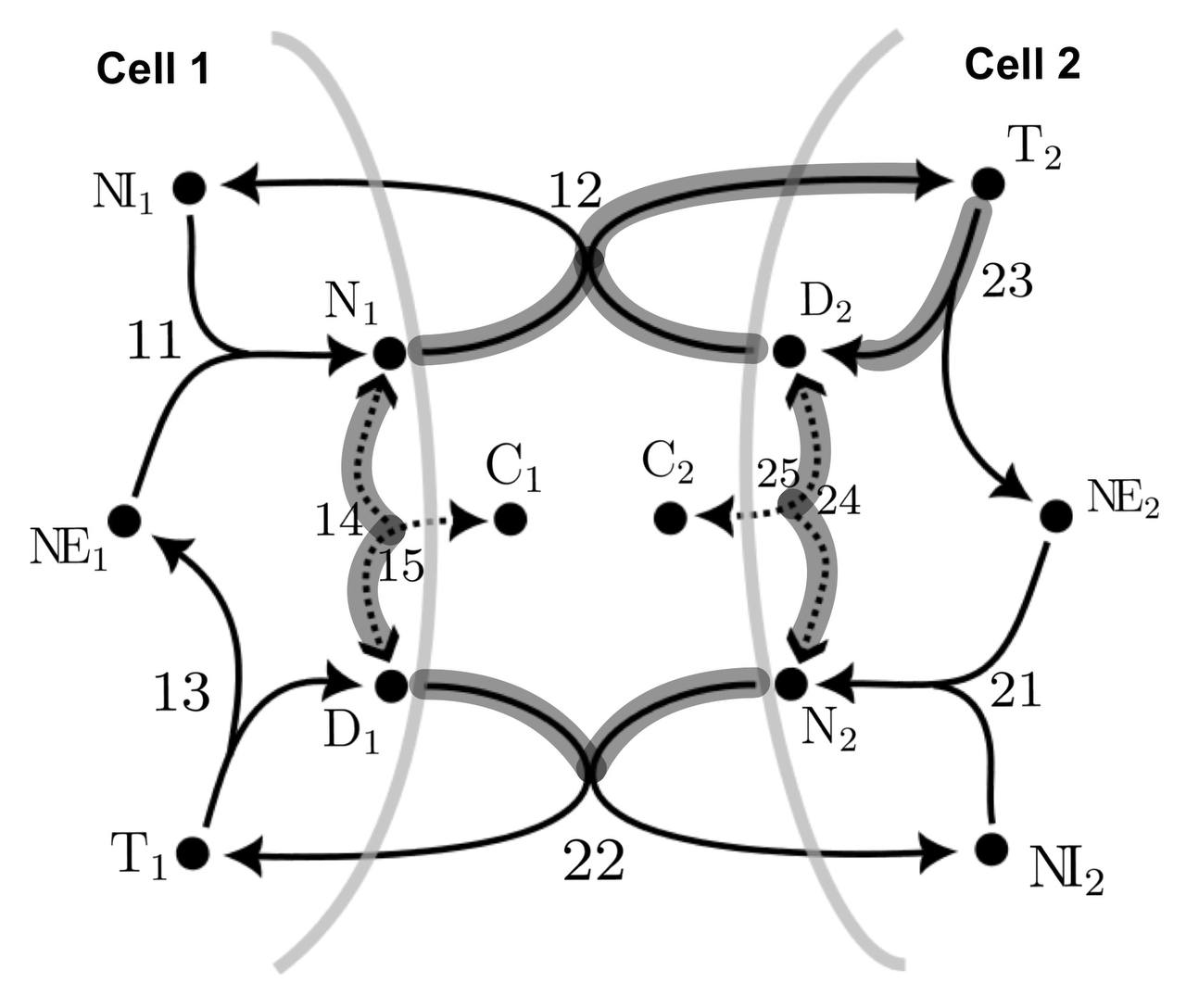}\\
  \hline
  \includegraphics[width=0.38\textwidth,height=0.38\textwidth,keepaspectratio]{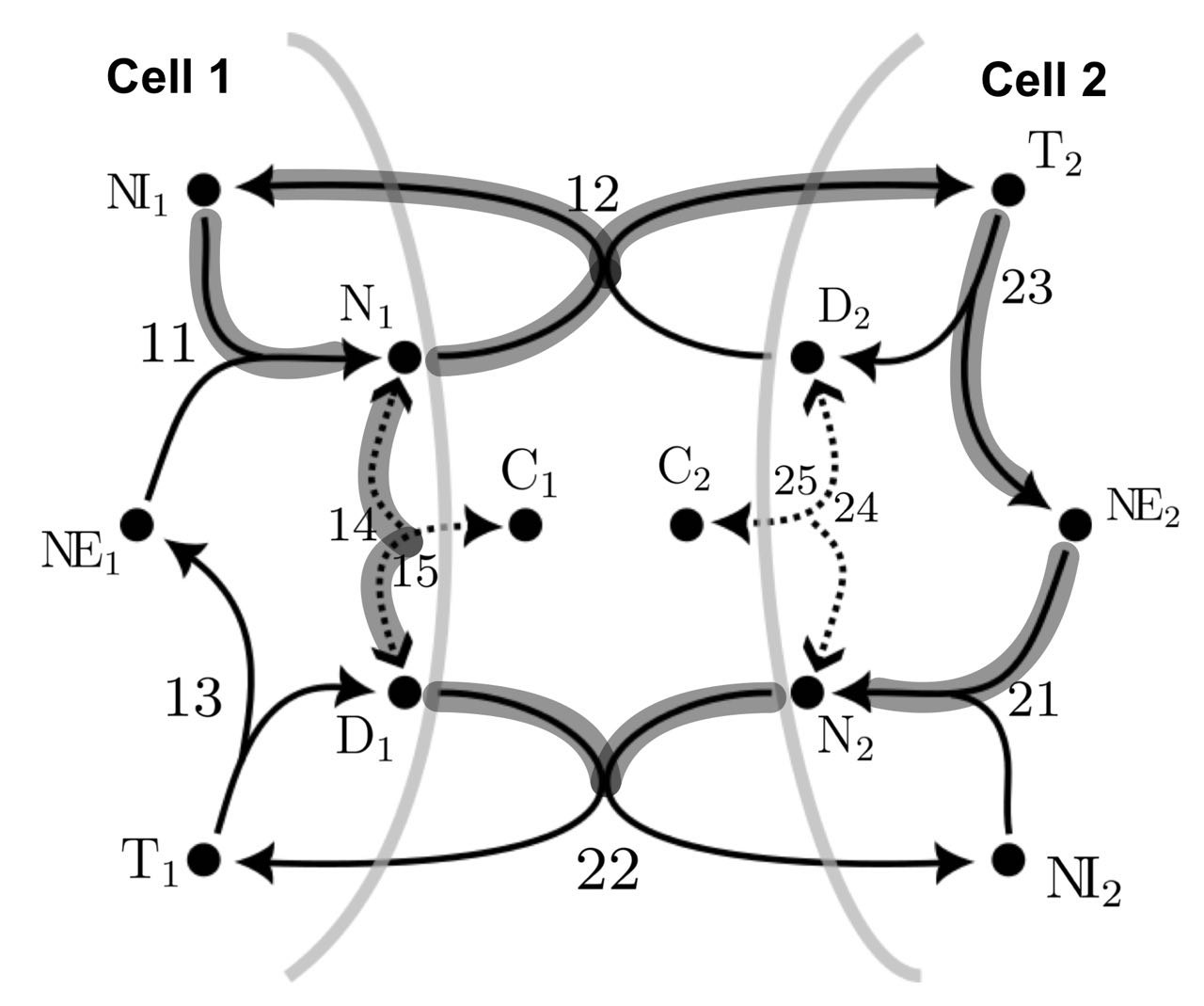} & \includegraphics[width=0.38\textwidth,height=0.38\textwidth,keepaspectratio]{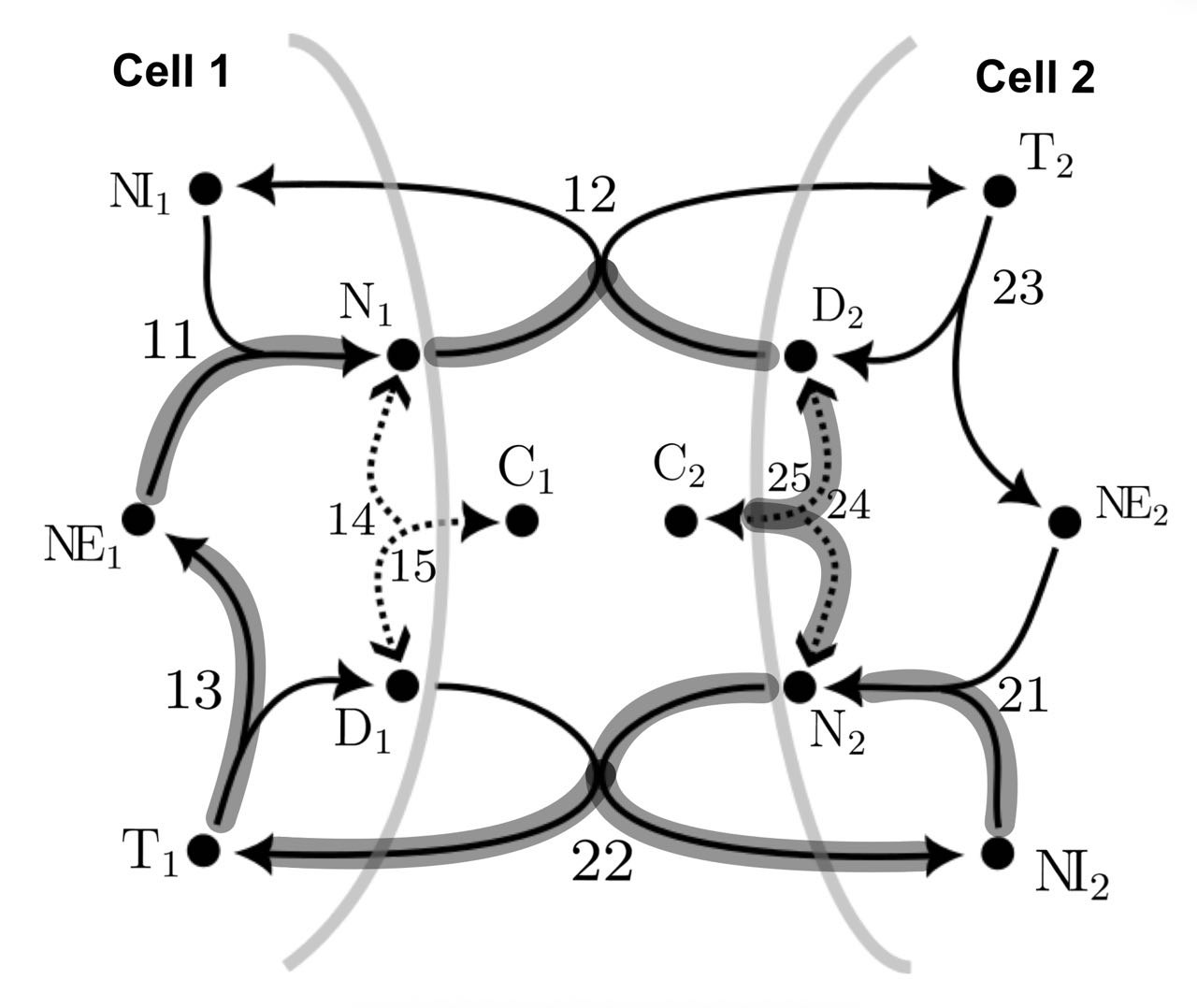}\\
  \hline
  \includegraphics[width=0.38\textwidth,height=0.38\textwidth,keepaspectratio]{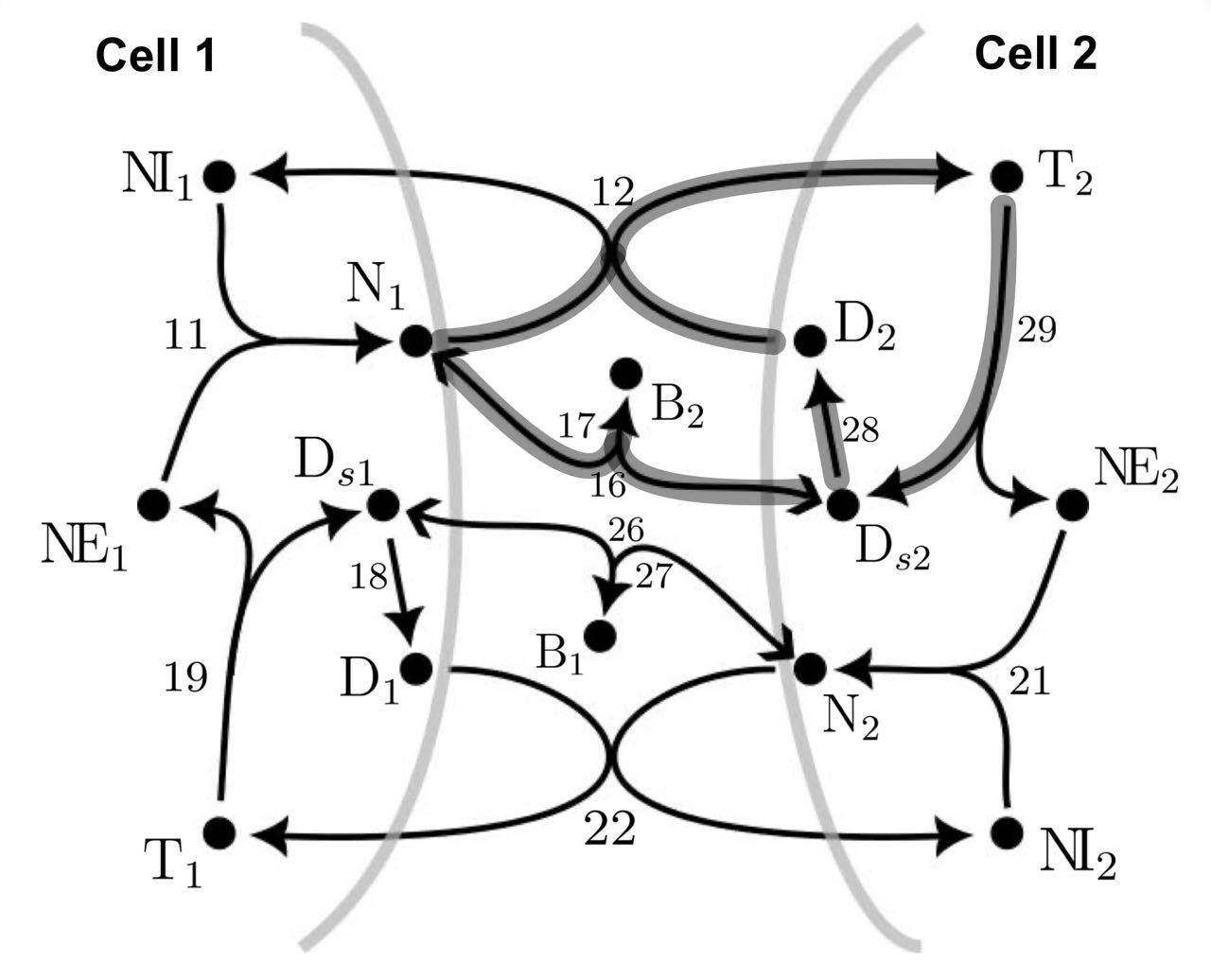} & \includegraphics[width=0.38\textwidth,height=0.38\textwidth,keepaspectratio]{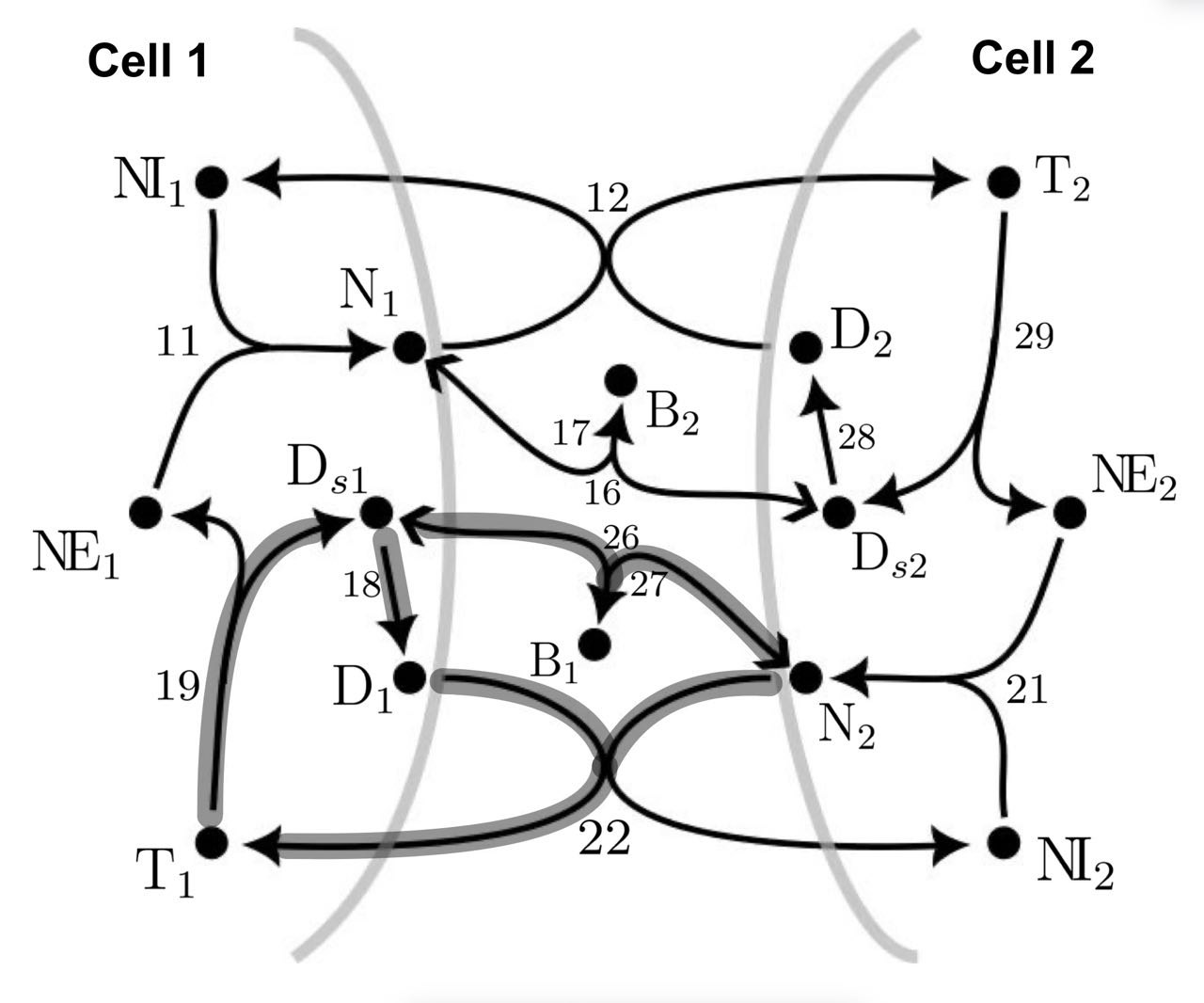} \\
  \hline
\end{tabular}
\caption{Instability motifs in the cis-model \ref{eq:modeli}+\ref{eq:modelii} (rows $1,2,3$; 
motifs \eqref{eq:motif12a}, \eqref{eq:motif12b}, and \eqref{eq:motif12c}) 
and in the ligand-activation model \ref{eq:modeliii} (row $4$; 
motif \eqref{eq:motif3}). The cis-reactions $14-15$, and $24-25$, 
are marked by 
dashed arrows. The left and right columns display the symmetric variants. 
The motifs in rows $1,2$ share the same topological structure, 
differing only by the use of {$\ce{N\!I_j}$} in the first and 
{$\ce{T_j}$} 
in the second
{motif}. All motifs contain an even number of reactant–reactant 
interactions 
involving one {Notch $\ce{N}$ species and one Delta $\ce{D}$ 
(or $\ce{D_s}$) 
species} - four in rows $1,2$, and two in rows $3,4$. 
Since these motifs do include reactant–reactant 
interactions,
they are classified as non-autocatalytic, as their associated CS-matrices are 
not Metzler (see proof in Sec.~\ref{sec:proofs}). {See also Step 3 on page \pageref{step:3}.}}
\label{fig:figure2}
\end{figure}
\paragraph{The central model: \ref{eq:modeli}.}
As customary in chemistry, we use the notation $[\ce{X}]$ to denote the concentration of a species $\ce{X}$. The ODEs 
modeling the time-evolution of the concentrations of the central model are:
\begin{equation}\label{eq:system11}
\begin{cases}
\dot{[\ce{N\!E}_1}]=-r_{11}([\ce{N\!I_1}], [\ce{N\!E_1}]) + r_{13}([\ce{T}_1])\\
\dot{[\ce{N\!I}_1}]=r_{12}([\ce{N_1}],[\ce{D}_{2}])-r_{11}([\ce{N\!I_1}],[\ce{N\!E_1}])\\
[\dot{\ce{N}}_1]=-r_{12}(\ce{[N_1]},\ce{[D_2]})+r_{11}(\ce{[N\!I_1]},\ce{[N\!E_1]})\\
[\dot{\ce{D}}_{1}]=r_{13}([\ce{T_1}])-r_{22}(\ce{[N_2]},\ce{[D_1]})\\
[\dot{\ce{T}}_1]=-r_{13}([\ce{T_1}])+r_{22}([\ce{N_2}],\ce{[D_1]})\\
[\dot{\ce{N\!E}}_2]=-r_{21}(\ce{[N\!I_2]},\ce{[N\!E_2]})+r_{23}(\ce{[T_2]})\\
[\dot{\ce{N\!I}}_2]=r_{22}(\ce{[N_2]},\ce{[D_1]})-r_{21}(\ce{[N\!I_2]},\ce{[N\!E_2]})\\
[\dot{\ce{N}}_2]=-r_{22}(\ce{[N_2]},\ce{[D_1]})+r_{21}(\ce{[N\!I_2]},\ce{[N\!E_2]})\\
[\dot{\ce{D}}_{2}]=-r_{12}(\ce{[N_1]},\ce{[D_2]})+r_{23}(\ce{[T_2]})\\
[\dot{\ce{T}}_2]=r_{12}(\ce{[N_1]},\ce{[D_2]})-r_{23}(\ce{[T_2]})\\
\end{cases}
\end{equation}
{\emph{Kinetic symmetry} requires 
\begin{equation}\label{eq:constraints1}
r_{1k}\equiv r_{2k} \quad \text{for all $k=1,2,3$.}
\end{equation}
However, our analysis of (\ref{eq:system11})} shows that this} system does not support the capacity for 
differentiation irrespective of enforcing the 
constraints \eqref{eq:constraints1} {or not}:
\begin{theorem}\label{thm:11}
The ODE system \eqref{eq:system11} admits only one single steady-state for 
all choices of monotone chemical functions $\mathbf{r}$, {which is locally asymptotically stable}. The 
associated network \emph{\ref{eq:modeli}} is non-autocatalytic. 
In particular, no unstable-positive feedback exists and \emph{\ref{eq:modeli}} has no capacity for differentiation.
\end{theorem}

\paragraph{The cis-model: \ref{eq:modeli}+\ref{eq:modelii}.} To highlight 
the cis-interactions added to the central model~\ref{eq:modeli}, we mark 
the corresponding terms in bold. The ODE system reads:
\begin{equation}\label{eq:system12}
\begin{cases}
\dot{[\ce{N\!E}_1}]=-r_{11}([\ce{N\!I_1}], [\ce{N\!E_1}]) + r_{13}([\ce{T}_1])\\
\dot{[\ce{N\!I}_1}]=r_{12}([\ce{N_1}],[\ce{D}_{2}])-r_{11}([\ce{N\!I_1}],[\ce{N\!E_1}])\\
[\dot{\ce{N}}_1]=-r_{12}(\ce{[N_1]},\ce{[D_2]})+r_{11}(\ce{[N\!I_1]},\ce{[N\!E_1]})\boldsymbol{+r_{15}([\ce{C_1}])-r_{14}([\ce{N_1}],[\ce{D}_{1}])}\\
[\dot{\ce{D}}_{1}]=r_{13}([\ce{T_1}])-r_{22}(\ce{[N_2]},\ce{[D_1]})\boldsymbol{+r_{15}([\ce{C_1}])-r_{14}([\ce{N_1}],[\ce{D}_{1}])}\\
[\dot{\ce{T}}_1]=-r_{13}([\ce{T_1}])+r_{22}([\ce{N_2}],\ce{[D_1]})\\
\boldsymbol{[\dot{\ce{C}}_1]=-r_{15}([\ce{C_1}])+r_{14}([\ce{N_1}],[\ce{D}_{1}])}\\
[\dot{\ce{N\!E}}_2]=-r_{21}(\ce{[N\!I_2]},\ce{[N\!E_2]})+r_{23}(\ce{[T_2]})\\
[\dot{\ce{N\!I}}_2]=r_{22}(\ce{[N_2]},\ce{[D_1]})-r_{21}(\ce{[N\!I_2]},\ce{[N\!E_2]})\\
[\dot{\ce{N}}_2]=-r_{22}(\ce{[N_2]},\ce{[D_1]})+r_{21}(\ce{[N\!I_2]},\ce{[N\!E_2]})\boldsymbol{+r_{25}([\ce{C_2}])-r_{24}([\ce{N_2}],[\ce{D}_{2}])}\\
[\dot{\ce{D}}_{2}]=-r_{12}(\ce{[N_1]},\ce{[D_2]})+r_{23}(\ce{[T_2]})\boldsymbol{+r_{25}([\ce{C_2}])-r_{24}([\ce{N_2}],[\ce{D}_{2}])}\\
[\dot{\ce{T}}_2]=r_{12}(\ce{[N_1]},\ce{[D_2]})-r_{23}(\ce{[T_2]})\\
\boldsymbol{[\dot{\ce{C}}_2]=-r_{25}([\ce{C_2}])+r_{24}([\ce{N_2}],[\ce{D}_{2}])}
\end{cases}
\end{equation}
In this case, kinetic symmetry holds when the reaction rates satisfy:
\begin{equation}\label{eq:constraints2}
{r_{1k}\equiv r_{2k} \quad \text{for all $k=1,2,3,4,5.$}}
\end{equation}
%For this model, we have the following result:
\begin{theorem}\label{thm:12}
Consider the ODE system \eqref{eq:system12} with the constraints 
\eqref{eq:constraints2}. The system has the capacity for zero-eigenvalue 
bifurcations and thus for differentiation. Moreover, up to symmetry, there are exactly three {instability} motifs, {associated to unstable-positive feedback}, which generate the necessary instability. {These} are:
%\begin{enumerate}
%\item 
\begin{equation}\label{eq:motif12a}
\begin{cases}
\ce{N\!I_j} + \;... \quad&\overset{j1}{\ce{->}} \quad \ce{N_j} \\
\ce{N_j}+\ce{D_k} \quad &\overset{j2}{\ce{->}} \quad \ce{N\!I_j} \\  
\ce{N_j}+\ce{D_j} \quad &\overset{j4}{\ce{->}} \quad...\\
\ce{N_k}+\ce{D_j} \quad &\overset{k2}{\ce{->}} \quad ...\\
\ce{N_k}+\ce{D_k} \quad &\overset{k4}{\ce{->}} \quad...\\
\end{cases}\quad\quad\text{for $(j,k)=(1,2)\text{ and }(j,k)=(2,1),$}
\end{equation}
%\item 
\begin{equation}\label{eq:motif12b}
\begin{cases}
\ce{N_j}+\ce{D_k} \quad &\overset{j2}{\ce{->}} \quad ....\\
\ce{T_j} \quad &\overset{j3}{\ce{->}} \quad \ce{D_j}\\
\ce{N_j}+\ce{D_j} \quad &\overset{j4}{\ce{->}} \quad...\\
\ce{N_k}+\ce{D_j} \quad &\overset{k2}{\ce{->}} \quad \ce{T_j}\\
\ce{N_k}+\ce{D_k} \quad &\overset{k4}{\ce{->}} \quad...\\
\end{cases}\quad\quad\text{for $(j,k)=(1,2)\text{ and }(j,k)=(2,1),$}
\end{equation}
%\item 
\begin{equation}\label{eq:motif12c}
\begin{cases}
\ce{N\!I_j} + \;... \quad&\overset{j1}{\ce{->}} \quad \ce{N_j} \\
\ce{N_j} \quad &\overset{j2}{\ce{->}} \quad \ce{N\!I_j} + \ce{T_k}\\
\ce{N_j}+\ce{D_j} \quad &\overset{j4}{\ce{->}} \quad...\\
\ce{N\!E_k} + \;... \quad&\overset{k1}{\ce{->}} \quad \ce{N_k} \\
\ce{N_k}+\ce{D_j} \quad &\overset{k2}{\ce{->}} \quad ...\\
\ce{T_k} \quad &\overset{k3}{\ce{->}} \quad \ce{N\!E_k}\\
\end{cases}\quad\quad\text{for $(j,k)=(1,2)\text{ and }(j,k)=(2,1).$}
\end{equation}
%\end{enumerate}
{All three} motifs identify non-autocatalytic positive feedback and  the associated network \emph{\ref{eq:modeli}$+$\ref{eq:modelii}} is in particular non-autocatalytic.
\end{theorem}

Remarkably, all three instability motifs in the cis model~\ref{eq:modeli}+\ref{eq:modelii} include either both central Notch--Delta interactions (reactions 12 and 22), as in \eqref{eq:motif12a} and \eqref{eq:motif12b}, or at least one of them, as in \eqref{eq:motif12c}. Similarly, each motif includes either both cis-reactions (reactions 14 and 24), as in \eqref{eq:motif12a} and \eqref{eq:motif12b}, or at least one of them, as in \eqref{eq:motif12c}. {Note, that the motifs} \eqref{eq:motif12a} and \eqref{eq:motif12b} share an identical network structure, meaning that, if labels are ignored, the two motifs are indistinguishable, as shown in Fig.~\ref{fig:figure2}, first and second row. More explicitly, disregarding the reaction labels, the simple relabeling of species
\[
(\ce{N\!I}_j, \ce{N}_j, \ce{D}_k, \ce{D}_j, \ce{N}_k) \longrightarrow (\ce{T}_j, \ce{D}_j, \ce{N}_k, \ce{N}_j, \ce{D}_k)
\]
transforms motif \eqref{eq:motif12a} into motif \eqref{eq:motif12b}: the role of $\ce{N\!I}_j$ in the former is taken over by $\ce{T}_j$ in the latter, with the roles of $\ce{D}$ and $\ce{N}$ also switched.
\begin{remark}
To handle the complexity of identifying all instability motifs, 
we relied on the Python module \emph{\textsc{BiRNe}} \emph{\cite{Golnik25}}, see Sec.~\ref{sec:proofs}. 
Incidentally, \emph{\textsc{BiRNe}} also identifies a class of subnetworks that give rise to oscillatory behavior 
in a parameter-rich kinetics setting, i.e., it identifies the existence of parameter choices for which periodic solutions exist. The underlying theoretical framework is based on global Hopf bifurcation \emph{\cite{blokOsci}}. 
Scanning \emph{\ref{eq:modeli}+\ref{eq:modelii}} via \emph{\textsc{BiRNe}} indeed certified the model's capacity for periodic oscillations. 
Since this is not our focus here, we just mention this aspect and leave a more comprehensive discussion 
for future investigations.
\end{remark}

\paragraph{The ligand activation hypothesis: \ref{eq:modeliii}.} 

The associated ODE system reads:

\begin{equation}\label{eq:system13}
\begin{cases}
[\dot{\ce{N\!E}}_1]=-r_{11}([\ce{N\!I_1]}, [\ce{N\!E_1}]) + r_{19}([\ce{T_1}])\\
[\dot{\ce{N\!I}}_1]=r_{12}([\ce{N_1}],[\ce{D}_{2}])-r_{11}([\ce{N\!I}_1],[\ce{N\!E_1}])\\
[\dot{\ce{N}}_1]=-r_{12}([\ce{N_1}],[\ce{D}_{2}])+r_{11}([\ce{N\!I_1}],[\ce{N\!E_1}])-r_{16}([\ce{D_{s2}}],[\ce{N_1}])+r_{17}([\ce{B_2}])\\
[\dot{\ce{D}}_{s1}]=r_{19}([\ce{T_1}])-r_{26}([\ce{D_{s1}}],[\ce{N_2}])+r_{27}([\ce{B_1}])-r_{18}([\ce{D_{s1}}])\\
[\dot{\ce{D}}_{1}]=r_{18}([\ce{D_{s1}}])-r_{22}([\ce{N_2}],[\ce{D}_{1}])\\
[\dot{\ce{T}}_1]=-r_{19}([\ce{T_1}])+r_{22}([\ce{N_2}],[\ce{D}_{1}])\\
[\dot{\ce{B}}_1]=r_{26}([\ce{D_{s1}}],[\ce{N_2}])-r_{27}([\ce{B_1}])\\

[\dot{\ce{N\!E}}_2]=-r_{21}([\ce{N\!I_2}],[\ce{N\!E_2}])+r_{29}([\ce{T_2}])\\
[\dot{\ce{N\!I}}_2]=r_{22}([\ce{N_2}],[\ce{D}_{1}])-r_{21}([\ce{N\!I_2}],[\ce{N\!E_2}])\\
[\dot{\ce{N}}_2]=-r_{26}([\ce{D_{s1}}],[\ce{N_2}])+r_{27}([\ce{B_1}])-r_{22}([\ce{N_2}],[\ce{D}_{1}])+r_{21}([\ce{N\!I_2}],[\ce{N\!E_2}])\\
[\dot{\ce{D}}_{s2}]=r_{29}([\ce{T_2}])-r_{16}([\ce{D_{s2}}],[\ce{N_1}])+r_{17}([\ce{B_2}])-r_{28}([\ce{D_{s2}}])\\
[\dot{\ce{D}}_{2}]=-r_{12}([\ce{N_1}],[\ce{D}_{2}])+r_{28}([\ce{D_{s2}}])\\
[\dot{\ce{T}}_2]=r_{12}([\ce{N_1}],[\ce{D}_{2}])-r_{29}([\ce{T_2}])\\
[\dot{\ce{B}}_2]=r_{16}([\ce{D_{s2}}],[\ce{N_1}])-r_{17}([\ce{B_2}])\\
\end{cases}
\end{equation}

In this case, kinetic symmetry holds when the reaction rates satisfy:
\begin{equation}\label{eq:constraints3}
r_{1k}\equiv r_{2k} \quad \text{for all $k=1,2,6,7,8,9$}
\end{equation}

For this model, we have the following result:
\begin{theorem}\label{thm:13}
Consider the ODE system \eqref{eq:system13} with the constraints \eqref{eq:constraints3}. The system has the capacity for zero-eigenvalue bifurcations and thus for differentiation. 
Moreover, up to symmetry, there is only one {instability} motif, 
{associated to unstable-positive feedback}, which generates the necessary instability:
\begin{equation}\label{eq:motif3}
    \begin{cases}
     \ce{N_j} + \ce{D_{k}} \quad &\overset{j2}{{\ce{->}}} \quad \ce{T_k} + \:...\\
     \ce{N_j} + \ce{D}_{sk} \quad &\overset{j6}{{\ce{->}}}\quad ...\\
 \ce{D}_{sk} \quad &\overset{k8}{{\ce{->}}} \quad \ce{D}_{k}\\
\ce{T_k} \quad &\overset{k3}{\ce{->}} \quad \ce{D}_{ks} + \;...\\
\end{cases}\quad\quad\text{for $(j,k)=(1,2)\text{ and }(j,k)=(2,1)$}
\end{equation}

The motif identifies non-autocatalytic positive feedback and the associated network \emph{\ref{eq:modeliii}} is in particular non-autocatalytic.
\end{theorem}

\begin{remark}\label{rmk:nucleus2} The unique instability motif \eqref{eq:motif3} of the ligand-activation model \emph{\ref{eq:modeliii}} includes the blocked reaction $j6$ (16 or 26 up to symmetry). These reactions are therefore necessary to introduce instability into the system. In contrast, the extension \emph{\ref{eq:modeli'}} of the central model \emph{\ref{eq:modeli}}, which explicitly models the translocation of the intracellular domain $\ce{N\!I_j}$ to the nucleus as $\ce{N\!I}^{\text{nucleus}}_j$, lacks this type of interaction and consequently has no capacity for instability and differentiation. We refer again to \hyperref[pt:translocation]{p.~\pageref*{pt:translocation}} for a rigorous mathematical argument.  
\end{remark}

\section{Minimal Mechanisms}\label{sec:minimal}

In the previous sections {we have set up and discussed 
mathematical models} related to biological information. Now we come from the 
other side, arguing purely mathematically, and try to answer
the following questions:

\emph{\centerline{
What are the essential minimal interactions required between two cells, 
}  
\centerline{
mathematically speaking, in order to allow for symmetry breaking as described 
above?
}
\centerline{What can we learn for biology from this?}}

From our previous analysis, it appears that the `central'
reaction involving the interaction of Notch $\ce{N}_j$ and Delta $\ce{D}_k, \; j\neq k$, and vice versa,  is still an essential part of the 
Notch pathway. This is an asymmetric interaction (Notch with Delta) with the additional splitting of Notch into two parts, the intracellular and the extracellular domain.

%Looking mathematically at this asymmetric reaction, however, 
%we do not even have to distinguish between the three different entities 
%involved. We can just work with one type for all three entities 
%while still retaining the asymmetry of the interaction.  
Now let's only look at this splitting of Notch in a mathematical minimal way, without focussing on other differences between Notch and Delta, i.e., consider two entities of cell 1 that interact with one entity of cell 2. Let this entity
be called $\Lambda$ and 
as before, let the subscripts $1, 2$ refer to cell 1 and cell 2, 
respectively. Then we have
\begin{equation}\label{minmodelsynthesis}\tag{\textbf{MI}}
  2\Lambda_1 + \Lambda_2 \quad \overset{\ce{1}}{\underset{\ce{2}}{\rightleftharpoons}} \quad \Lambda_1 + 2\Lambda_2 \; .
\end{equation}
Interestingly, this minimal model 
\ref{minmodelsynthesis} 
has the capacity for differentiation. Further, it is
\emph{autocatalytic}, in contrast to 
\ref{eq:modeli}, \ref{eq:modeli}+\ref{eq:modelii}, \ref{eq:modeliii}, 
which are non-autocatalytic. In fact, reaction $1$ in \ref{minmodelsynthesis} is autocatalytic in $\Lambda_2$, since one molecule of 
$\Lambda_2$ yields two molecules of $\Lambda_2$. 
The same holds for reaction $2$ and $\Lambda_1$. %The existence of such a simple mechanism for differentiation may point towards a potentially relevant evolutionary feature, which we will further address in Sec.~\ref{sec:autovsnonauto}. 
The ODEs associated to \ref{minmodelsynthesis} are:
\begin{equation}\label{eq:system1dfull}
\begin{cases}
    [\dot{\Lambda}_1] = -r_1([\Lambda_1],[\Lambda_2]) + r_2([\Lambda_2],[\Lambda_1])=-r([\Lambda_1],[\Lambda_2]) + r([\Lambda_2],[\Lambda_1])\\
    [\dot{\Lambda}_2] = \textcolor{white}{+} r_1([\Lambda_1],[\Lambda_2]) - r_2([\Lambda_2],[\Lambda_1])
=\textcolor{white}{+} r([\Lambda_1],[\Lambda_2]) - r([\Lambda_2],[\Lambda_1])\end{cases},
\end{equation}
at kinetic symmetry, i.e. $r_1(x,y) = r_2(x,y) =: r(x,y)$. 
%the ODE system reduces to
%\begin{equation}\label{eq:system1dfull}
%\begin{cases}
%    [\dot{\Lambda}_1] = \\
%    [\dot{\Lambda}_2] = r([\Lambda_1],[\Lambda_2]) - r([\Lambda_2],[\Lambda_1])
%\end{cases}
%\end{equation}
We have $[\dot{\Lambda}_1] + [\dot{\Lambda}_2] = 0$, and 
consequently the conservation law
$[\Lambda_1] + [\Lambda_2] = K.$
The same observation can also be obtained by considering the stoichiometric matrix 
\begin{equation}
    S = \begin{pmatrix}
        -1 & 1\\
        1 & -1
    \end{pmatrix} \; .
\end{equation}
Then $(1,1)$ is a left-kernel vector, corresponding to the conservation law. By fixing $K$ and writing 
$
 [\Lambda_2] = K - [\Lambda_1]$,
system \eqref{eq:system1dfull} reduces to one ODE:
\begin{equation}\label{eq:system1d1d}
[\dot{\Lambda}_1] = H(\Lambda_1) := -r([\Lambda_1], K - [\Lambda_1]) + r(K - [\Lambda_1], [\Lambda_1]).
\end{equation}
{Let} $\partial_1 r$ and $\partial_2 r$ {denote} 
the derivatives of $r$ w.r.t. {its} first and the second component, {then}
\begin{equation}\label{eq:derivminmodel}
\begin{split}
    H'([\Lambda_1])=&-\partial_1 r([\Lambda_1], K - [\Lambda_1])+\partial_2 r([\Lambda_1], K- [\Lambda_1]) \\&-\partial_1 r(K-[\Lambda_1], [\Lambda_1])+\partial_2 r(K-[\Lambda_1], [\Lambda_1])
\end{split} .
\end{equation}
Due to \eqref{eq:system1d1d} {and since Def.~\ref{Monotone chemical functions} of monotone chemical functions implies that $r(0,\Lambda_1)=r(\Lambda_1,0)=0$}, there are 
always at least three nontrivial steady-states:
\begin{equation}
E_{B1} := \{[\bar{\Lambda}_1] = 0\}; \quad
E_{B2} := \{[\bar{\Lambda}_1] = K\}; \quad
{E_k} := \{[\bar{\Lambda}_1] = K/2\}.
\end{equation}
The existence of connecting 
orbits between these steady-states, i.e., the existence of further 
steady-states, depends on the chosen {nonlinearity $r$}. 

We analyze {our ODE} \eqref{eq:system1d1d} with decreasing 
levels of generality. Since any orbit is bounded and 
we are looking at one ODE, if the homogeneous steady-state 
{$E_k$} is unstable, then any diverging orbit from 
{$E_k$} necessarily converges to another inhomogeneous 
steady-state $E_i$, with $E_i > {E_k}$ if the initial 
condition $[\Lambda_1](0) > E_k$, and $E_i < E_k$
otherwise. 
In particular, multistability with at least two inhomogeneous 
steady-states follows. Without further assumptions, the condition for 
linear instability of $E_k$ can be computed via 
\begin{equation}
   H'\bigg( \frac{K}{2}\bigg) = 2\bigg(  \partial_2 r \bigg( \frac{K}{2},\frac{K}{2}\bigg)-\partial_1 r\bigg( \frac{K}{2},\frac{K}{2}\bigg)\bigg),
\end{equation}
giving the following bifurcation condition for $H' = 0$ at $E_k$:
\begin{equation}\label{eq:min_bif}
    \partial_2 r \big|_{\Lambda_1=E_k} = \partial_1 r \big|_{\Lambda_1=E_k}
\end{equation}
and consequent instability ($H'>0$) whenever $\partial_2 r \big|_{\Lambda_1=E_k} > \partial_1 r \big|_{\Lambda_1=E_k}$. %The $\mathbb{Z}_2$-symmetry makes  $H$ antisymmetric around $E_h$, and thus $H'' \equiv 0$ at $E_h$. 
In turn, we study the (in)stability of the inhomogeneous steady-states 
$E_{B1}$ and $E_{B2}$ by assuming {that} 
\begin{equation}\label{eq:formofr}
    r(x,y) = {\left[f(x)\right]^2} g(y),
\end{equation}
where both $f$ and $g$ are monotone chemical functions. In particular, see Def.~\ref{Monotone chemical functions}, $f$ and $g$ are nonnegative monotone functions, $f(x)=0$ (resp. $g(y)=0$) if and only if $x=0$ (resp. $y=0$). 
This functional form naturally models a reaction of the form $2\ce{X} + \ce{Y}$, where the stoichiometric coefficient of {$\ce{X}$} 
appears as an exponent. This includes mass-action, Michaelis–Menten, and Hill kinetics. Under this assumption, the inhomogeneous steady-states $E_{B1}$ and $E_{B2}$ are never linearly stable. $H'(0)=H'(K)$, independently of the choice of $r$. Then, assuming \eqref{eq:formofr}, we compute 
\begin{equation}
\begin{split}
H'(0)&=-2f(0)f'(0)g(K)+f(0)^2g'(K)-2f(K)f'(K)g(0)+f(K)^2g'(0)\\
&=f^2(K)g'(0)\ge 0. 
\end{split}
\end{equation}
In particular, if $g'(0)>0$, then  both $E_{B1}$ and $E_{B2}$ are always 
linearly unstable and orbits diverging from the homogeneous 
steady-state {$E_k$} must always 
converge to another steady-state $E_{i}\not\in \{E_{B_1},E_{B_2}\}$. 
Thus, the most generic geometrical picture consists of three 
steady-states 
when {$E_k$ is stable. If $E_k$} loses stability, it does so via a supercritical 
pitchfork bifurcation, producing two additional inhomogeneous steady-states.

Now {consider the}  
example:
\begin{equation}\label{eq:minnonlinearity}
r(x,y) = {\left[f(x)\right]^2} g(y)= \bigg(\frac{x}{1+\beta x}\bigg)^2 y,
\end{equation}
which suits our {purpose} and is 
{chemically consistent with the stoichiometry of reactants} 
$2\ce{X} + \ce{Y}$. Here $g(y)$ is of mass-action form, so $g'(0)=1>0$ and $f(x)$ is of Michaelis--Menten type. For $K=1$, {our ODE} 
%\eqref{eq:system1dfull} 
becomes
\begin{equation}\label{eq:explicitMI}
\dot{[\Lambda]}_1 = -  \bigg(\frac{[\Lambda]_1}{1 + \beta [\Lambda]_1}\bigg)^2(1 - [\Lambda]_1)+\bigg(\frac{1 - [\Lambda]_1}{1 + \beta(1 - [\Lambda]_1)}\bigg)^2[\Lambda]_1  ,
\end{equation}
whose steady-states $E \notin \{E_{B1}, E_{B2}\}$ can be found analytically. Dividing the steady-state equation by $[\Lambda_1](1-[\Lambda_1])$, corresponding to the steady-states {$E_{B1},E_{B2}$}, one obtains
\begin{equation}
{-} \bigg(1+\beta(1-[\Lambda_1])\bigg)^2[\Lambda_1]
{+ } \bigg(1+\beta[\Lambda_1]\bigg)^2(1-[\Lambda_1]) = 0,
\end{equation}
{which %, due to the presence of the steady-state $E_h=1/2$, 
factorizes} as
\begin{equation}\label{eq:pitchforknf}
(2[\Lambda_1]-1)\big(\beta^2[\Lambda_1]^2-\beta^2[\Lambda_1]+1\big):=\mathcal{P}([\Lambda]_1,\beta)=0.
\end{equation}
The ODE {$\dot{\Lambda}=\mathcal{P}(\Lambda,\beta)$} is 
essentially the normal form of a pitchfork bifurcation, see e.g. 
\cite{GuHoBook}. While the first factor always identifies the reference 
steady-state {$E_k$}, the second factor has roots
{
\begin{equation}
E_{i1},E_{i2}
%=\frac{\beta\pm\sqrt{\beta^2-4}}{2\beta}.
= \frac{1}{2} \pm \sqrt{\frac{1}{4} - \frac{1}{\beta^2}}.
\end{equation}
}

Therefore, as the bifurcation parameter $\beta\in\mathbb{R}_{\ge 0}$ varies, the following occurs (see Fig.~\ref{fig:MI} for an illustration):
\begin{enumerate}
    \item For $\beta\in[0,2)$, there are three steady-states. 
$E_{B1}$ and $E_{B2}$ are unstable, and 
{$E_k$} is stable. Any orbit is a heteroclinic connection 
from $E_{Bj}$ to {$E_k$}, $j=1,2$.
    \item At $\beta=2$, corresponding to \eqref{eq:min_bif}, a supercritical pitchfork bifurcation occurs.
    \item For $\beta>2$, {the pitchfork bifurcation generates two additional inhomogeneous stable steady-states $E_{i1},E_{i2} \not\in \{E_{B_1},E_{B_2}\}$.} Any nonstationary orbit converges either to $E_{i1}$ or to $E_{i2}$. When $\beta$ tends to $+\infty$, $E_{i1}$ and $E_{i2}$ converge to $E_{B1}$ and $E_{B2}$, respectively. 
\end{enumerate}
\begin{figure}
    \centering
    \includegraphics[width=0.40\linewidth]{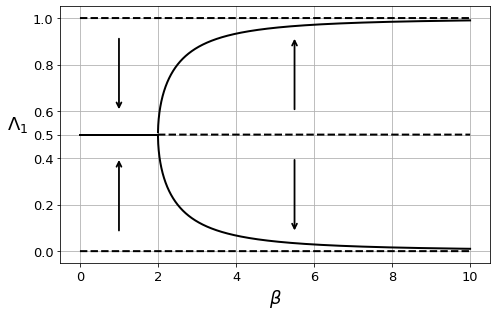}
    \caption{{The steady states of \eqref{eq:explicitMI} are plotted as a function of $\beta$ using Python. Continuous lines signify stable steady states, and dashed lines unstable ones. The steady states $E_{B1}=0$ and $E_{B2}=1$ are unstable, irrespective of the value of $\beta$. In turn, at $\beta = 2$, the reference steady state $E_k=0.5$ loses stability and undergoes a supercritical pitchfork bifurcation. For $\beta>2$ bistability occurs.}}
    \label{fig:MI}
\end{figure}
{In summary, this minimal network has the capacity for differentiation.}

\subsection{Connecting back to biology}\label{q:2}

Why did evolution/biology not follow/invent this simple 
mathematical rule?\\ 
How do the two viewpoints still connect?

Consider the following basic Notch--Delta interaction:
\begin{equation}\label{eq:MII}\tag{\textbf{MII}}
\begin{split}
    \ce{N\!I_1 + N\!E_1 + D_2} \quad &\underset{1}{\longrightarrow} \quad \ce{N\!I_1 + N\!E_2 + D_2}\\
\ce{N\!I_2 + N\!E_2 + D_1} \quad&\underset{2}{\longrightarrow} 
\quad \ce{N\!I_2 + N\!E_1 + D_1} \; . 
\end{split}
\end{equation}
This network is not able to produce differentiation 
(see Prop.~\ref{prop:min245} for mathematical details). {\ref{eq:MII} can be interpreted as a `specialized' version of \ref{minmodelsynthesis}, where the entity $\Lambda$ assumes the distinct roles of $\ce{N\!I}$, $\ce{N\!E}$ and $\ce{D}$. An interesting question arising is: Which essential feature is lost when comparing the unspecialized mathematically minimal network \ref{minmodelsynthesis}, which does allow for differentiation, and this basic specialized biological model \ref{eq:MII}, which does not?}

In the literature it is discussed that the extracellular domain of 
Notch $\ce{N\!E}$ and Delta $\ce{D}$ might have had a close ancestor, 
let's call it {$\Lambda$} again. This translates \ref{eq:MII} into 
\begin{equation}\tag{\textbf{MIII}}\label{eq:Miii}
\begin{split}
    \ce{N\!I_1 + \Lambda_1 + \Lambda_2}\quad &\underset{1}
{\longrightarrow}\quad \ce{N\!I_1 + 2 \Lambda_2}\\
\ce{N\!I_2 + \Lambda_2 + \Lambda_1} \quad&\underset{2}
{\longrightarrow}\quad \ce{N\!I_2 +2\Lambda_1}
\end{split}
\end{equation}

Nicely enough, this network has the capacity for differentiation (see Prop.~\ref{prop:3}). {Similarly} to \ref{minmodelsynthesis} also \ref{eq:Miii} is 
\emph{autocatalytic}, {since} in its reaction 1, 
$\Lambda_2$ catalyses its own production, and vice versa $\Lambda_1$ 
in reaction 2. \begin{remark}\label{rmk:miiib}
If we neglect the participation of $\ce{N\!I}$ in \emph{\ref{eq:Miii}}, we obtain the reduced network
\begin{equation}\tag{\textbf{MIIIb}}\label{eq:Miiib}
\begin{split}
\ce{ \Lambda_1 + \Lambda_2}\quad &\underset{1}
{\longrightarrow}\quad \ce{2 \Lambda_2}\\
\ce{\Lambda_2 + \Lambda_1} \quad&\underset{2}
{\longrightarrow}\quad \ce{2\Lambda_1}.
\end{split}
\end{equation}
Since the stoichiometric matrix $S$  as well as the reactivity matrix $R$ are identical in \emph{\ref{eq:Miiib}} and \emph{\ref{minmodelsynthesis}}, the associated ODEs \eqref{eq:system1dfull} are identical as well. For differentiation to occur through \eqref{eq:min_bif}, a certain asymmetry is needed. This asymmetry can be encoded structurally in the stoichiometry of the reactants, as in \emph{\ref{minmodelsynthesis}}, or less generally only at the level of the chosen nonlinearity, as here in \emph{\ref{eq:Miiib}}. 
\end{remark}

For a full mathematical classification, we now also analyze the 
other situations, namely $\ce{N\!I}$ being close to 
{$\ce{N\!E}$, i.e.
\begin{equation}\tag{\textbf{MIV}}\label{eq:miv}
 \begin{split}
    2{\Lambda}_1 + \ce{D}_2\quad &\underset{1}{\longrightarrow} 
\quad{\Lambda}_1 + {\Lambda}_2 + \ce{D}_2\\
2{\Lambda}_2  + \ce{D_1} \quad &\underset{2}{\longrightarrow} 
\quad {\Lambda}_2 + {\Lambda}_1 + \ce{D}_1
\end{split}
\end{equation}
and $\ce{N\!I}$ being close to $\ce{D}$, i.e. 
\begin{equation}\tag{\textbf{MV}}\label{eq:Mv}
\begin{split}
    \ce{\Lambda_1 + N\!E_1 + \Lambda_2} \quad 
&\underset{1}{\longrightarrow} \quad \ce{\Lambda_1 + N\!E_2 + \Lambda_2}\\
\ce{\Lambda_2 + N\!E_2 + \Lambda_1} \quad&\underset{2}{\longrightarrow} 
\quad \ce{\Lambda_2 + N\!E_1 + \Lambda_1}
\end{split}
\end{equation}
}

Both \ref{eq:miv} and \ref{eq:Mv} {do not allow for} 
differentiation (see again Prop.~\ref{prop:min245} for details). 
\\
{ 
Below we summarize the capacity for differentiation (\textbf{Diff.})
in the five minimal models.} 
{
\renewcommand{\arraystretch}{1.3}
\begin{tabular}{||c|c|c|c||}
\hline
\textbf{Model} & \textbf{Specialization} & \textbf{Reaction network} & \textbf{Diff.}\\
\hline\hline
\hline
\textbf{MI} & $(\ce{N\!I}=\ce{N\!E}=\ce{D}):=\Lambda$ &  $2\Lambda_1 + \Lambda_2 \quad \underset{2}{\overset{1}{\rightleftharpoons}} \quad \Lambda_1 + 2\Lambda_2$
& $\checkmark$\\
\hline
\textbf{MII} & $(\ce{N\!I} \neq \ce{N\!E} \neq \ce{D})$  & {\centering $\begin{aligned}
    \ce{N\!I_1 + N\!E_1 + D_2} \quad &\underset{1}{\longrightarrow} \quad \ce{N\!I_1 + N\!E_2 + D_2}\\
\ce{N\!I_2 + N\!E_2 + D_1} \quad&\underset{2}{\longrightarrow} \quad \ce{N\!I_2 + N\!E_1 + D_1}
\end{aligned}$}  &  $\pmb{\times}$\\
\hline
\textbf{MIII} & $(\ce{N\!E}=\ce{D}):=\Lambda\neq \ce{N\!I}$ & {\centering $\begin{aligned}
    \ce{N\!I_1 + \Lambda_1 + \Lambda_2}\quad &\underset{1}
{\longrightarrow}\quad \ce{N\!I_1 + 2 \Lambda_2}\\
\ce{N\!I_2 + \Lambda_2 + \Lambda_1} \quad&\underset{2}
{\longrightarrow}\quad \ce{N\!I_2 +2\Lambda_1}
\end{aligned}$} & $\checkmark$ \\
\hline
\textbf{MIV} & $(\ce{N\!I}=\ce{N\!E}):=\Lambda\neq \ce{D}$ & 
{\centering $\begin{aligned}
    2\Lambda_1 + \ce{D}_2\quad &\underset{1}{\longrightarrow} 
\quad\Lambda_1 + \Lambda_2 + \ce{D}_2\\
2\Lambda_2  + \ce{D_1} \quad &\underset{2}{\longrightarrow} \quad 
\Lambda_2 + \Lambda_1 + \ce{D}_1
\end{aligned}$} &  $\pmb{\times}$ \\
\hline
\textbf{MV} & $(\ce{N\!I}=\ce{D}):=\Lambda\neq \ce{N\!E}$ & 
{\centering $\begin{aligned}
    \ce{\Lambda_1 + N\!E_1 + \Lambda_2} \quad &\underset{1}
{\longrightarrow} \quad \ce{\Lambda_1 + N\!E_2 + \Lambda_2}\\
\ce{\Lambda_2 + N\!E_2 + \Lambda_1} \quad&\underset{2}{\longrightarrow} 
\quad \ce{\Lambda_2 + N\!E_1 + \Lambda_1}
\end{aligned}$}  & $\pmb{\times}$ \\
\hline\hline
\end{tabular}\\

We  already know 
the capacity for differentiation for \ref{minmodelsynthesis}. The
next two propositions cover the other four models. 
Interestingly \ref{eq:Miii} formally mirrors \ref{eq:miv}, with inverse orientation, so that the autocatalytic self-amplification process in \ref{eq:Miii}, which has the capacity for differentiation, results in a non-autocatalytic `self-attenuation' process in \ref{eq:miv}, which does not.  Intuitively, this is analogous to the idea that an iterated process such as $\ce{X} \longrightarrow 2\ce{X}$ leads to self-amplification of $\ce{X}$, whereas  $2\ce{X \longrightarrow X}$ leads to self-attenuation.

\begin{prop}\label{prop:3}
Consider the ODEs associated to \emph{\ref{eq:Miii}},
\begin{equation}
    \begin{cases}
    
    [\dot{\Lambda}_1]=-r_1([\ce{N\!I}_1],[\Lambda_1],[\Lambda_2])+r_2([\ce{N\!I}_{{2}}],[\Lambda_2],[\Lambda_1])\\
    [\dot{\Lambda}_2]=r_1([\ce{N\!I}_1],[\Lambda_1],[\Lambda_2])-r_2([\ce{N\!I}_{{2}}],[\Lambda_2],[\Lambda_1])\\
[\dot{\ce{N\!I}_1}]=[\dot{\ce{N\!I}_2}]=0\\ 
    \end{cases},
\end{equation}
with the kinetic symmetry constraint $r_1\equiv r_2$. The system has the capacity for zero-eigenvalue bifurcations and thus for differentiation. 
\end{prop}

\begin{prop}\label{prop:min245}
Consider the ODEs associated to \emph{\ref{eq:MII}}, 
\begin{equation}\label{eq:odemodelmin2}
\begin{cases}
[\dot{\ce{N\!E}_1}]=-r_1([\ce{N\!I}_1],[\ce{N\!E}_1],[\ce{D}_2])+r_2([\ce{N\!I}_2],[\ce{N\!E}_2],[\ce{D}_1])\\
[\dot{\ce{N\!E}_2}]=r_1([\ce{N\!I}_1],[\ce{N\!E}_1],[\ce{D}_2])-r_2([\ce{N\!I}_2],[\ce{N\!E}_2],[\ce{D}_1])\\
[\dot{\ce{N\!I}_1}]=[\dot{\ce{N\!I}_2}]=[\dot{\ce{D}}_1]=[\dot{\ce{D}}_2]=0\\ 
\end{cases},
\end{equation}
the ODEs associated to \emph{\ref{eq:miv}},
\begin{equation}\label{eq:odemodelmin4}
\begin{cases}
[\dot{\Lambda_1}]=-r_1([\Lambda_1],[\ce{D}_2])+r_2([\Lambda_2],[\ce{D}_1])\\
[\dot{\Lambda_2}]=r_1([\Lambda_1],[\ce{D}_2])-r_2([\Lambda_2],[\ce{D}_1])\\
[\dot{\ce{D}}_1]=[\dot{\ce{D}}_2]=0\\
\end{cases},
\end{equation}
and the ODEs associated to \emph{\ref{eq:Mv}},
\begin{equation}\label{eq:odemodelmin5}
\begin{cases}
[\dot{\ce{N\!E}_1}]=-r_1([\Lambda_1],[\ce{N\!E}_1],[\Lambda_2])+r_2([\Lambda_2],[\ce{N\!E}_2],[\Lambda_1])\\
[\dot{\ce{N\!E}_2}]=r_1([\Lambda_1],[\ce{N\!E}_1],[\Lambda_2])-r_2([\Lambda_2],[\ce{N\!E}_2],[\Lambda_1])\\
[\dot{\Lambda}_1]=[\dot{\Lambda}_2]=0
\end{cases}.
\end{equation}
Systems \eqref{eq:odemodelmin2}, \eqref{eq:odemodelmin4}, and \eqref{eq:odemodelmin5} all admit only one single steady-state for all choices of monotone reaction functions $\mathbf{r}$ and thus do not have the capacity for differentiation. %The reaction networks  \emph{\ref{eq:MII}}, \emph{\ref{eq:miv}}, \emph{\ref{eq:Mv}} all are non-autocatalytic. 
\end{prop}

\subsection{Autocatalytic and non-autocatalytic instabilities}\label{sec:autovsnonauto}

The models \ref{eq:modeli}, \ref{eq:modeli}+\ref{eq:modelii}, and \ref{eq:modeliii}, are non-autocatalytic. In contrast, both minimal models that exhibit the capacity for differentiation, \ref{minmodelsynthesis} and \ref{eq:Miii}, are autocatalytic. 

Why should evolution dismiss a functioning autocatalytic process and favor a larger non-autocatalytic network instead? We can not really answer this question, but would still like to discuss it here.

The central model \ref{eq:modeli}  may be regarded as an elaborated version of \ref{eq:MII}:  Thm.~\ref{thm:11} proves that also \ref{eq:modeli} does not have the capacity for differentiation. By contrast, the cis model \ref{eq:modeli}+\ref{eq:modelii} and the ligand-activation model \ref{eq:modeliii} do exhibit differentiation and Thm.~\ref{thm:12} and Thm.~\ref{thm:13} show that differentiation arises precisely from the additional interactions (the cis and blocked-ligand mechanisms), which induce non-autocatalytic instabilities. Nevertheless, the autocatalytic \ref{minmodelsynthesis} and \ref{eq:Miii} do not need this.

For further insight and comparison, we now investigate minimal non-autocatalytic instability motifs in a symmetric setting. As argued in Sec.~\ref{sec:preliminary}, the existence of a CS-matrix 
representing an unstable-positive feedback is a necessary condition for the capacity for
differentiation. In case of two variables, such an unstable-positive 
feedback is not only necessary but also sufficient, as established by the 
expansion in Lemma \ref{lem:CSexpansion} for the determinant of the Jacobian. Accordingly, a $2\times2$ non-autocatalytic unstable-positive feedback, in symmetric form, must necessarily take the explicit form:
\begin{equation}
S[\pmb{\kappa}]=
\begin{pmatrix}
\eta_1 & \eta_2\\
\eta_2  & \eta_1 
\end{pmatrix}, \quad\quad \text{
with $\eta_1<0$ and}\quad\quad
\operatorname{sign} \det S[\pmb{\kappa}]=(-1)^{2-1}=-1.
\end{equation}
Thus 
$ |\eta_1| < |\eta_2|$. If $\eta_2>0$ the system 
would be autocatalytic, since then the network possesses
a CS-matrix $S[\pmb{\kappa}]$ that is both unstable and Metzler.
%, and therefore autocatalytic. 
The only remaining case is $\eta_2<0$. W.l.o.g. fix one parameter to $-1$.
Then a minimal non-autocatalytic unstable-positive feedback takes one of 
the following two forms
\begin{equation}\label{eq:upfm}
S[\pmb{\kappa}]=
\begin{pmatrix}
-1 & -\eta\\
-\eta & -1
\end{pmatrix},\quad\text{for $\eta>1$}\quad\text{or}\quad
S[\pmb{\kappa}]=
\begin{pmatrix}
-\eta & -1\\
-1 & - \eta
\end{pmatrix},\quad\text{for $\eta<1$}.
\end{equation}
Both cases correspond to the same minimal network with two reactions, 1 and 2 below. Since this minimal model is intentionally designed to isolate the non-autocatalytic instability motif, we have to make an important exception here and do include constant production in order to ensure the existence of a positive steady-state. The resulting model (see also Fig.~\ref{fig:nonaut}, left) is
\begin{equation}\label{eq:minmodelabstr}\tag{\textbf{NonAut-I-$\eta$}}
\begin{aligned}
 \eta \ce{\Lambda_1} + \ce{\Lambda_2} \quad &\overset{1}{\longrightarrow} \quad \dots \quad &\quad \quad  
\eta \ce{\Lambda_2} + \Lambda_1 \quad &\overset{2}{\longrightarrow} \quad \dots\\
... \quad &\overset{{P_{\Lambda_2\;}}}{\longrightarrow} \quad \Lambda_2 \quad &\quad \quad ... \quad &\overset{{P_{\Lambda_1}\;}}{\longrightarrow} \quad \Lambda_1
\end{aligned}
\end{equation}

\begin{prop}\label{prop:nautI} Consider the system of ODEs associated to \emph{\ref{eq:minmodelabstr}}: 
\begin{equation}\label{eq:minmodelabstrode}
\begin{cases}
[\dot{\Lambda}_1]= P_{\Lambda_1}- \eta \;r_1([\Lambda_1],[\Lambda_2])-r_2([\Lambda_2],[\Lambda_1])\\
[\dot{\Lambda}_2]= P_{\Lambda_2}-r_1([\Lambda_1],[\Lambda_2])- \eta\;r_2([\Lambda_2],[\Lambda_1])
\end{cases}
\end{equation}
with kinetic symmetry constraints: 
\begin{equation}\label{eq:minmodelabstrsymm}
r_1([\Lambda_1],[\Lambda_2])\equiv r_2([\Lambda_2],[\Lambda_1])\equiv r(\cdot,\cdot), \quad P_{\Lambda_1} \equiv P_{\Lambda_2}\equiv P.
\end{equation}
For $\eta\neq 1$, system \eqref{eq:minmodelabstrode} has the capacity for zero-eigenvalue  bifurcation and thus differentiation with bifurcation condition:
\begin{equation}
    \partial_1 r = \partial_2 r.
\end{equation}
The instability motif, associated to non-autocatalytic unstable-positive feedback, which generates the necessary instability is:
\begin{equation}\label{nonautiinstmotif}
    \begin{cases}
        \eta \Lambda_1 +\Lambda_2\quad &\overset{1}{\longrightarrow}\dots\\
\eta \Lambda_2 + \Lambda_1 \quad &\overset{2}{\longrightarrow} \quad \dots\\
    \end{cases}
\end{equation}
    
\end{prop}

\begin{remark}
In \emph{\ref{eq:minmodelabstr}} there are no conservation laws. The concentrations of $\Lambda_1$ or $\Lambda_2$ are thus not bounded and may diverge, so that the loss of stability of the homogeneous steady-state $\bar{\Lambda}_1=\bar{\Lambda}_2$ may result in, e.g., ${\Lambda}_1$ converging to 0, while $\Lambda_2$ diverges to infinity.
\end{remark}
\begin{remark} Interestingly, in \emph{\textbf{NonAut-I-2}} reactions 1 and 2  share the same reactant stoichiometry as the reactions in \emph{\ref{minmodelsynthesis}}. However, the autocatalytic \emph{\ref{minmodelsynthesis}} does not need production terms because its combined reactions produce back the same reactants. In contrast, the non-autocatalytic \emph{\ref{eq:minmodelabstr}} does not possesses a sequence of reactions $(j_1,\dots,j_n)$ 
where a species $\Lambda_j$ appears as a 
reactant and ends with $j_n$ where the same $\Lambda_j$ appears as a 
product. In this sense, \emph{\ref{eq:minmodelabstr}} is closer to a simple continuous-stirred tank reactor
\begin{equation}
\begin{split}
    \ce{->} \quad\Lambda_1\quad\ce{->}\\
    \ce{->}\quad \Lambda_2\quad\ce{->}
\end{split}  
\end{equation}
but with a symmetric outflow law given by reactions $1$ and $2$. This 
type of outflow puts $\Lambda_1$ and $\Lambda_2$ in a reactant–reactant relation 
and introduces instability through the non-autocatalytic unstable 
positive feedback \eqref{nonautiinstmotif}. The system can therefore still be viewed  as a continuous-flow reactor with no reactant–product loop, but its dynamics are fundamentally altered by the introduction of such a positive feedback.
\end{remark}

Let us now describe a minimal non-autocatalytic instability without production. If we abstractly consider a mutual transfer 
of biochemical 
information between two cells, a simple 
mechanism of this kind with the capacity for differentiation could then be as follows (see also Fig.~\ref{fig:nonaut}, right):
\begin{equation}\label{eq:modelmin2}\tag{\textbf{NonAut-II-$\eta$}}
\begin{aligned}
\eta\Lambda_1+\Lambda_2  \quad&\overset{1}{\ce{->}} \quad \eta \Lambda_1+\ce{I_2} \quad &\quad \quad \eta \Lambda_2+\Lambda_1  \quad&\overset{
3}{\ce{->}} \quad \eta \Lambda_2+\ce{I_1}\\
     \ce{I_2}\quad &\overset{2}{\ce{->}} \quad \Lambda_2 \quad &\quad \quad   \ce{I_1}\quad &\overset{4}{\ce{->}} \quad \Lambda_1\\
\end{aligned}
\end{equation}
Here $\Lambda_1,\Lambda_2$ denote entities in the two cells that interact with each other, while $\ce{I}_j$ denotes an intermediate entity in cell $j$, arising from $\Lambda_j$ upon its interaction with $\eta$ entities $\Lambda_i$ in cell $i$. This may correspond, for instance, to the translocation of $\Lambda_j$ to the nucleus, to its endocytosis and release, or to a silenced state. The process occurs symmetrically and simultaneously in both cells. \ref{eq:modelmin2}  has the capacity for differentiation, but with a caveat for $\eta=1$, see Remark \ref{rmk:asymmetricreactions}.

%, without including the basic asymmetric interaction with splitting described in \ref{minmodelsynthesis}, but it is larger than \ref{minmodelsynthesis} and needs asymmetric reaction functions, see Remark~\ref{rmk:asymmetricreactions} below.
\begin{prop} \label{prop:nautII}
Consider the ODE system associated to the network \emph{\ref{eq:modelmin2}}:
\begin{equation}\label{eq:minmodel1eq}
\begin{cases}
    [\dot{{\Lambda}}_1]=-r_{3}([\Lambda_2],[\Lambda_1])+r_{4}(\ce{I}_1),\\
[\dot{\ce{I}}_1]=r_{3}([\Lambda_2],[\Lambda_1])-r_{4}(\ce{I}_1)\\
    [\dot{{\Lambda}}_2]=-r_{1}([\Lambda_1],[\Lambda_2])+r_{2}(\ce{I}_2) \\
    [\dot{\ce{I}}_2]=r_{1}([\Lambda_1],[\Lambda_2])-r_{2}(\ce{I}_2)\\
\end{cases}
\end{equation}
with kinetic symmetry constraints:
\begin{equation}\label{eq:minmodel1symm}
  r_{1}([\Lambda_1],[\Lambda_2])\equiv  r_{3}([\Lambda_2],[\Lambda_1])\equiv r_o(\cdot,\cdot), \quad \quad  \text{and}\quad\quad r_{2}(\cdot)=r_{4}(\cdot)=r_e(\cdot).
\end{equation}
Let $\partial_1 r_o$ and $\partial_2 r_o$ denote the derivatives of $r_o$ w.r.t. its first and second arguments, {and $\partial_I r_e$ the derivative of $r_e$ w.r.t. its single argument.} System \eqref{eq:minmodel1eq} has the capacity for zero-eigenvalue bifurcation and consequently differentiation, with bifurcation condition:
\begin{equation}\label{eq:bifcond}
    \partial_2 r_o +\partial_Ir_e= \partial_1 r_o,
\end{equation}
The existence of an unstable steady-state requires 
\begin{equation}\label{eq:bifcondbetter}
   \partial_2 r_o +\partial_Ir_e < \partial_1 r_o.
\end{equation}
The unique associated instability motif, represented by a non-autocatalytic unstable-positive feedback, which generates the necessary instability is:
\begin{equation}\label{eq:instability motifnonautII}
    \begin{cases}
     \eta \Lambda_1 + \Lambda_2  \quad&\overset{1}{\ce{->}} \quad \eta\Lambda_1  \;+\; ...\\
   \eta \Lambda_2 + \Lambda_1  \quad&\overset{3}{\ce{->}} \quad \eta\Lambda_2  \;+\; ...\\
   
  \end{cases}
\end{equation}
\end{prop}
The instability motif \eqref{eq:instability motifnonautII} resembles the instability motifs \eqref{eq:motif12a}, \eqref{eq:motif12b}, and \eqref{eq:motif12c} of the cis model \ref{eq:modeli}+\ref{eq:modelii}, and \eqref{eq:motif3} of the ligand-activation model \ref{eq:modeliii}, in the sense that it contains an even number of reactant–reactant interactions (in this case two), together with reactant–product loops attached to them. See Fig.~\ref{fig:figure2} and Fig.~\ref{fig:nonaut}, right.

\begin{figure}[H]
\centering
\setlength{\tabcolsep}{4pt} % no space between columns
\renewcommand{\arraystretch}{0} % no space between rows
\begin{tabular}{|c|c|}
  \hline
  \includegraphics[width=0.40\textwidth, keepaspectratio]{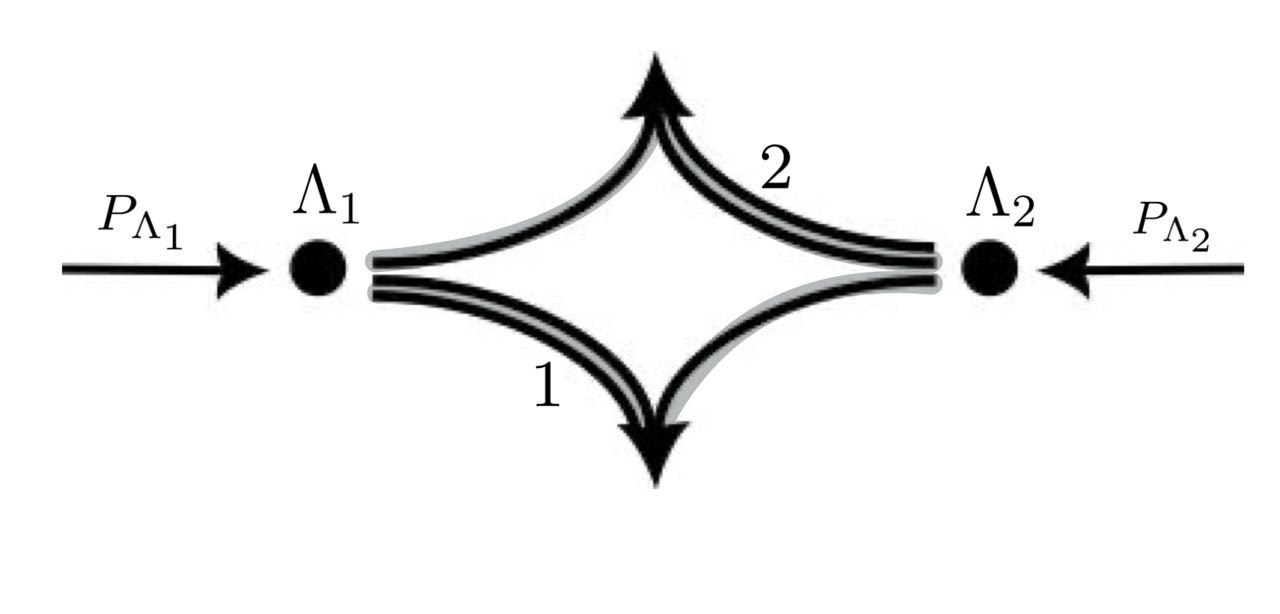} &
\includegraphics[width=0.40\textwidth,keepaspectratio]{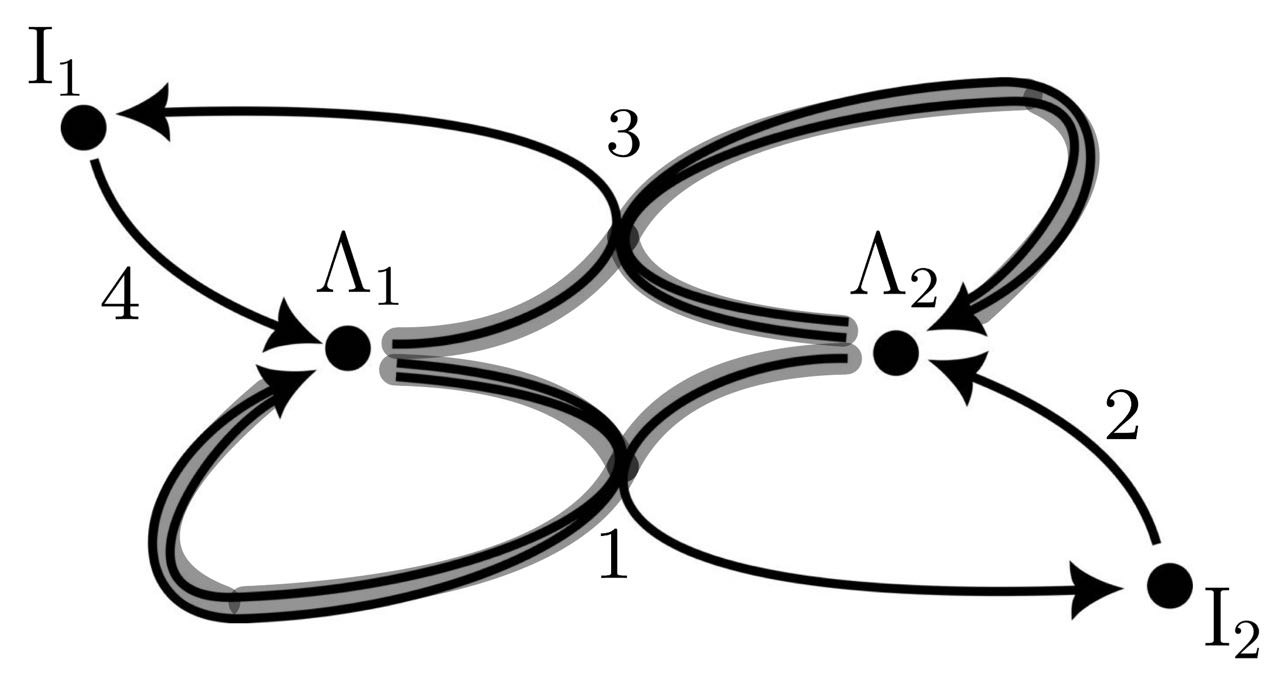}\\
  \hline
\end{tabular}
\caption{Networks \textbf{NonAut-I-2} (left) and \textbf{NonAut-II-2} (right), with instability motifs.}
\label{fig:nonaut}
\end{figure}

\begin{remark}\label{rmk:asymmetricreactions}
In contrast to \emph{\ref{eq:minmodelabstr}}, the model \emph{\ref{eq:modelmin2}} has the capacity for differentiation for all $\eta$. However, for condition \eqref{eq:bifcondbetter} to hold  for \emph{\textbf{NonAut-II-1}} at kinetic symmetry at a homogeneous steady-state, 
it is necessary that the functions $r_1$, $r_3$ are \emph{not} symmetric, i.e., $r(x,y)\neq r(y,x).$ 
This is a crucial restriction, although parameter-rich kinetics such as Michaelis--Menten are generally not symmetric:
\begin{equation}
    r(x,y)=k\dfrac{x}{(1+ax)}\dfrac{y}{(1+by)}, \quad \mbox{for} \quad
a \neq b \; ,
\end{equation}
while parametrically `poor' mass-action kinetics 
$r(x,y)=kxy$ naturally are.  Even if $\Lambda_1$ and $\Lambda_2$ appear as reactants in both reaction 1 and 3 with the same stoichiometry, an asymmetric reaction function means that the interaction $\Lambda_1+\Lambda_2$ is different from that of $\Lambda_2+\Lambda_1$, similarly to \emph{\ref{eq:Miiib}}, see Remark \emph{\ref{rmk:miiib}}. This type of asymmetry needs further explanation. \emph{\textbf{NonAut-II-2}}, though, is non-autocatalytic, does not need production or decay terms, and the asymmetry is encoded in the stoichiometry. %Thus it is closest to our biological models \emph{\ref{eq:modeli}}+\emph{\ref{eq:modelii}} and \emph{\ref{eq:modeliii}}.}
\end{remark}

\begin{figure}[htbp]
\centering
\setlength{\tabcolsep}{6pt}
\begin{tabular}{|c|c|c|}
\hline
\includegraphics[width=0.3\textwidth]{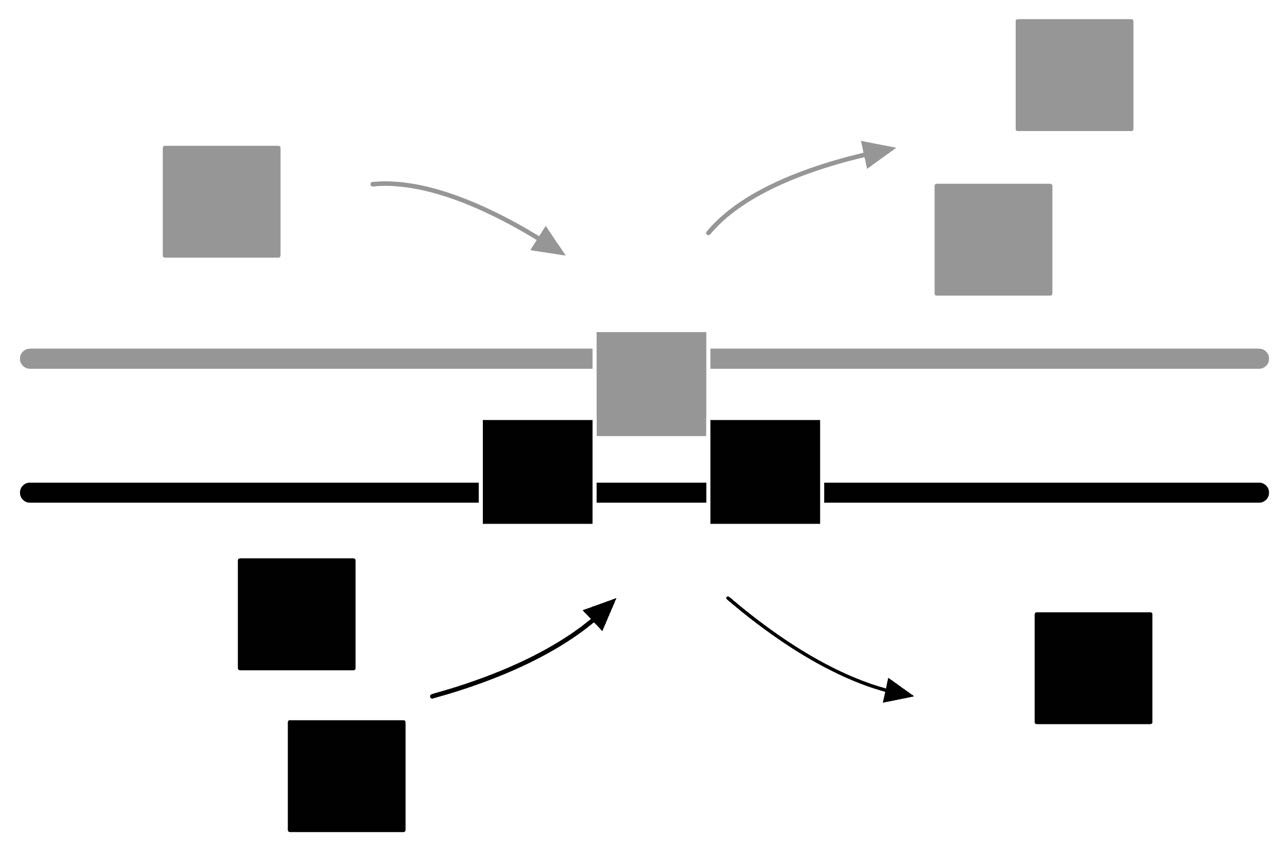} &
\includegraphics[width=0.27\textwidth]{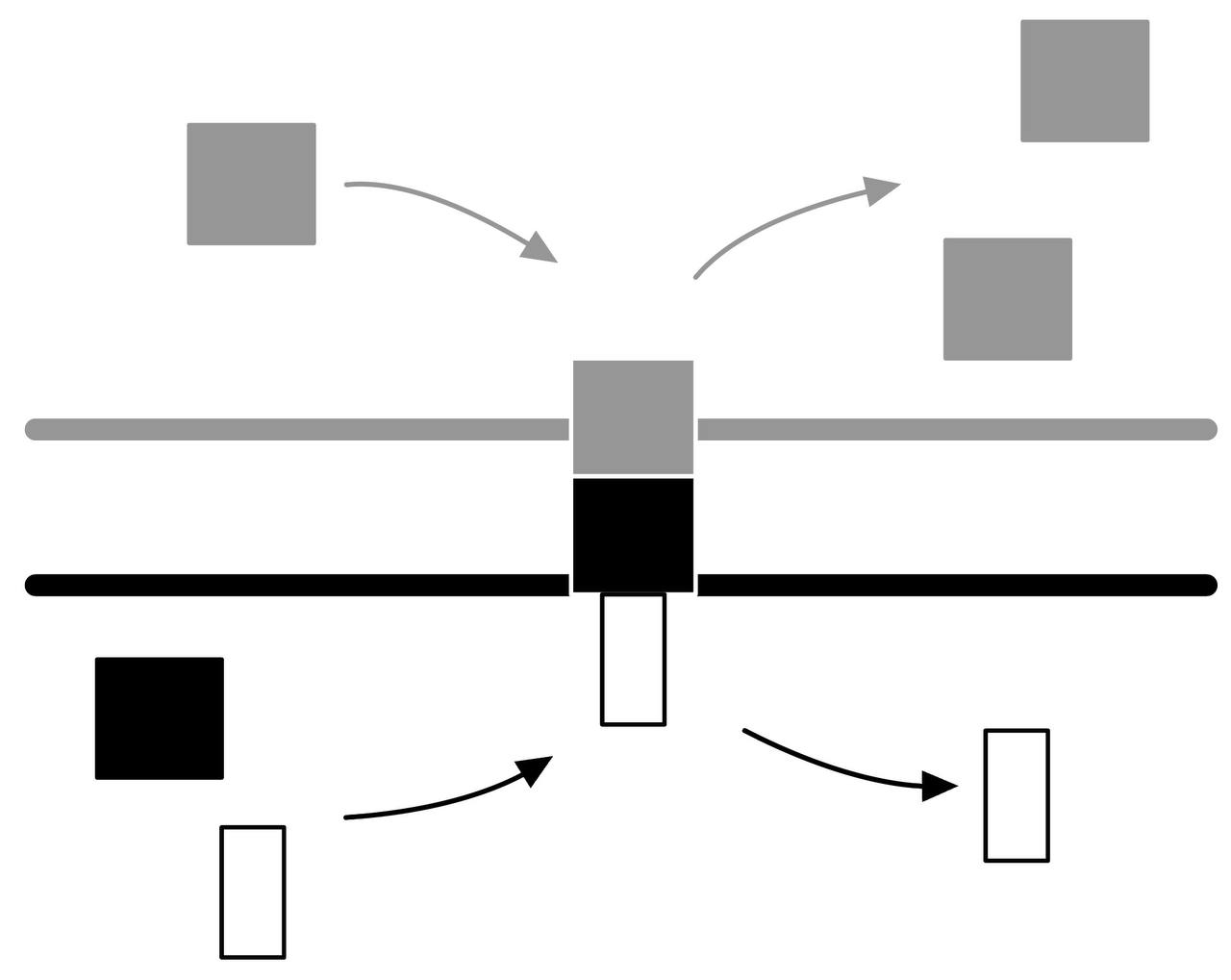}  &  \includegraphics[width=0.3\textwidth]{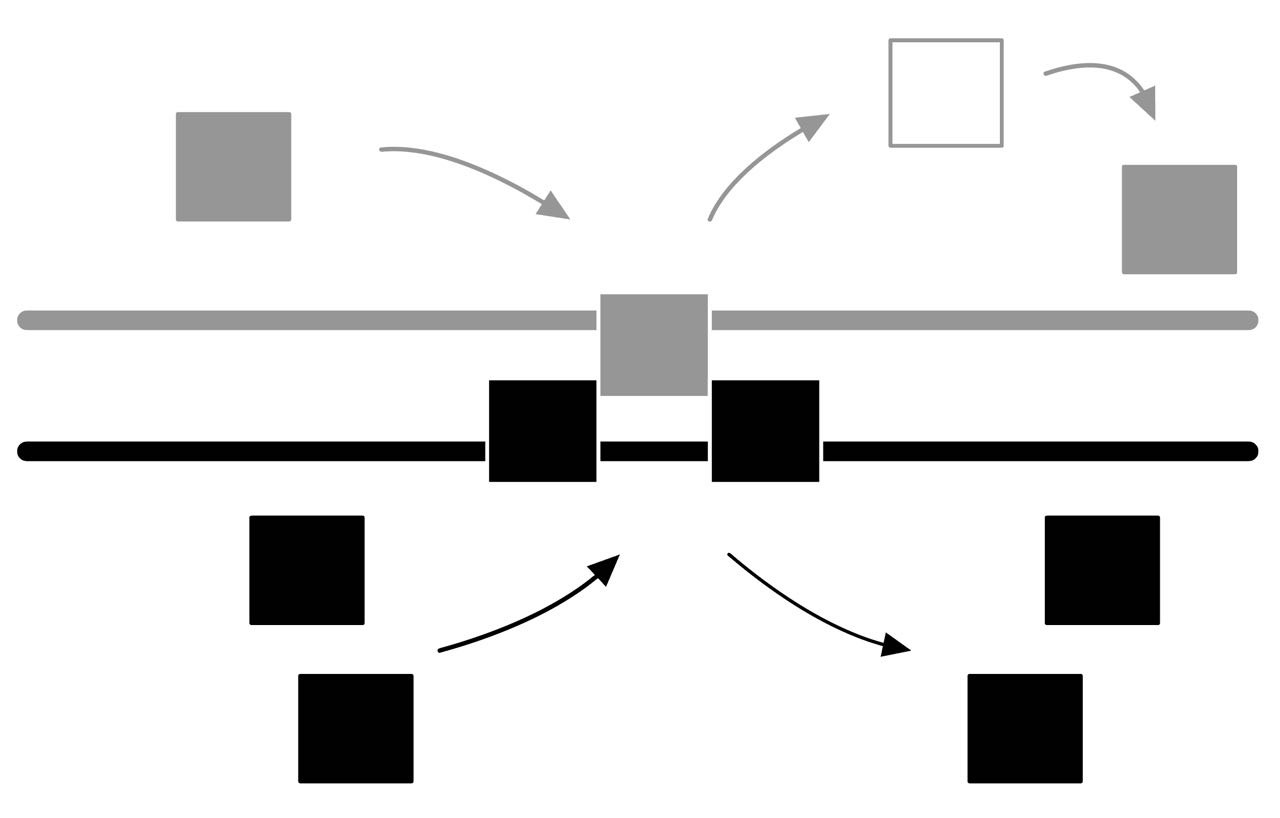}\\
\hline
\end{tabular}
\begin{tabular}{|c|c|}
\hline
\includegraphics[width=0.26\textwidth]{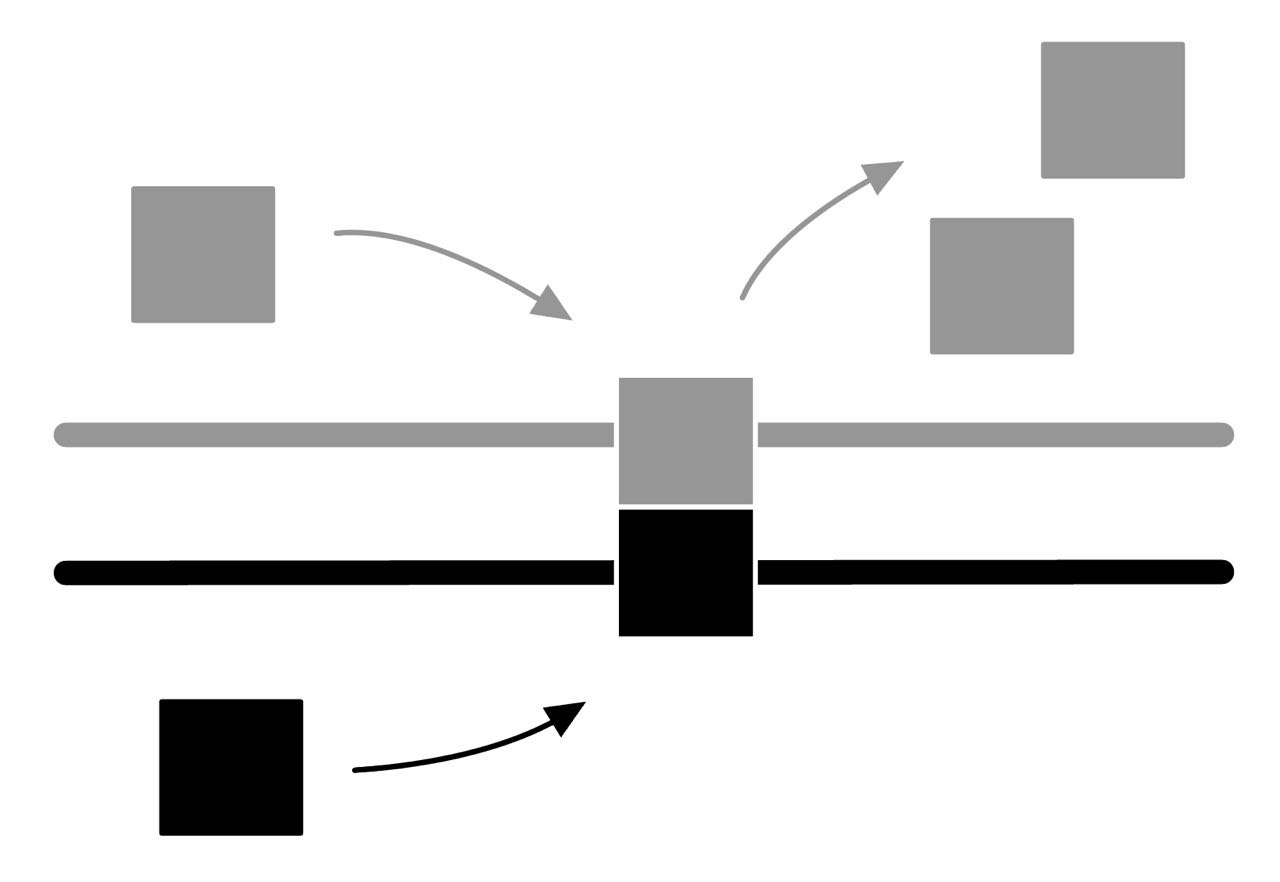} &
\includegraphics[width=0.26\textwidth]{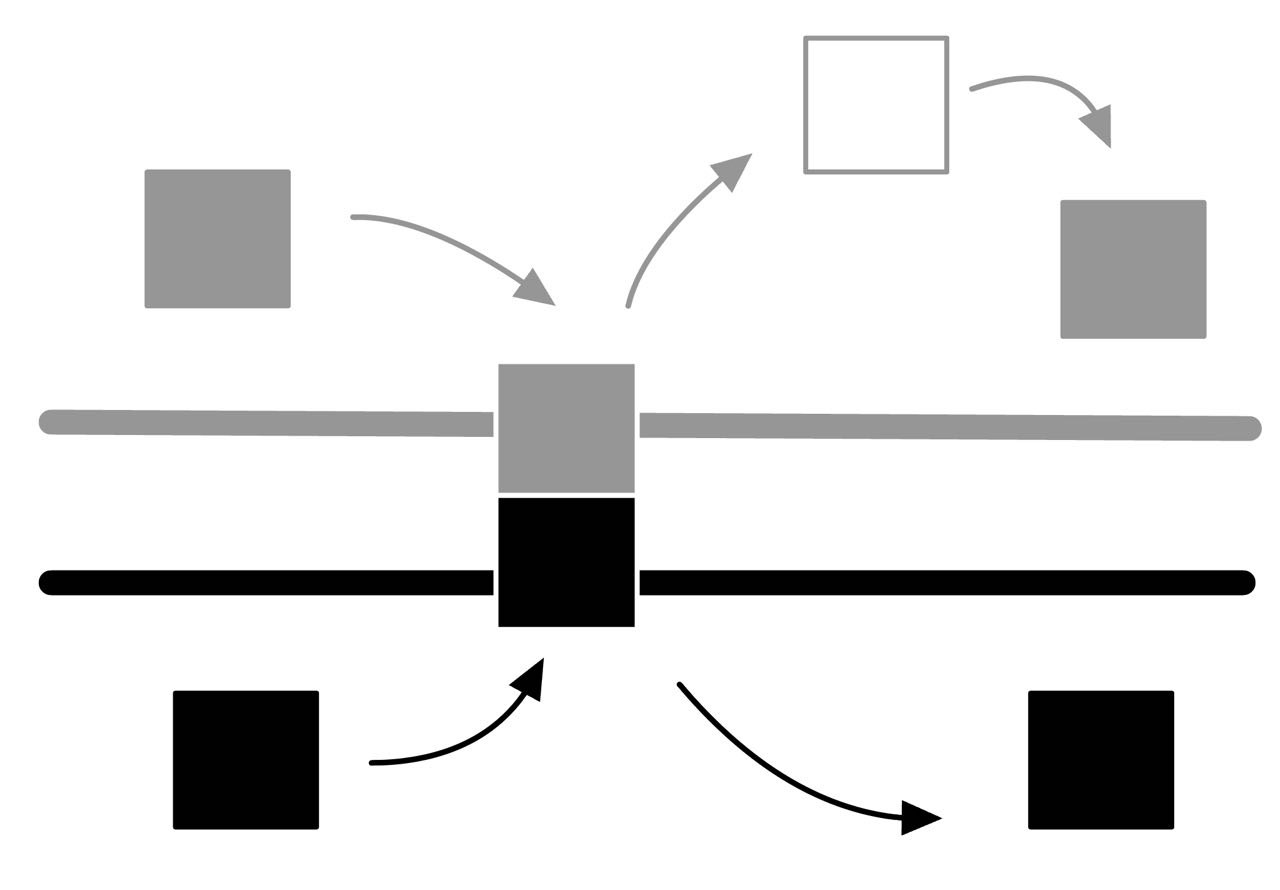} \\
\hline
\end{tabular}
\caption{Top row: models \ref{minmodelsynthesis} (top left), \ref{eq:Miii} (top center), \textbf{NonAut-II-2} (top right). Bottom row: models \ref{eq:Miiib} (bottom left), and \textbf{NonAut-II-1} (bottom right). Black squares denote $\Lambda_1$ and grey squares $\Lambda_2$. The white rectangle denotes $\ce{N\!I_1}$  in \ref{eq:Miiib},  and $\ce{I_2}$ in \ref{eq:modelmin2}. As in Fig.~\ref{fig:figure1}, all interactions are understood to occur symmetrically. All models possess the capacity for differentiation. In the top row, the asymmetric interaction between black and grey squares is structurally explicit; in the bottom row, the asymmetry arises solely from the reaction function modeling the interaction and this needs a further biochemical explanation, see Remarks \ref{rmk:miiib} and \ref{rmk:asymmetricreactions}.
}
\label{fig:biomin}
\end{figure}

%\begin{remark}
%Mathematically analogous variants of \emph{\ref{eq:modelmin2}} include  \begin{equation*}
%\begin{aligned}
%\eta \Lambda_1 +  \Lambda_2 %\quad&\overset{
%1}{\ce{->}} \quad \ce{I_1}+ \Lambda_2 \quad &\quad \quad \eta \Lambda_2 + \Lambda_1 \quad&\overset{3}{\ce{->}} \quad \ce{I_2}+\Lambda_1\\
%    \ce{I_1}\quad &\overset{2}{\ce{->}} \quad \eta \Lambda_1  \quad &\quad \quad  \ce{I_2}\quad &\overset{4}{\ce{->}} \quad \eta \Lambda_2\\
%\end{aligned},
%\end{equation*}
%or, for $\eta=2$,
%\begin{equation*}
%\begin{aligned}
%2 \Lambda_1 +  \Lambda_2 \quad&\overset{
%1}{\ce{->}} \quad \ce{I_1}+ \Lambda_1+\Lambda_2 \quad &\quad \quad 2 \Lambda_2 + \Lambda_1 \quad&\overset{3}{\ce{->}} \quad \ce{I_2}+\Lambda_2 + \Lambda_1\\
%    \ce{I_1}\quad &\overset{2}{\ce{->}} \quad \Lambda_1  \quad &\quad \quad  \ce{I_2}\quad &\overset{4}{\ce{->}} \quad  \Lambda_2\\
%\end{aligned}.
%\end{equation*}
%Identically to \emph{\ref{eq:modelmin2}}, both are non-autocatalytic, do not need production or decay terms, and have the capacity for differentiation. We omit the detailed analysis, as Prop.~\ref{prop:nautII} applies with trivial modifications. 
%\end{remark}

\section{Outlook}

{
When looking beyond two interacting cells in future investigations,
the next natural step is to also consider the formation of salt and pepper
or striped patterns in more than two cells. Very few cells can then still be
modeled in a similar way as we have done before. For many locally interacting
cells though, discrete spatial operators may have to be considered and
possibly spatial continuum limits, thus PDEs and the occurrence of e.g. spatially
oscillatory patterns.}

{It would be interesting to see whether the same mathematical models are
successful here too, or if not, why?}

{Further, it is important for us not to add too many new features at the same
time in our mathematical models, but rather to formulate each next extension
carefully in order to be able to judge the relevance (or not) of additional
(biological and mathematical) functions.}

\section{Discussion}\label{sec:discussion} 

{We are interested in understanding the essential features of the Notch pathway
which drives two initially identical cells to finally produce different levels
of concentrations of Notch or Delta. In order to do so we employed biochemical
information as well as abstract
mathematical considerations for symmetry breaking. We work under the assumption
that evolution may have optimized this pathway since it is an essential
pathway preserved over many species.
Of course, looking at two cells is only our starting point,
since the Notch pathway is relevant also for more than two locally
interacting cells.}

{Our approach in the present paper is deliberately structural:
we assume that the kinetics are parameter-rich,
Def.~\ref{def:prich}. On the one hand, this broad class of nonlinearities is flexible and
particularly suited to biochemical contexts, as it encompasses e.g. Michaelis--Menten
and Hill kinetics.
On the other hand, our analysis serves to determine a preliminary structural
capacity of the network, rather than to represent a detailed biochemical mechanism,
since we assume complete freedom in the choice of parameters. Consequently, this
parameter-rich investigation should be understood as a qualitative analysis,
whose conclusions must be further assessed in detailed models where the parameters
are quantitatively related to biologically realistic ranges. In particular,
elementary mass-action kinetics constitute a parametrically more restrictive class
and are therefore not covered by the present framework.}

{From our analysis it turns out that the models \ref{eq:modeli}+\ref{eq:modelii} and \ref{eq:modeliii} can give rise to
differentiation in the sense we discussed before. Both models do include the
central features of the \ref{eq:modeli} model:
\begin{enumerate}[label=(\roman*)]
\item an asymmetric interaction between Notch of one cell and Delta
of a neighboring cell, with two Notch components and one Delta component,
\item the splitting of Notch and the subsequent transfer of one of its parts
to the neighboring cell,
\end{enumerate}
together with either the cis-interaction or the ligand activation. The fact that the central model \ref{eq:modeli} itself does not have the capacity for differentiation
gave rise to our construction and analysis of minimal models, keeping the following question in mind:}

{Are the central features (i) and (ii) of \ref{eq:modeli} formally necessary at all,
i.e. those features one might consider at first glance the typical ones
for the Notch pathway?}

{The minimal models \ref{minmodelsynthesis} and \ref{eq:Miii} have the capacity for differentiation.
They include both of these central features, but
interestingly, the involved entities cannot be
specialized the way they actually are in the \ref{eq:modeli} model.\\
\textbf{NonAut-II-2} has the capacity for differentiation and exhibits (i),
but not the subsequent transfer (ii). Instead an activation step is introduced.}

{\textbf{NonAut-II-1} also has differentiation capacity but does neither show (i), nor the transfer (ii), and  introduces the activation step too. Additionally, it needs
asymmetric reaction functions for an in principle symmetric interaction between
$\Lambda_1$ and $\Lambda_2$, which asks for an explanation.
This is similar for \ref{eq:Miiib}, which also does not show (i), but includes the transfer (ii).
Thus at first glance a concrete biochemical interpretation for these two models,
\textbf{NonAut-II-1} and \ref{eq:Miiib}, seems more difficult than for the other cases.}

{This indicates that some essential features of \ref{eq:modeli} are necessary for the pathway,
but a further mechanism, such as the cis-interaction in \ref{eq:modelii} or the ligand activation
in \ref{eq:modeliii}, is essentially needed for symmetry-breaking.}

{ 
The instability motifs identified here should be interpreted as structural
prerequisites for differentiation, whose dynamical realizations depend on the admissible
kinetics and parameter values. Nevertheless, by construction, these instability motifs
persist in any larger network containing them, since the addition of further reactions
and species does not alter their minimal structure, even though the dynamics of the
resulting instability (e.g. differentiation vs oscillatory behavior) must be assessed case-by-case.}

{
A further question of interest is about autocatalysis. Here, we rely on a purely
stoichiometric definition of autocatalysis to classify
the models as autocatalytic or non-autocatalytic. Both, \ref{eq:modeli}+\ref{eq:modelii} and \ref{eq:modeliii}, are
non-autocatalytic. In the class of minimal models which can give rise to differentiation,
\ref{minmodelsynthesis} and \ref{eq:Miii} are both autocatalytic instead. Among the relevant non-autocatalytic reduced
networks are \ref{eq:minmodelabstr} and \ref{eq:modelmin2}, but \ref{eq:minmodelabstr} needs production terms in order
to work. \textbf{NonAut-II-1} needs the above mentioned asymmetric reaction functions,
as does \ref{eq:Miiib}. In turn, \textbf{NonAut-II-2} does not need asymmetric reaction functions and is
in that sense closest to \ref{eq:modeli}+\ref{eq:modelii} and \ref{eq:modeliii}, and a smaller network too.}

{
If we keep the reasoning from the start of our paper, namely to exclude production
and degradation as central players for the differentiation process, we should dismiss
\ref{eq:minmodelabstr} here. Then we are left with the question of autocatalysis vs non-autocatalysis,
at least when comparing all minimal models, and additionally with the question, which
types of asymmetries (interaction partners, reaction functions) do allow for a
biochemical explanation. The latter one seems more natural for asymmetries induced by
stoichiometry. Discussing this issue probably needs more experimental information.
As for the former, the models we derived from the biological information available to us
are all non-autocatalytic. But this does not mean that we can answer the question here
why or if evolution favours non-autocatalytic networks over autocatalytic ones in the
context of cell differentiation.}

{Therefore let’s shift the point of view. Does our analysis allow for a reasonable
hypothesis when looking back to the early evolution of the Notch pathway or to other
pathways relevant for cell differentiation? If so, then our suggestions would be to
consider the following questions:}

{
Did early versions of the Notch pathway start from the finally extracellular and
intracellular Notch domains and Delta all being similar entities?
Both \ref{eq:modeli}+\ref{eq:modelii} and \ref{eq:modeliii} - which allow for differentiation - include the asymmetric
interaction (i), with splitting of two entities from cell 1 and one entity from cell 2
which afterwards are redistributed as one entity in
cell 1 and two entities in cell 2 (ii), and vice versa.
Did early versions of the Notch pathway not yet exhibit this asymmetric interaction with
splitting, i.e. (i) and (ii), but rather have one entity of cell 1 interacting with
another single entity of cell 2, or even have one entity of cell 1 interacting with the
corresponding entity of cell 2, like the dynamics we have described in our
Gedankenexperimente?}

{Is there perhaps a possibility to develop biomimetic or biochemical experiments e.g.
for the minimal models \ref{minmodelsynthesis} and \textbf{NonAut-II-2} where one could see different concentrations
arising in two different but interacting ‘containers’ which had the same initial
concentrations?}

\section{Proofs}\label{sec:proofs}

\subsection*{Proofs of Sec.~\ref{sec:mainresult}}

{
The first two steps of proofs of Thm.~\ref{thm:11}, Thm.~\ref{thm:12}, and Thm.~\ref{thm:13}, are similar, so we give a short overview here.}

{\emph{Step 1: The network is consistent.} We will verify that each of the ODE systems \eqref{eq:system11}, \eqref{eq:system12} and \eqref{eq:system13} admits a positive steady-state for a choice of monotone chemical functions. We will write the ODE system as  $\dot{x}=S\mathbf{r}(x).$ The existence of a positive steady-state is equivalent to the existence of a strictly positive kernel vector $v>0$ of the stoichiometric matrix $S$.}

{\emph{Step 2: The network is nondegenerate.} In order to verify this, we compute the $n$ linearly independent conservation laws of the respective system by providing a basis of the left kernel space of $S$. Then, we give an invertible $\tilde{k}$-CS matrix, where $\tilde{k}=|M|-n$.}

\proof[Proof of Thm.~\ref{thm:11}]

{We divide the proof in three steps, where the first two steps have shortly
been described above. The third step deals with the uniquess and local stability of the steady-state,
irrespective of the reaction rate parameters.}

\textit{Step 1.} The system admits at least one positive steady state. This is equivalent to finding a positive right-kernel vector of the stoichiometric matrix $S$ of \ref{eq:modeli}, which reads
\begin{equation*}S=
\scalebox{0.70}{$\begin{pmatrix}
         -1 & 0 & 1 & 0 & 0 & 0\\
         -1 & 1 & 0 & 0 & 0 & 0\\
         1 & -1 & 0 & 0 & 0 & 0\\
         0 & 0 & 1 & 0 & -1 & 0\\
         0 & 0 & -1 & 0 & 1 & 0\\
         0 & 0 & 0 & -1 & 0 & 1\\
         0 & 0 & 0 & -1 & 1 & 0\\
         0 & 0 & 0 & 1 & -1 & 0 \\
         0 & -1 & 0 & 0 & 0 & 1\\
         0 & 1 & 0 & 0 & 0 & -1
     \end{pmatrix}$.}
\end{equation*}
Only the stoichiometric coefficients $\{0,1\}$ are present and any species participates only in two reactions, one as a reactant and one as a product. This implies that the positive vector $v=(c,c,c,c,c,c)$, $c\in\mathbb{R}_{>0}$, is indeed a right-kernel vector of $S$, $Sv=0$. Thus the network is consistent and one positive steady state exists.

\textit{Step 2.}
For nondegeneracy, the dimension of system \eqref{eq:system11} must be equal the number of species (10) minus the number of independent linear conservation laws, i.e., the dimension $n$ of the left kernel $W$ of the stoichiometric matrix $S$. For this, we compute an arbitrary basis as:
\begin{equation*}
\begin{split}
    W=\operatorname{span}\langle
    &(
        0,1,1,0,0,0,0,0,0,0
    ),
    (
        0,0,0,1,1,0,0,0,0,0
    ),
    (
        1,0,1,0,1,1,0,1,0,1
    ),\\
   &(
        0,0,0,0,0,0,1,1,0,0
    ),
    (
     0,0,0,0,0,0,0,0,1,1
    )\rangle
\end{split}
\end{equation*}
Thus there are 5 linearly independent conservation laws ($n=5$), namely
\begin{equation*}
\begin{cases}
w_1 := \ce{N\!I_1}+\ce{N_1},\\
w_2 := \ce{D_1}+\ce{T_1},\\
w_3 := \ce{N\!E_1}+\ce{N_1}+\ce{T_1}+\ce{N\!E_2}+\ce{N_2}+\ce{T_2},\\
w_4 := \ce{N\!I_2}+\ce{N_2},\\
w_5 := \ce{D_2}+\ce{T_2}.
\end{cases}
\end{equation*}
All species appear in at least one of the conserved quantities, which guarantees that the orbits of the system \eqref{eq:system11} 
cannot diverge and are bounded. To show that 
the system is nondegenerate, it is enough to show that the coefficient 
$a_{|\mathbf{M}|-n}=a_5$ of the characteristic polynomial $g(\lambda)$ of the symbolic 
Jacobian is non-identically zero. Via Lemma \ref{lem:CSexpansion}, this is equivalent to the existence of at least one $5\times5$ invertible 
CS-matrix.\\
We provide an arbitrary example. Consider the CS-triple 
\begin{equation*}
\pmb{\kappa}=\scalebox{0.80}{$(\{\ce{N\!E_1},\ce{N_1},\ce{T_1},\ce{N_2},\ce{T_2}\},\{11,12,13,22,23\},J(\ce{N\!E_1},\ce{N_1},\ce{T_1},\ce{N_2},\ce{T_2})=(11,12,13,22,23)).$}
\end{equation*}
Its associated CS-matrix $S[\pmb{\kappa}]$ reads:
\begin{equation*}
S[\pmb{\kappa}]=\scalebox{0.70}{$
    \begin{pmatrix}
        -1 & 0 &0 & 0 & 0\\
        1 & -1 & 0 & 0 & 0\\
        0 & 0 &-1 & 1 & 0\\
        0 & 0 & 0 & -1 & 0\\
        0 & 1 & 0 & 0 & -1
    \end{pmatrix}$},
\end{equation*}
which is invertible (all eigenvalues are $-1$). The network is thus nondegenerate.

\textit{Step 3: The positive steady state is unique and always locally stable.}
Since
$G=SR$, its nonzero spectrum is identical to that of $H:=RS$. Here:
\begin{equation*}
R=\scalebox{0.70}{$\begin{pmatrix}
     r_{11\ce{N\!E_1}} & r_{11\ce{N\!I_1}} & 0 & 0 & 0 & 0 & 0 & 0 & 0 & 0\\
     0 & 0 & r_{12\ce{N_1}} & 0 & 0 & 0 & 0 & 0 & r_{12\ce{D_2}} & 0\\
     0 & 0 & 0 & 0 & r_{13\ce{T_1}} & 0 & 0 & 0 & 0 & 0\\
     0 &0 & 0 & 0 & 0 &r_{21\ce{N\!E_2}} & r_{21\ce{N\!I_2}} & 0 & 0 & 0\\
      0 & 0 & 0 & r_{22\ce{D_1}} & 0 & 0 & 0 & r_{22\ce{N_2}} & 0 & 0\\
       0 & 0 & 0 & 0 & 0 & 0 & 0 & 0 & 0 & r_{23\ce{T_2}}
    \end{pmatrix}$,}
\end{equation*}
which yields that $H$ equals
\begin{equation*}
\scalebox{0.70}{$\begin{pmatrix}
 -r_{11\ce{N\!E_1}} - r_{11\ce{N\!I_1}} & r_{11\ce{N\!I_1}} & r_{11\ce{N\!E_1}} & 0 & 0 & 0\\
 r_{12\ce{N_1}} & -r_{12\ce{D_2}} - r_{12\ce{N_1}} & 0 & 0 & 0 & r_{12\ce{D_2}}\\
 0 & 0 & -r_{13\ce{T_1}} & 0 & r_{13\ce{T_1}} & 0\\
 0 & 0 & 0 & -r_{21\ce{N\!E_2}} - r_{21\ce{N\!I_2}} & r_{21\ce{N\!I_2}} & r_{21\ce{N\!E_2}}\\
 0 & 0 & r_{22\ce{D_1}} & r_{22\ce{N_2}} & -r_{22\ce{D_1}} - r_{22\ce{N_2}} & 0\\
 0 & r_{23\ce{T_2}} & 0 & 0 & 0 & -r_{23\ce{T_2}}
\end{pmatrix},$}
\end{equation*}
which is \emph{weakly diagonally dominant} by rows, i.e., $|H_{ii}|\ge \sum_{j\neq i} |H_{ij}|$
with negative diagonal entries: $H_{ii}<0$ for any $i$. In \cite[Prop.~2]{BlanchiniGiordano21}, it was shown in full generality that $H=RS$ is weakly diagonally dominant with negative diagonal whenever the stoichiometric coefficients are $\{0,1\}$ only, and each species participates at most in two reactions, as it is the case here. \phantomsection\label{pt:translocation}Our analysis holds irrespective of whether or not one considers the formulation \ref{eq:modeli'} with an explicit translocation of the intracellular domain $\ce{N\!I}_j$ to the nucleus. Indeed, in \ref{eq:modeli'} all stoichiometric coefficients are in $\{0,1\}$, and each species participates in at most two reactions. For $j=1,2$, the species $\ce{N\!I}_j$ participates in reactions $(j1,j2'')$, while $\ce{N\!I^{Nucleus}}_j$ participates in reactions $(j2',j2'')$.

{Due to weak diagonal dominance of the negative-diagonal matrix $H=RS$, the Gershgorin disk theorem \cite{gerschgorin1931} implies that the eigenvalues of the symbolic Jacobian $G=SR$ also have nonpositive real parts: $H$ and $G$ share the same nonzero spectrum. The five linearly independent conservation laws imply that five eigenvalues are identically zero. The remaining $10-5=5$ eigenvalues determine the asymptotic local stability of any existing steady state. The product of these five eigenvalues is the coefficient $a_5$ of the characteristic polynomial of $G$, which is not identically zero as proved above. If we would assume that there exists a choice of $R^*$ for which $a_5=0$, then, since $a_5$ is a nonzero multilinear homogeneous polynomial in the symbolic variables appearing in $R^*$ (Lemma~\ref{lem:CSexpansion}), the existence of a positive realization with $a_5=0$ implies the existence of positive realizations $R_1^*$ and $R_2^*$ for which $a_5>0$ and $a_5<0$, respectively. Such a sign change occurs if and only if $a_5$ contains two monomials with coefficients of opposite sign. In the present setting this is equivalent to the existence of two $5$-CS matrices whose determinants have opposite signs. The case $a_5>0$, however, would imply the existence of a positive real eigenvalue, a contradiction. Therefore, the five eigenvalues have negative real parts and every steady state of the system is asymptotically locally stable for all choices of monotone chemical functions. Finally, to exclude multistationarity, we invoke \cite[Theorem 3]{BaPa16}, which states that multistationarity for consistent nondegenerate networks is equivalent to the capacity of the Jacobian of the reduced system to have a real zero eigenvalue. In this case, this is again equivalent to the existence of a choice $R^*$ such that $a_5 = 0$, a fact that we have already excluded. Thus the system always admits only one locally asymptotically stable steady state.} \endproof

{After Step 1 and Step 2 as in the proof of Thm.~\ref{thm:11}, the proofs for  Thm.~\ref{thm:12} and Thm.~\ref{thm:13} proceed, following the workflow introduced in Sec.~\ref{sec:preliminary}. First we give an overview of the Steps 3 and 4 for this case.}

\emph{Step 3.}\label{step:3} Via the Python module \textsc{BiRNe} \cite{Golnik25}, we compute and analyze the characteristic polynomial of the symbolic Jacobian matrix $G=SR$. Via Lemma \ref{lem:CSexpansion}, we will look for all $k\times k$ CS-matrices $S[\pmb{\kappa}]$, {$k\in\{1,\dots,|M|-n\}$}, such that
\begin{equation}\label{eq:signdetbadwf}
\operatorname{sign}\operatorname{det}S[\pmb{\kappa}]=(-1)^{k-1},
\end{equation}
and identify the \emph{unstable-positive feedbacks} as the minimal ones among those, i.e. the CS-matrices $S[\pmb{\kappa}]$ with \eqref{eq:signdetbadwf} and such that none of its principal submatrices $S[\pmb{\kappa}']$ satisfies 
$
\operatorname{sign}\operatorname{det}S[\pmb{\kappa}']=(-1)^{k'-1}.$
%{\color{red}BELOW one finally understands - for the first time - what a motif is. SO before one has at least to refer to this part. BETTER explain at least a motif much earlier.}
\textsc{BiRNe} scans the coefficients $a_i$ in the $g(\lambda)$ \eqref{eq:charpoly} starting from $a_0 = 1$. In each coefficient, the summands $\operatorname{det}S[\pmb{\kappa}]\prod_{X_m \in \kappa} R_{J(X_m)m}$ in \eqref{eq:CSexpansion} that satisfy condition \eqref{eq:signdetbadwf} are identified and labeled by their monomials $\prod_{X_m \in \kappa} R_{J(X_m)m}$. From this structure, a Hasse diagram \cite{baker_partial_1972} is constructed, where each vertex corresponds to a summand in $g(\lambda)$ fulfilling \eqref{eq:signdetbadwf}.
A directed edge between two vertices is drawn, if and only if, 
the monomial label of one summand is a subset of that of another, with the direction following the order of increasing indices $i$, i.e. from $a_i$ to $a_l$ with $i < l$. By this construction, the root vertices, i.e. those with no incoming edges, identify all the CS-matrices that are minimal w.r.t. the property of having exactly one real positive eigenvalue, which are {exactly} the unstable-positive feedbacks. %We will then check explicitly those candidates obtained via \textsc{BiRNe} and exclude any unstable principal submatrix. Since minimality w.r.t. real positive eigenvalues is guaranteed by construction, we will explicitly exclude any principal submatrix with complex eigenvalues with positive real part.
If no unstable-positive feedback is represented by a Metzler CS-matrix, we classify the networks as non-autocatalytic.

\emph{Step 4.} Finally, we enforce the kinetic symmetric, 
respectively \eqref{eq:constraints2} for \eqref{eq:system12} and \eqref{eq:constraints3} for \eqref{eq:system13}. We will assume that the steady-state is homogeneous, i.e. {$\bar{[\ce{X}]}_1=\bar{[\ce{X}]}_2$} respectively for $\ce{X}=\{\ce{N\!E},\ce{N\!I},\ce{N},\ce{D}, \ce{T},\ce{C}\}$ in \eqref{eq:system12} and $\ce{X}=\{\ce{N\!E},\ce{N\!I},\ce{N},\ce{D}, \ce{T},\ce{B}\}$ in \eqref{eq:system13}. These two assumptions imply the symmetric constraint for the reactivities in $R$, $r_{j\ce{X}}=r_{\sigma(j)\sigma(\ce{X})},$ where $\sigma$ associates a reaction $1j$ to the corresponding reaction $2j$, and a species $\ce{X}_1$ to the corresponding species $\ce{X}_2$ and vice versa. Again, we compute $g(\lambda)$ under such symmetric constraint, and show that {we can solve the polynomial equation $a_{|M|-n}=0$ w.r.t. the symbolic variables $r_{jm}$},
where $a_{|M|-n}$ is the determinant of the Jacobian of the system reduced to one stoichiometric compatibility class.

\proof[Proof of Thm.~\ref{thm:12}]
\emph{Step 1.}
The stoichiometric matrix $S$ of  \ref{eq:modeli}+\ref{eq:modelii} reads:
\begin{equation}\label{eq:s12}
S=\scalebox{0.70}{$\begin{pmatrix}
         -1 & 0 & 1 & 0 & 0 & 0 & 0 & 0 &0 &0\\
         -1 & 1 & 0 & 0 & 0 & 0 & 0 & 0 &0 &0\\
         1 & -1 & 0 & -1 & 1 & 0 & 0 & 0&0 &0\\
         0 & 0 & 1 & - 1 & 1 & 0 & -1 & 0&0 &0\\
         0 & 0 & -1 & 0 & 0 & 0 & 1 & 0&0 &0\\
         0 & 0 & 0 & 1 & -1& 0 & 0 & 0&0 &0\\
         0 & 0 & 0 & 0 & 0& -1 & 0 & 1&0 &0\\
         0 & 0 & 0 & 0 & 0& -1 & 1 & 0&0 &0\\
         0 & 0 & 0 &0 & 0& 1 & -1 & 0 & -1 & 1\\
         0 & -1 & 0 &0 & 0& 0 & 0 & 1 & -1 & 1\\
         0 & 1 & 0 & 0 & 0&0 & 0 & -1&0 &0\\
         0 & 0 & 0 & 0 & 0&0 & 0 & 0 & 1 &-1
     \end{pmatrix}$,}
\end{equation}
from which we can compute that any positive vector $v$ such that $Sv=0$ is of the form $v=(c,c,c,h,h,c,c,c,l,l)$ for $c,h,l\in \mathbb{R}_{>0}$. 
In particular, the flux cone $\mathcal{F}:=\{v\in \mathbb{R}^{10}_{>0} \; | 
\; 
{Sv = 0}
\}$ is three dimensional, and thus nonempty, and the network is consistent and admits a positive steady-state.

\emph{Step 2.} We compute a basis for the left kernel $W$ of $S$ as:
\begin{equation*}
\begin{split}
    W=\operatorname{span}\langle
(0,1,1,0,0,1,0,0,0,0,0,0
),
(0,0,0,1,1,1,0,0,0,0,0,0
  ),\\
(1,0,1,0,1,1,1,0,1,0,1,1
    ),(0,0,0,0,0,0,0,1,1,0,0,1),(0,0,0,0,0,0,0,0,0,1,1,1
    )\rangle,
    \end{split}
\end{equation*}
and thus there are 5 linearly independent conservation laws ($n=5$), namely
\begin{equation*}
\begin{cases}
w_1 := \ce{N\!I_1}+\ce{N_1}+\ce{C_1},\\
w_2 := \ce{D_1}+\ce{T_1}+\ce{C_1},\\
w_3 := \ce{N\!E_1}+\ce{N_1}+\ce{T_1}+\ce{C_1}+\ce{N\!E_2}+\ce{N_2}+\ce{T_2}+\ce{C_2},\\
w_4 := \ce{N\!I_2}+\ce{N_2}+\ce{C_2},\\
w_5 := \ce{D_2}+\ce{T_2}+\ce{C_2}.
\end{cases}
\end{equation*}
All species appear in at least one of the conserved quantities, which guarantees that orbits of the system \eqref{eq:system12} cannot diverge and they are bounded. To show that the system 
is nondegenerate, it is enough to show that the coefficient $a_{|M|-n}=a_7$ of 
$g(\lambda)$ is 
non-identically zero. Via Lemma \ref{lem:CSexpansion}, this is in turn 
equivalent to the existence of at least one $7\times7$ invertible CS-matrix. 
{Again, we provide an arbitrary example. Consider} 
\begin{equation*}
    \pmb{\kappa}=\scalebox{0.70}{$(\{\ce{N\!E_1},\ce{N_1},\ce{T_1},\ce{C_1}\ce{N_2},\ce{T_2},\ce{C_2}\},\{11,12,13,15,22,23,25\},J(\ce{N\!E_1},\ce{N_1},\ce{T_1},\ce{C_1},\ce{N_2},\ce{T_2},\ce{C_2})=(11,12,13,15,22,23,25)).$}
\end{equation*}
Its associated CS-matrix $S[\pmb{\kappa}]$ reads:
\begin{equation*}
S[\pmb{\kappa}]=\scalebox{0.70}{$
    \begin{pmatrix}
        -1 & 0 &0 &0& 0 & 0 &0\\
        1 & -1 & 0 &1 & 0 & 0 &0\\
        0 & 0 &-1 & 0& 1 & 0 &0\\
        0 & 0 & 0 &-1 & 0 &0&0\\
        0 & 0 & 0 & 0&-1 & 0&1\\
        0 & 1 & 0 & 0&0 & -1&0\\
        0 & 0 & 0 & 0& 0 & 0 & -1
    \end{pmatrix}$},
\end{equation*}
which is invertible, since all eigenvalues are $-1$. The network is thus nondegenerate.

\emph{Step 3.} We compute the symbolic Jacobian $G=SR$ as the product of the stoichiometric matrix \eqref{eq:s12} and
\begin{equation*}
R=\scalebox{0.70}{$\begin{pmatrix}
r_{11\ce{N\!E}_1} & r_{11\ce{N\!I}_1} & 0 & 0 & 0 & 0 & 0 & 0 & 0 & 0 & 0 & 0\\
0 & 0 & r_{12\ce{N}_1} & 0 & 0 & 0 & 0 & 0 & 0 & r_{12\ce{D}_2} & 0 & 0\\
0 & 0 & 0 & 0 & r_{13\ce{T}_1} & 0 & 0 & 0 & 0 & 0 & 0 & 0\\
0 & 0 & r_{14\ce{N}_1} & r_{14\ce{D}_1} & 0 & 0 & 0 & 0 & 0 & 0 & 0 & 0\\
0 & 0 & 0 & 0 & 0 & r_{15\ce{C}_1} & 0 & 0 & 0 & 0 & 0 & 0\\
0 & 0 & 0 & 0 & 0 & 0 & r_{21\ce{N\!E}_2} & r_{21\ce{N\!I}_2} & 0 & 0 & 0 & 0\\
0 & 0 & 0 & r_{22\ce{D}_1} & 0 & 0 & 0 & 0 & r_{22\ce{N}_2} & 0 & 0 & 0\\
0 & 0 & 0 & 0 & 0 & 0 & 0 & 0 & 0 & 0 & r_{23\ce{T}_2} & 0\\
0 & 0 & 0 & 0 & 0 & 0 & 0 & 0 & r_{24\ce{N}_2} & r_{24\ce{D}_2} & 0 & 0\\
0 & 0 & 0 & 0 & 0 & 0 & 0 & 0 & 0 & 0 & 0 & r_{25\ce{C}_2}
\end{pmatrix}$,}
\end{equation*}
which we do not write down here explicitly for the 
sake of space. Then we compute the characteristic polynomial of $G$ and, 
with the expansion in Lemma \ref{lem:CSexpansion}, we look for unstable-positive feedback.
We use the \textsc{BiRNe} Python module \cite{Golnik25}, which takes as input a list of reactions and produces as output the list of CS-matrices $S[\pmb{\kappa}]$ that are minimal with the property of having one real positive eigenvalue.  There are exactly six CS-triples ($\pmb{\kappa}_1$, $\pmb{\kappa}_2$, $\pmb{\kappa}_3$, $\pmb{\kappa}_4$, $\pmb{\kappa}_5$, $\pmb{\kappa}_6$) whose CS-matrices are unstable-positive feedbacks:
{\small
\begin{equation*}
\begin{cases}
    \kappa_1=\{\ce{N\!I}_1,\ce{N}_1,\ce{D}_1,\ce{N}_2,\ce{D}_2\},\; E_{\kappa_1}=\{ 11,12,14,22,24\}, \; J(\kappa_1)=\{ 11,12,14,22,24\}\\
    \kappa_2=\{\ce{N}_1,\ce{D}_1,\ce{N\!I}_2,\ce{N}_2,\ce{D}_2\},\; E_{\kappa_2}=\{12,14,21,22,24\},\; J(\kappa_2)=\{ 12,14,21,22,24\}\\
    \kappa_3= \{\ce{N}_1,\ce{D}_1,\ce{T}_1,\ce{N}_2,\ce{D}_2\},\; E_{\kappa_3}=\{12,13,14,22,24\},\; J(\kappa_3)=\{14,22,13,24,12\}\\
    \kappa_4=\{\ce{N}_1,\ce{D}_1,\ce{N}_2,\ce{D}_2,\ce{T}_2\},\; E_{\kappa_4}=\{12,14,22,23,24\},\; J(\kappa_4)=\{14,22,24,12,23\}\\
    \kappa_5=\{\ce{N\!I}_1,\ce{N}_1,\ce{D}_1,\ce{N\!E}_2,\ce{N}_2,\ce{T}_2 \},\; E_{\kappa_5}=\{11,12,14,21,22,23\}, \;J(\kappa_5)=\{11,12,14,21,22,23\}\\
    \kappa_6=\{\ce{N\!E}_1,\ce{N}_1,\ce{T}_1,\ce{N\!I}_2,\ce{N}_2,\ce{D}_2\}, \; E_{\kappa_6}=\{11,12,13,21,22,24\},\; J(\kappa_6)=\{11,12,13,21,22,24\}\\
\end{cases}
\end{equation*}}
{The associated CS-matrices read
\begin{equation*}
    S[\pmb{\kappa}_1]=\scalebox{0.70}{$\begin{blockarray}{cccccc}
& 11 & 12 & 14 & 22 & 24\\
\begin{block}{c(ccccc)}
\ce{N\!I_1} & -1 & 1 & 0 & 0 & 0\\
\ce{N_1} & 1 & -1 & -1 & 0 & 0\\
\ce{D}_{1} & 0 & 0 & -1 & -1 & 0\\
\ce{N_2} & 0 & 0 & 0 & -1 & -1\\
\ce{D}_{2} & 0 & -1 & 0 & 0 & -1\\
\end{block}
\end{blockarray}$}\quad\quad S[\pmb{\kappa}_2]=\scalebox{0.70}{$\begin{blockarray}{cccccc}
& 12 & 14 & 21 & 22 & 24\\
\begin{block}{c(ccccc)}
\ce{N_1}      & -1 & -1 & 0 & 0 & 0\\
\ce{D_1}      & 0 & -1 & 0 & -1 & 0\\
\ce{N\!I}_{2} & 0 & 0 & -1 & 1 & 0\\
\ce{N_2}      & 0 & 0 & 1 & -1 & -1\\
\ce{D}_{2}    & -1 & 0 & 0 & 0 & -1\\
\end{block}
\end{blockarray}$}
\end{equation*}
\begin{equation*}
S[\pmb{\kappa}_3]=\scalebox{0.70}{$
\begin{blockarray}{cccccc}
& 14 & 22 & 13 & 24 & 12\\
\begin{block}{c(ccccc)}
\ce{N_1} & -1 & 0 & 0 & 0 & -1\\
\ce{D}_{1} & -1 & -1 & 1 & 0 & 0\\
\ce{T_1} & 0 & 1 & -1 & 0 & 0\\
\ce{N_2} & 0 & -1 & 0 & -1 & 0\\
\ce{D}_{2} & 0 & 0 & 0 & -1 & -1\\
\end{block}
\end{blockarray}$}\quad\quad S[\pmb{\kappa}_4]=\scalebox{0.70}{$\begin{blockarray}{cccccc}
& 14 & 22 & 24 & 12 & 23\\
\begin{block}{c(ccccc)}
\ce{N_1}   & -1 & 0 & 0 & -1 & 0\\
\ce{D}_{1} & -1 & -1 & 0 & 0 & 0\\
\ce{N_2}   & 0 & -1 & -1 & 0 & 0\\
\ce{D}_{2} & 0 & 0 & -1 & -1 & 1\\
\ce{T_2} &  0 & 0 & 0 & 1 & -1\\
\end{block}
\end{blockarray}$}
\end{equation*}
\begin{equation*}
S[\pmb{\kappa}_5]=\scalebox{0.70}{$\begin{blockarray}{ccccccc}
& 11 & 12 & 14 & 21 & 22 & 23\\
\begin{block}{c(cccccc)}
\ce{N\!I_1} & -1 & 1 & 0 & 0 & 0 & 0\\
\ce{N_1} & 1 & -1 & -1 & 0 & 0 & 0\\
\ce{D}_{1} & 0 & 0 & -1 & 0 & -1 & 0\\
\ce{N\!E_2} & 0 & 0 & 0 & -1 & 0 & 1\\
\ce{N_2} & 0 & 0 & 0 & 1 & -1 & 0\\
\ce{T_2} & 0 & 1 & 0 & 0 & 0 & -1\\
\end{block}
\end{blockarray}$}\quad S[\pmb{\kappa}_6]=\scalebox{0.70}{$\begin{blockarray}{ccccccc}
& 11 & 12 & 13 & 21 & 22 & 24\\
\begin{block}{c(cccccc)}
\ce{N\!E_1} & -1 & 0 & 1 & 0 & 0 & 0\\
\ce{N_1}    & 1 & -1 & 0 & 0 & 0 & 0\\
\ce{T}_{1}  & 0 & 0 & -1 & 0 & 1 & 0\\
\ce{N\!I_2} & 0 & 0 & 0 & -1 & 1 & 0\\
\ce{N_2}    & 0 & 0 & 0 & 1 & -1 & -1\\
\ce{D_2}    & 0 & -1 & 0 & 0 & 0 & -1\\
\end{block}
\end{blockarray}$}
\end{equation*}
%Above, we have excluded unstable principal submatrices with complex conjugated eigenvalues by direct computation, omitted here for brevity.
Note that the pairs $(\pmb{\kappa}_1,\pmb{\kappa}_2)$, $(\pmb{\kappa}_3,\pmb{\kappa}_4)$, $(\pmb{\kappa}_5,\pmb{\kappa}_6)$ are the ones related by symmetry. The reactions and 
species, which appear in such CS-triples, respectively constitute the instability 
motifs \eqref{eq:motif12a}, \eqref{eq:motif12b}, \eqref{eq:motif12c}. As no matrix $S[\pmb{\kappa}_i]$, $i=1,...6$, is Metzler, i.e. they all possess negative off-diagonal entries, the network is non-autocatalytic. %{\color{red}SAY more here already or  refer to a better explanation.}

\emph{Step 4.} We assume the kinetic symmetry condition \eqref{eq:constraints2} at a \emph{homogeneous} steady state, i.e. where {$[\bar{\ce{X}}_1] = [\bar{\ce{X}}_2]$} for all $\ce{X} \in \{\ce{N\!E}, \ce{N\!I}, \ce{N}, \ce{D}, \ce{T}, \ce{C}\}$.
Then, there is no quantitative distinction between the two cells.
To simplify notation, we denote the reaction rates as $r_j := r_{1j} = r_{2j}$ for $j \in \{1,2,3,4,5\}$, and {$[\bar{X}] := [\bar{X}_1] = [\bar{X}_2]$.}
This yields the following symbolic reactivity matrix $R_\sigma$ under kinetic symmetry:
\begin{equation*}
R_\sigma= \scalebox{0.70}{$\begin{pmatrix}
r_{1\ce{N\!E}} & r_{1\ce{N\!I}} & 0 & 0 & 0 & 0 & 0 & 0 & 0 & 0 & 0 & 0\\
0 & 0 & r_{2\ce{N}} & 0 & 0 & 0 & 0 & 0 & 0 & r_{2\ce{D}} & 0 & 0\\
0 & 0 & 0 & 0 & r_{3\ce{T}} & 0 & 0 & 0 & 0 & 0 & 0 & 0\\
0 & 0 & r_{4\ce{N}} & r_{4\ce{D}} & 0 & 0 & 0 & 0 & 0 & 0 & 0 & 0\\
0 & 0 & 0 & 0 & 0 & r_{5\ce{C}} & 0 & 0 & 0 & 0 & 0 & 0\\
0 & 0 & 0 & 0 & 0 & 0 & r_{1\ce{N\!E}} & r_{1\ce{N\!I}} & 0 & 0 & 0 & 0\\
0 & 0 & 0 & r_{2\ce{D}} & 0 & 0 & 0 & 0 & r_{2\ce{N}} & 0 & 0 & 0\\
0 & 0 & 0 & 0 & 0 & 0 & 0 & 0 & 0 & 0 & r_{3\ce{T}} & 0\\
0 & 0 & 0 & 0 & 0 & 0 & 0 & 0 & r_{4\ce{N}} & r_{4\ce{D}} & 0 & 0\\
0 & 0 & 0 & 0 & 0 & 0 & 0 & 0 & 0 & 0 & 0 & r_{5\ce{C}}
\end{pmatrix}$},
\end{equation*}
where due to the $\mathbb{Z}_2$-symmetry, each symbol appears exactly two times. The symbolic Jacobian $G_\sigma=SR_\sigma$, at kinetic symmetry, is 
\begin{equation*}
\scalebox{0.70}{$\begin{pmatrix}
 -r_{1\ce{N\!E}} & -r_{1\ce{N\!I}} & 0 & 0 & r_{3\ce{T}} & 0 & 0 & 0 & 0 & 0 & 0 & 0\\
 -r_{1\ce{N\!E}} & -r_{1\ce{N\!I}} & r_{2\ce{N}} & 0 & 0 & 0 & 0 & 0 & 0 & r_{2\ce{D}} & 0 & 0\\
  r_{1\ce{N\!E}} &  r_{1\ce{N\!I}} & -r_{2\ce{N}} - r_{4\ce{N}} & -r_{4\ce{D}} & 0 & r_{5\ce{C}} & 0 & 0 & 0 & -r_{2\ce{D}} & 0 & 0\\
  0 & 0 & -r_{4\ce{N}} & -r_{2\ce{D}} - r_{4\ce{D}} & r_{3\ce{T}} & r_{5\ce{C}} & 0 & 0 & -r_{2\ce{N}} & 0 & 0 & 0\\
  0 & 0 & 0 & r_{2\ce{D}} & -r_{3\ce{T}} & 0 & 0 & 0 & r_{2\ce{N}} & 0 & 0 & 0\\
  0 & 0 & r_{4\ce{N}} & r_{4\ce{D}} & 0 & -r_{5\ce{C}} & 0 & 0 & 0 & 0 & 0 & 0\\
  0 & 0 & 0 & 0 & 0 & 0 & -r_{1\ce{N\!E}} & -r_{1\ce{N\!I}} & 0 & 0 & r_{3\ce{T}} & 0\\
  0 & 0 & 0 & r_{2\ce{D}} & 0 & 0 & -r_{1\ce{N\!E}} & -r_{1\ce{N\!I}} & r_{2\ce{N}} & 0 & 0 & 0\\
  0 & 0 & 0 & -r_{2\ce{D}} & 0 & 0 & r_{1\ce{N\!E}} & r_{1\ce{N\!I}} & -r_{2\ce{N}} - r_{4\ce{N}} & -r_{4\ce{D}} & 0 & r_{5\ce{C}}\\
  0 & 0 & -r_{2\ce{N}} & 0 & 0 & 0 & 0 & 0 & -r_{4\ce{N}} & -r_{2\ce{D}} - r_{4\ce{D}} & r_{3\ce{T}} & r_{5\ce{C}}\\
  0 & 0 & r_{2\ce{N}} & 0 & 0 & 0 & 0 & 0 & 0 & r_{2\ce{D}} & -r_{3\ce{T}} & 0\\
  0 & 0 & 0 & 0 & 0 & 0 & 0 & 0 & r_{4\ce{N}} & r_{4\ce{D}} & 0 & -r_{5\ce{C}}
\end{pmatrix}$}
\end{equation*}
We compute the characteristic polynomial $g_\sigma(\lambda)$ of $G_\sigma$. Since the network is nondegenerate, the determinant of the Jacobian matrix reduced to a stoichiometric compatibility class is interpreted as the coefficient $a_{|M|-n}=a_{|12|-5}=a_{7}$, which reads:
\begin{equation*}
\begin{split}
    a_7=
&2\,r_{3\ce{T}}\,r_{1\ce{N\!E}}
\bigl(r_{2\ce{D}}r_{4\ce{N}} + r_{4\ce{D}}r_{2\ce{N}} 
+ r_{5\ce{C}}r_{2\ce{N}}\bigr)\cdot \dots\\
&\dots \cdot \bigl(
r_{4\ce{D}}r_{2\ce{N}}r_{3\ce{T}}
- r_{2\ce{D}}r_{4\ce{N}}r_{3\ce{T}}
+ r_{5\ce{C}}r_{2\ce{N}}r_{3\ce{T}}
+ r_{2\ce{D}}r_{5\ce{C}}r_{1\ce{N\!E}}
+ r_{2\ce{D}}r_{5\ce{C}}r_{1\ce{N\!I}}
+ r_{5\ce{C}}r_{2\ce{N}}r_{1\ce{N\!E}}\\
&+ r_{2\ce{D}}r_{4\ce{N}}r_{1\ce{N\!I}}
- r_{4\ce{D}}r_{2\ce{N}}r_{1\ce{N\!I}}
+ r_{4\ce{D}}r_{3\ce{T}}r_{1\ce{N\!E}}
+ r_{5\ce{C}}r_{3\ce{T}}r_{1\ce{N\!E}}
+ r_{4\ce{D}}r_{3\ce{T}}r_{1\ce{N\!I}}
+ r_{5\ce{C}}r_{3\ce{T}}r_{1\ce{N\!I}}
\\&
+ r_{4\ce{N}}r_{3\ce{T}}r_{1\ce{N\!E}}
+ r_{4\ce{N}}r_{3\ce{T}}r_{1\ce{N\!I}}
\bigr).
\end{split}
\end{equation*}
Since the factor is a multilinear homogeneous polynomials in the symbols $r_{j\ce{X}}$ with monomials having coefficients of both positive and negative sign, there exists a choice of symbols such that $a_7=0$, and thus the system has the capacity for a zero-eigenvalue bifurcation and differentiation. To see this last statement, it suffices to apply the intermediate value theorem. Indeed, for e.g. $(r_{4\ce{D}},r_{2\ce{N}},r_{3\ce{T}})$ big enough, $a_7$ is positive, while for e.g. $(r_{2\ce{D}},r_{4\ce{N}},r_{3\ce{T}})$ big enough, $a_7$ is negative, thus there exists a value in between for which $a_7$ changes sign.
\endproof

\proof[Proof of Thm.~\ref{thm:13}]
\emph{Step 1.}
The stoichiometric matrix $S$ of  
\ref{eq:modeliii} reads:
\begin{equation}\label{eq:s13}
   S=\scalebox{0.70}{$\begin{pmatrix}
         -1 & 0 & 0 & 0 & 0& 1 & 0 & 0 & 0 & 0 & 0 & 0 \\
         -1 & 1 & 0 & 0 & 0& 0 & 0 & 0 & 0& 0 & 0 & 0 \\
         1 & -1 & -1 & 1 & 0&0 & 0 & 0 & 0& 0 & 0 & 0 \\
         0 & 0 & 0 & 0 & -1&1 & 0 & 0& -1 & 1 & 0  & 0\\
         0 & 0 & 0 & 0 & 1&0 & 0 & -1 & 0& 0 & 0 & 0 \\
         0 & 0 & 0 & 0 & 0&-1 & 0 & 1 & 0& 0 & 0 & 0 \\
         0 & 0 &0 & 0 & 0& 0 & 0 & 0 & 1 & -1 & 0 & 0\\
         0 & 0 & 0 & 0 & 0&0 & -1 & 0 & 0 & 0 & 0 & 1\\
         0 & 0 &0 & 0 & 0& 0 & -1 & 1 & 0& 0 & 0 & 0 \\
         0 & 0 &0 & 0 & 0& 0 & 1 & -1 & -1 & 1 & 0 & 0 \\
         0 & 0 &-1 & 1 & 0& 0 & 0 & 0 & 0 & 0 & -1 & 1\\
        0 & -1 &0 & 0 & 0& 0 & 0 & 0 & 0& 0 & 1 & 0 \\
        0 & 1 & 0 & 0 & 0&0 & 0 & 0 & 0 & 0 & 0 & -1\\
         0 & 0 & 1 & -1 & 0&0 & 0 & 0 & 0& 0 & 0 & 0 
     \end{pmatrix}$}.
\end{equation}
from which we can compute that any positive vector $v$ such that
$Sv=0$
is of the form $v=(c,c,h,h,c,c,c,c,l,l,c,c)$ for $c,h,l\in \mathbb{R}_{>0}$. 
In particular, the flux cone 
$\mathcal{F}:=\{v\in \mathbb{R}^{12}_{>0} \; | \; 
%\text{\eqref{eq:fluxcone3} holds}
Sv=0 \}$ is three dimensional, and thus nonempty, and the network is consistent and admits a positive steady-state.

\emph{Step 2.} We compute a basis for the left kernel $W$ of $S$ as:
\begin{equation*}
\begin{split}
    W=\operatorname{span}\langle
    &\begin{pmatrix}
    0,0,0,1,1,1,1,0,0,0,0,0,0,0
    \end{pmatrix},
    \begin{pmatrix}
    0,1,1,0,0,0,0,0,0,0,0,0,0,1
    \end{pmatrix},\\
    &\begin{pmatrix}
    1,0,1,0,0,1,1,1,0,1,0,0,1,1
    \end{pmatrix},
    \begin{pmatrix}
    0,0,0,0,0,0,0,0,0,0,1,1,1,1
    \end{pmatrix},\\
    &\begin{pmatrix}
    0,0,0,0,0,0,1,0,1,1,0,0,0,0
    \end{pmatrix}\rangle.
    \end{split}
\end{equation*}
Thus there are 5 linearly independent conservation laws ($n=5$), namely
\begin{equation*}
\begin{cases}
w_1 := \ce{D_{s1}}  +   \ce{D_{1}}     + \ce{T_1}  +   \ce{B_1},\\
w_2 :=  \ce{N\!I_1}   +  \ce{N_1}   +     \ce{B_2},\\
w_3 := \ce{N\!E_1}  +     \ce{N_1} +   \ce{T_1} + \ce{B_1}  + \ce{N\!E_2}    +\ce{N_2} + \ce{T_2} + \ce{B_2},\\
w_4 := \ce{D_{s2}}+\ce{D_{2}}+\ce{T_2}+\ce{B_2}\\
w_5 := \ce{B_1}+\ce{N\!I_2}   +  \ce{N_2} .
\end{cases}
\end{equation*}
All species appear in at least one of the conserved quantities, which guarantees that orbits of the system \eqref{eq:system13} cannot diverge and they are bounded. Now we show that the coefficient $a_{|M|-n}=a_9$ of $g(\lambda)$ is non-identically zero. Via Lemma \ref{lem:CSexpansion}, this is in turn equivalent to the existence of at least one $9\times9$ invertible CS-matrix. We provide an arbitrary example: consider the CS-triple 
\begin{equation*}
\begin{split}
    \pmb{\kappa}=(&\{\ce{N\!E_1},\ce{N_1},\ce{D_{s1}},\ce{T_1},\ce{B_1},\ce{N_2},\ce{D_{s2}},\ce{T_2},\ce{B_2}\},\{11,12,17,18,19,22,27,28,29\},\\&J(\ce{N\!E_1},\ce{N_1},\ce{D_{s1}},\ce{T_1},\ce{B_1},\ce{N_2},\ce{D_{s2}},\ce{T_2},\ce{B_2})=(11,12,18,19,27,22,28,29,17)).
    \end{split}
\end{equation*}
Its associated CS-matrix $S[\pmb{\kappa}]$ reads:
\begin{equation}
S[\pmb{\kappa}]=\scalebox{0.70}{$
    \begin{pmatrix}
       -1 & 0 &0 & 1 & 0 & 0 &0 & 0 &0\\
       1 & -1 & 0 &0 & 0 & 0 &0& 0 & 1\\
       0 & 0 &-1 & 1& 1 & 0 &0 &0 & 0\\
       0 & 0 & 0 &-1 & 0 &1&0 & 0 & 0\\
       0 & 0 & 0 & 0&-1 & 0&0 &0 & 0\\
       0 & 0 & 0 & 0&1 & -1&0 & 0 & 0\\
     0 & 0 & 0 & 0& 0 & 0 & -1 & 1 & 1\\
        0 & 1 &0 & 0 &0 & 0 & 0 & -1 & 0\\
    0 & 0 & 0 & 0&0 & 0 & 0 & 0 & -1\\
 \end{pmatrix}$},
\end{equation}
which is invertible since all eigenvalues are $-1$. The network is thus nondegenerate.

\emph{Step 3.} As in the proof of Thm.~\ref{thm:12}, we compute the symbolic Jacobian $G=SR$ as the product between the stoichiometric matrix \eqref{eq:s13} and:
\begin{equation*}
R=\scalebox{0.70}{$
\begin{pmatrix}
r_{11\ce{N\!E_1}} & r_{11\ce{N\!I_1}} & 0 & 0 & 0 & 0 & 0 & 0 & 0 & 0 & 0 & 0 & 0 & 0\\
0 & 0 & r_{12\ce{N_1}} & 0 & 0 & 0 & 0 & 0 & 0 & 0 & 0 & r_{12\ce{D_2}} & 0 & 0\\
0 & 0 & r_{16\ce{N_1}} & 0 & 0 & 0 & 0 & 0 & 0 & 0 & r_{16\ce{D_{s2}}} & 0 & 0 & 0\\
0 & 0 & 0 & 0 & 0 & 0 & 0 & 0 & 0 & 0 & 0 & 0 & 0 & r_{17\ce{B_2}}\\
0 & 0 & 0 & r_{18\ce{D_{s1}}} & 0 & 0 & 0 & 0 & 0 & 0 & 0 & 0 & 0 & 0\\
0 & 0 & 0 & 0 & 0 & r_{19\ce{T_1}} & 0 & 0 & 0 & 0 & 0 & 0 & 0 & 0\\
0 & 0 & 0 & 0 & 0 & 0 & 0 & r_{21\ce{N\!E_2}} & r_{21\ce{N\!I_2}} & 0 & 0 & 0 & 0 & 0\\
0 & 0 & 0 & 0 & r_{22\ce{D_1}} & 0 & 0 & 0 & 0 & r_{22\ce{N_2}} & 0 & 0 & 0 & 0\\
0 & 0 & 0 & r_{26\ce{D_{s1}}} & 0 & 0 & 0 & 0 & 0 & r_{26\ce{N_2}} & 0 & 0 & 0 & 0\\
0 & 0 & 0 & 0 & 0 & 0 & r_{27\ce{B_1}} & 0 & 0 & 0 & 0 & 0 & 0 & 0\\
0 & 0 & 0 & 0 & 0 & 0 & 0 & 0 & 0 & 0 & r_{28\ce{D_{s2}}} & 0 & 0 & 0\\
0 & 0 & 0 & 0 & 0 & 0 & 0 & 0 & 0 & 0 & 0 & 0 & r_{29\ce{T_2}} & 0
\end{pmatrix}.$}
\end{equation*}
{Again, we omit the} explicit form of $G$ and 
compute the characteristic polynomial of $G$ and, using the 
expansion in Lemma \ref{lem:CSexpansion}, we look for unstable-positive 
feedback. With the \textsc{BiRNe} Python module \cite{Golnik25}, we find that there are exactly two CS-triples ($\pmb{\kappa}_1$, $\pmb{\kappa}_2$) whose CS-matrices are unstable-positive feedbacks:
\begin{equation*}
\begin{cases}
\pmb{\kappa}_1=(\kappa_1=\{\ce{D_{s1}},\ce{D_1},\ce{T_1},\ce{N_2}\},\;E_{\kappa_1}=\{18,19,22,26\},\;J(\kappa_1)=\{18,22,19,26\})\\
\pmb{\kappa}_2=(\kappa_2=\{\ce{N_1},\ce{D_{s2}},\ce{D_2},\ce{T_2}\},\;E_{\kappa_2}=\{12,16,28,29\}, \;J(\kappa_2)=\{16,28,12,29\})
\end{cases}
\end{equation*}
with associated {CS-matrices}
\begin{equation*}
S[\pmb{\kappa}_1]=\scalebox{0.70}{$
\begin{blockarray}{ccccc}
& 18 & 22 & 19 & 26 \\
\begin{block}{c(cccc)}
\ce{D_{s1}} & -1 & 0 & 1 & -1 \\
\ce{D_{1}} & 1 & -1 & 0 & 0 \\
\ce{T_1} & 0 & 1 & -1 & 0\\
\ce{N_2} & 0 & -1 & 0 & -1\\
\end{block}
\end{blockarray}$}\quad\quad\quad 
S[\pmb{\kappa}_2]=\scalebox{0.70}{$
\begin{blockarray}{ccccc}
& 16 & 28 & 12 & 29 \\
\begin{block}{c(cccc)}
\ce{N_1}   & -1 & 0 & -1 & 0 \\
\ce{D_{s2}}& -1 & -1 & 0 & 1  \\
\ce{D_{2}} & 0 & 1 & -1 & 0  \\
\ce{T_2}  & 0 & 0 & 1 & -1 \\
\end{block}
\end{blockarray}$}
\end{equation*}
The pair of CS $(\pmb{\kappa}_1,\pmb{\kappa}_2)$ is related by the symmetry. By selecting the corresponding reactions and species in the network, which appear in such CS-triples, we get the instability motif \eqref{eq:motif3}. As both $S[\pmb{\kappa}_1]$ and $S[\pmb{\kappa}_2]$ possess negative off-diagonal entries, they are non-Metzler, and thus the network is non-autocatalytic.

\emph{Step 4.} 
We assume the kinetic symmetry condition \eqref{eq:constraints2} at a 
\emph{homogeneous} steady state, i.e. where 
{$[\bar{\ce{X}}_1]=[\bar{\ce{X}}_2]$} for $\ce{X}\in\{\ce{N\!E},\ce{N\!I},\ce{N},\ce{D_s},\ce{D},\ce{T},\ce{B}\}.$ 
Again, we denote the reaction rates as $r_j := r_{1j} = r_{2j}$ for $j \in \{1,2,6,7,8,9\}$, and the species concentrations as {$[\bar{\ce{X}}] := [\bar{\ce{X}}_1] = [\bar{\ce{X}}_2]$.}
This yields the following symbolic reactivity matrix:
\begin{equation*}
R_\sigma=\scalebox{0.70}{$\begin{pmatrix}
r_{1\ce{N\!E}} & r_{1\ce{N\!I}} & 0 & 0 & 0 & 0 & 0 & 0 & 0 & 0 & 0 & 0 & 0 & 0\\
0 & 0 & r_{2\ce{N}} & 0 & 0 & 0 & 0 & 0 & 0 & 0 & 0 & r_{2\ce{D}} & 0 & 0\\
0 & 0 & r_{6\ce{N}} & 0 & 0 & 0 & 0 & 0 & 0 & 0 & r_{6\ce{D_{s}}} & 0 & 0 & 0\\
0 & 0 & 0 & 0 & 0 & 0 & 0 & 0 & 0 & 0 & 0 & 0 & 0 & r_{7\ce{B}}\\
0 & 0 & 0 & r_{8\ce{D_{s}}} & 0 & 0 & 0 & 0 & 0 & 0 & 0 & 0 & 0 & 0\\
0 & 0 & 0 & 0 & 0 & r_{9\ce{T}} & 0 & 0 & 0 & 0 & 0 & 0 & 0 & 0\\
0 & 0 & 0 & 0 & 0 & 0 & 0 & r_{1\ce{N\!E}} & r_{1\ce{N\!I}} & 0 & 0 & 0 & 0 & 0\\
0 & 0 & 0 & 0 & r_{2\ce{D}} & 0 & 0 & 0 & 0 & r_{2\ce{N}} & 0 & 0 & 0 & 0\\
0 & 0 & 0 & r_{6\ce{D_{s}}} & 0 & 0 & 0 & 0 & 0 & r_{6\ce{N}} & 0 & 0 & 0 & 0\\
0 & 0 & 0 & 0 & 0 & 0 & r_{7\ce{B}} & 0 & 0 & 0 & 0 & 0 & 0 & 0\\
0 & 0 & 0 & 0 & 0 & 0 & 0 & 0 & 0 & 0 & r_{8\ce{D_{s}}} & 0 & 0 & 0\\
0 & 0 & 0 & 0 & 0 & 0 & 0 & 0 & 0 & 0 & 0 & 0 & r_{9\ce{T}} & 0
\end{pmatrix}$}.
\end{equation*}
Due to the $\mathbb{Z}_2$-symmetry, each symbol appears now exactly two times. From $R_\sigma$, the symbolic Jacobian $G_\sigma=SR_\sigma$, at kinetic symmetry, is computed as
{\tiny
\setlength{\arraycolsep}{2pt}
\begin{equation*}
\begin{pmatrix}
 -r_{1\ce{N\!E}} & -r_{1\ce{N\!I}} & 0 & 0 & 0 & r_{9\ce{T}} & 0 & 0 & 0 & 0 & 0 & 0 & 0 & 0\\
-r_{1\ce{N\!E}} & -r_{1\ce{N\!I}} & r_{2\ce{N}} & 0 & 0 & 0 & 0 & 0 & 0 & 0 & 0 & r_{2\ce{D}} & 0 & 0\\
r_{1\ce{N\!E}} & r_{1\ce{N\!I}} & -r_{2\ce{N}} - r_{6\ce{N}} & 0 & 0 & 0 & 0 & 0 & 0 & 0 & -r_{6\ce{D_{s}}} & -r_{2\ce{D}} & 0 & r_{7\ce{B}}\\
0 & 0 & 0 & -r_{6\ce{D_{s}}} - r_{8\ce{D_{s}}} & 0 & r_{9\ce{T}} & r_{7\ce{B}} & 0 & 0 & -r_{6\ce{N}} & 0 & 0 & 0 & 0\\
0 & 0 & 0 & r_{8\ce{D_{s}}} & -r_{2\ce{D}} & 0 & 0 & 0 & 0 & -r_{2\ce{N}} & 0 & 0 & 0 & 0\\
0 & 0 & 0 & 0 & r_{2\ce{D}} & -r_{9\ce{T}} & 0 & 0 & 0 & r_{2\ce{N}} & 0 & 0 & 0 & 0\\
0 & 0 & 0 & r_{6\ce{D_{s}}} & 0 & 0 & -r_{7\ce{B}} & 0 & 0 & r_{6\ce{N}} & 0 & 0 & 0 & 0\\
0 & 0 & 0 & 0 & 0 & 0 & 0 & -r_{1\ce{N\!E}} & -r_{1\ce{N\!I}} & 0 & 0 & 0 & r_{9\ce{T}} & 0\\
0 & 0 & 0 & 0 & r_{2\ce{D}} & 0 & 0 & -r_{1\ce{N\!E}} & -r_{1\ce{N\!I}} & r_{2\ce{N}} & 0 & 0 & 0 & 0\\
0 & 0 & 0 & -r_{6\ce{D_{s}}} & -r_{2\ce{D}} & 0 & r_{7\ce{B}} & r_{1\ce{N\!E}} & r_{1\ce{N\!I}} & -r_{2\ce{N}} - r_{6\ce{N}} & 0 & 0 & 0 & 0\\
0 & 0 & -r_{6\ce{N}} & 0 & 0 & 0 & 0 & 0 & 0 & 0 & -r_{6\ce{D_{s}}} - r_{8\ce{D_{s}}} & 0 & r_{9\ce{T}} & r_{7\ce{B}}\\
0 & 0 & -r_{2\ce{N}} & 0 & 0 & 0 & 0 & 0 & 0 & 0 & r_{8\ce{D_{s}}} & -r_{2\ce{D}} & 0 & 0\\
0 & 0 & r_{2\ce{N}} & 0 & 0 & 0 & 0 & 0 & 0 & 0 & 0 & r_{2\ce{D}} & -r_{9\ce{T}} & 0\\
0 & 0 & r_{6\ce{N}} & 0 & 0 & 0 & 0 & 0 & 0 & 0 & r_{6\ce{D_{s}}} & 0 & 0 & -r_{7\ce{B}}
\end{pmatrix}
\end{equation*}}

We compute the characteristic polynomial $g_\sigma(\lambda)$ of $G_\sigma$. Since the network is nondegenerate, the Jacobian of the system reduced to a stoichiometric compatibility class is interpreted as the coefficient $a_{|M|-n}=a_{|14|-5}=a_{9}$, which reads:
\begin{equation*}
\begin{split}
{   a_9=}&
2 r_{9T} r_{1NE} r_{8Ds}
\left( r_{7B} r_{2N} - r_{2D} r_{6N} \right)\cdot\dots  \\
& \dots\cdot \Big(
r_{2D} r_{7B} r_{9T} r_{1NE}
+ r_{2D} r_{7B} r_{9T} r_{1NI}
+ r_{2D} r_{6N} r_{9T} r_{1NE}
+ r_{2D} r_{6N} r_{9T} r_{1NI}
\\
&- r_{2D} r_{6N} r_{9T} r_{8Ds}
+ r_{2D} r_{7B} r_{1NE} r_{8Ds}
+ r_{2D} r_{7B} r_{1NI} r_{8Ds}
+ r_{7B} r_{2N} r_{1NE} r_{8Ds}\\&
+ r_{2D} r_{6N} r_{1NI} r_{8Ds}
+ r_{2D} r_{9T} r_{1NE} r_{6Ds}
+ r_{2D} r_{9T} r_{1NI} r_{6Ds}
+ r_{7B} r_{9T} r_{1NE} r_{8Ds}\\&
+ r_{7B} r_{9T} r_{1NI} r_{8Ds}
+ r_{2N} r_{9T} r_{1NE} r_{6Ds}
+ r_{2N} r_{9T} r_{1NI} r_{6Ds}\\&
+ r_{7B} r_{2N} r_{9T} r_{8Ds} + r_{6N} r_{9T} r_{1NE} r_{8Ds}
+ r_{6N} r_{9T} r_{1NI} r_{8Ds}
\Big).
\end{split}
\end{equation*}
From the first factor in brackets, we see that $a_9=0$ and 
thus a zero eigenvalue bifurcation occurs and with this the capacity for differentiation if $r_{7B} r_{2N} = r_{2D} r_{6N}$.
\endproof

\subsection*{Proofs of Sec.~\ref{sec:minimal}}

\proof[Proof of Prop.~\ref{prop:3}]
The network is consistent by choosing $r_1 = r_2$. %  At kinetic symmetry, identical concentrations and thus for homegeneous steady-state identical concentrations of $[\ce{N\!I_1}]=[\ce{N\!I_2}]$ and $[\ce{N\!I_1}]=[\ce{N\!I_2}]$ provide a positive (homogeneous) steady state. 
Moreover, the system describes a one-to-one exchange of $\Lambda_1$ with $\Lambda_2$ mediated and catalyzed by $[\ce{N\!I_1}]$ and $[\ce{N\!I_2}]$. In particular, the quantity $[\Lambda_1]+[\Lambda_2]$ is conserved. Since only two equations in the system are non-identically zero, {the system can be reduced to one ODE}. 

Consequently, upon inspection of the Jacobian matrix, the capacity for zero-eigenvalue bifurcation is equivalent to the fact that the trace can change sign. Ordering the species as $(\ce{N\!I_1}, \Lambda_1, \ce{N\!I_2}, \Lambda_2)$ we get:
\begin{equation*}
    G=\scalebox{0.70}{$
     \begin{pmatrix}
       0 & 0 & 0 & 0 \\
         -\partial_{\ce{N\!I_1}} r_1 &\partial_{\Lambda_1}r_2 - \partial_{\Lambda_1}r_1 & \partial_{\ce{N\!I_2}} r_2  &\partial_{\Lambda_2}r_2 -\partial_{\Lambda_2}r_1\\
         0 & 0 & 0 & 0 \\
            \partial_{\ce{N\!I_1}} r_1 &-\partial_{\Lambda_1}r_2 +\partial_{\Lambda_1}r_1 & -\partial_{\ce{N\!I_2}} r_2  &-\partial_{\Lambda_2}r_2 +\partial_{\Lambda_2}r_1
     \end{pmatrix}$},
\end{equation*}
with trace $
\operatorname{Tr}G=(\partial_{\Lambda_1}r_2 - \partial_{\Lambda_1}r_1)+(-\partial_{\Lambda_2}r_2 +\partial_{\Lambda_2}r_1).$
At kinetic symmetry, we get that:
\begin{equation*}
\partial_{\Lambda_1}r_2=  \partial_{\Lambda_2}r_1 :=g_1, \quad \quad \text{and}\quad\quad
 \partial_{\Lambda_1}r_1=  \partial_{\Lambda_2}r_2 :=g_2,
\end{equation*}
and thus $
   \operatorname{Tr}G=2(g_1 - g_2).$
A zero-eigenvalue bifurcation is thus possible for $g_1=g_2$, and with this the capacity for differentiation.
\endproof

\proof[Proof of Prop.~\ref{prop:min245}]
The arguments are similar in all three cases. 
All three models are consistent: simply choosing {$r_1(x)=r_2(x)$, i.e. the positive kernel vector $v=(c,c)$, $c>0$ of the stoichiometric matrix,} yields the existence of (at least) one positive steady state.
Moreover, all three models are one ODE, as they contain two nonzero ODEs with one conservation law, respectively $(\ce{N\!E}_1+\ce{N\!E}_2)$ in \eqref{eq:odemodelmin2} and \eqref{eq:odemodelmin5}, $(\Lambda_1+\Lambda_2)$ in \eqref{eq:odemodelmin4}. This follows from the fact that the reaction networks all concern a one-to-one exchange of one entity: ($\ce{N\!E}_1 $ with $\ce{N\!E}_2$) in \eqref{eq:odemodelmin2} and \eqref{eq:odemodelmin5}, and ($\Lambda_1$ with $\Lambda_2$) in \eqref{eq:odemodelmin4}. Such exchange is catalyzed by the other species involved in the network.

For monostationarity and the consequent absence of capacity for differentiation, we check the trace of the respective symbolic Jacobians. A negative trace guarantees that the single nonzero eigenvalue is always negative and thus the steady-state is unique and locally stable for all choices of parameters, thereby excluding autocatalysis as well.

The three Jacobians are listed in sequence below. For \eqref{eq:odemodelmin2}, ordering the involved species as $(\ce{N\!I_1},\ce{N\!E_1}, \ce{D_1}, \ce{N\!I_2}, \ce{N\!E_2},\ce{D_2})$, the Jacobian matrix reads
\begin{equation*}
    G=\scalebox{0.70}{$
     \begin{pmatrix}
         0 & 0 & 0 & 0 & 0 & 0\\
         -\partial_{\ce{N\!I_1}} r_1 & -\partial_{\ce{N\!E_1}} r_1 &\partial_{\ce{D_1}}r_2 & \partial_{\ce{N\!I_2}} r_2 & \partial_{\ce{N\!E_2}}r_2 & -\partial_{\ce{D_2}}r_1\\
         0 & 0 & 0 & 0 & 0 & 0\\
         0 & 0 & 0 & 0 & 0 & 0\\
          \partial_{\ce{N\!I_1}} r_1 & \partial_{\ce{N\!E_1}} r_1 &-\partial_{\ce{D_1}}r_2 & -\partial_{\ce{N\!I_2}} r_2 & -\partial_{\ce{N\!E_2}}r_2 & \partial_{\ce{D_2}}r_1\\
         0 & 0 & 0 & 0 & 0 & 0
     \end{pmatrix}$},
\end{equation*}
 with negative trace
 $\operatorname{Tr}G=-\partial_{\ce{N\!E_1}} r_1 - \partial_{\ce{N\!E_2}}r_2<0$.  For \eqref{eq:odemodelmin4}, ordering the involved species as $(\Lambda_1, \ce{D_1}, \Lambda_2, \ce{D_2},)$, the Jacobian matrix reads
 \begin{equation*}
    G=\scalebox{0.70}{$
     \begin{pmatrix}
         -\partial_{\Lambda_1} r_1 &\partial_{\ce{D_1}}r_2 & \partial_{\Lambda_2} r_2  & -\partial_{\ce{D_2}}r_1\\
         0 & 0 & 0 & 0 \\
           \partial_{\Lambda_1} r_1 &-\partial_{\ce{D_1}}r_2 & -\partial_{\Lambda_2} r_2  & \partial_{\ce{D_2}}r_1\\
         0 & 0 & 0 & 0 \\
     \end{pmatrix}$},
\end{equation*}
 with negative trace
 $\operatorname{Tr}G=-\partial_{\Lambda_1} r_1 - \partial_{\Lambda_2}r_2<0$. For \eqref{eq:odemodelmin5}, ordering the involved species as $(\Lambda_1, \ce{N\!E_1}, \Lambda_2, \ce{N\!E_2})$, the Jacobian matrix reads
 \begin{equation*}
    G=\scalebox{0.70}{$    \begin{pmatrix}
0 & 0 & 0 & 0 \\
\partial_{\Lambda_1} r_{2}-\partial_{\Lambda_1}r_{1} & -\partial_{\ce{N\!E_1}}r_{1} & \partial_{\Lambda_2} r_{2}-\partial_{\Lambda_2}r_{1} &\partial_{\ce{N\!E_2}}r_{2}  \\
0 & 0 & 0 & 0 \\
-\partial_{\Lambda_1} r_{2}+\partial_{\Lambda_1}r_{1} & \partial_{\ce{N\!E_1}}r_{1}  & -\partial_{\Lambda_2} r_{2}+\partial_{\Lambda_2}r_{1} & -\partial_{\ce{N\!E_2}}r_{2} 
\end{pmatrix}$},
\end{equation*}
 with negative trace
 $\operatorname{Tr}G=-\partial_{\ce{N\!E_1}}r_{1} -\partial_{\ce{N\!E_2}}r_{2}<0$. 

In conclusion, \ref{eq:MII}, \ref{eq:miv} and \ref{eq:Mv} do not have the capacity for differentiation.
\endproof

\proof[Proof of Prop.~\ref{prop:nautI}]
{Under the kinetic  symmetry constraints, the system admits a steady-state whenever $
P=\eta \;r([\bar{\Lambda}_1],[\bar{\Lambda}_2])+ r([\bar{\Lambda}_2],[\bar{\Lambda}_1])$. Consider a homogeneous steady state $[\bar{\Lambda}_1]=[\bar{\Lambda}_2]$ and} let $\partial_1 r(\cdot,\cdot)$ and $\partial_2 r(\cdot,\cdot)$ denote differentiation of $r$ with respect to its first and second arguments, respectively. The symbolic Jacobian reads
\begin{equation*}
G=
    \begin{pmatrix}
        -\eta\;\partial_1 r - \partial_2 r & -\eta\;\partial_2 r - \partial_1 r\\[6pt]
         -\eta \;\partial_2 r - \partial_1 r &  -\eta \; \partial_1 r - \partial_2 r
    \end{pmatrix},
\end{equation*}
with 
\begin{equation}\label{eq:detfactab}
\begin{split}
\operatorname{det} G &=\operatorname{det}\begin{pmatrix}
    -\eta & -1\\
    -1 & -\eta
\end{pmatrix} (\partial_1 r)^2+\operatorname{det}\begin{pmatrix}
-1 & -\eta\\
-\eta & -1
\end{pmatrix} (\partial_2 r)^2 \\
&=(\eta^2-1){\left[(\partial_1 r)^2-(\partial_2 r)^2\right]}.
\end{split}
\end{equation}
Thus, a zero-eigenvalue bifurcation occurs for $\eta \neq 1$ precisely when $
    \partial_2 r = \partial_1 r,$
with the instability region identified when $\partial_1 r > \partial_2 r$ for $\eta<1$ and $\partial_1 r < \partial_2 r$ for $\eta>1$. The instability motif is given by the matrices in the expansion \eqref{eq:detfactab} of the determinant of the symbolic Jacobian. Depending on $\eta>1$ or $\eta<1$ the unstable CS matrix is $\begin{pmatrix}
-1 & -\eta\\
-\eta & -1
\end{pmatrix}$ or $\begin{pmatrix}
    -\eta & -1\\
    -1 & -\eta
\end{pmatrix}$. Both cases identify the same instability motif \eqref{nonautiinstmotif}.
\endproof

\proof[Proof of Prop.~\ref{prop:nautII}]
The stoichiometric matrix of \ref{eq:modelmin2} reads:
\begin{equation*}
S=
    \begin{pmatrix}
       0 & 0 & -1 & 1 \\
       0 & 0 & 1 & -1 \\
     -1 & 1 & 0 & 0\\
     1 & -1 & 0 & 0
    \end{pmatrix},
\end{equation*}
and has a natural basis of the (right and left) kernel space as $\operatorname{ker}S=\operatorname{span}\{v_1,v_2\}$ with $v_1^T=(1,1,0,0)$ and $v_2^T=(0,0,1,1)$. Consistency of the network follows from considering any positive flux vector, e.g.  $v_1^T+v_2^T=(1,1,1,1)$, while the two-dimensionality of the left kernel of $S$ implies the existence of two independent conservation laws, i.e. $([\Lambda_1]+[\ce{I}_1])$ and $([\Lambda_2]+[\ce{I}_2])$. This implies that the system is (at most) two dimensional, so we just check the symbolic Jacobian at kinetic symmetry. Let $\partial_1 r_o$ and $\partial_2 r_o$ denote the derivative of $r$ w.r.t. its first and second arguments. At kinetic symmetry, we get:
\begin{equation*}
\begin{cases}
    \partial_{\Lambda_1} r_1=\partial_{\Lambda_2} r_3=\partial_1r_o\\
    \partial_{\Lambda_1} r_3=\partial_{\Lambda_2} r_1=\partial_2r_o\\
    \partial_{\ce{I}_2} r_2=\partial_{\ce{I}_1} r_4=\partial_I r_e
\end{cases}\quad
\end{equation*}
and thereby the symbolic reactivity matrix $R$ reads:
\begin{equation*}
R=
    \begin{pmatrix}
        \partial_1r_o & 0& \partial_2r_o & 0\\
        0 & 0  & 0 & \partial_I r_e\\
        \partial_2r_o  & 0 & \partial_1r_o & 0\\
        0 & \partial_I r_e & 0 & 0
    \end{pmatrix},
\end{equation*}
yielding a symbolic Jacobian matrix at kinetic symmetry:
\begin{equation*}
    G=\begin{pmatrix}
-\partial_2 r_o & \partial_I r_e & -\partial_1 r_o & 0 \\
\partial_2 r_o & -\partial_I r_e & \partial_1 r_o & 0 \\
-\partial_1 r_o & 0 & -\partial_2 r_o & \partial_I r_e \\
\partial_1 r_o & 0 & \partial_2 r_o & -\partial_I r_e
\end{pmatrix},
\end{equation*}
with characteristic polynomial:
\begin{equation*}
    g(\lambda)=\lambda^{4}
+ \bigl(2\,\partial_I r_e + 2\,\partial_2 r_o\bigr)\lambda^{3}
+ \bigg((\partial_I r_e)^2
+ 2\,\partial_I r_e\,\partial_{{2}} r_o
+ (\partial_2 r_o)^2
- (\partial_1 r_o)^2\bigg)\lambda^{2}.
\end{equation*}
As the system is two-dimensional, a zero-eigenvalue bifurcation occurs if and only if $
    \partial_I r_e+\partial_2 r_o=\partial_1 r_o,$
with instability area being identified by $\partial_I r_e+\partial_2 r_o<\partial_1 r_o$. 
Via Lemma \ref{lem:CSexpansion}, the {unstable} CS matrix associated to the summand $(\partial_1 r_o)^2$ and the instability motif \eqref{eq:instability motifnonautII} is
\begin{equation*}
S[\pmb{\kappa}]=
\begin{blockarray}{ccc}
& 1 & 3\\
\begin{block}{c(cc)}
\Lambda_1 & 0 & -1\\
\Lambda_2 & -1 & 0\\
\end{block}
\end{blockarray}\;.
\end{equation*}
\endproof

\textbf{Acknowledgments:}
A.S. was supported by the DFG (German Research Foundation) under
Germany’s Excellence Strategy EXC 2044/390685587 and EXC 2044/2-390685587,
Mathematics M\"unster: Dynamics - Geometry - Structure and
partially hosted by the Max-Planck-Insitute for Mathematics in the
Sciences in Leipzig. N.V.'s work was funded by the DFG, project n. 512355535, and by the MATOMIC consortium of the Novo Nordisk Foundation, grant NNF21OC0066551. 
{\small
\bibliography{references.bib}
\bibliographystyle{alpha}}

\vspace{2cm}

Angela Stevens\\
University of M\"unster\\
Institute for Analysis and Numerics\\
Einsteinstr. 62\\
D-48149 M\"unster\\
Germany\\
angela.stevens@uni-muenster.de\\

Nicola Vassena\\
Leipzig University\\
Faculty of Mathematics and Computer Science\\
Augustuspl. 10\\
04109 Leipzig\\
Germany\\
nicola.vassena@uni-leipzig.de

\end{document}